\documentclass[reqno, 12pt]{amsart}

\usepackage{mathtools}
\usepackage[T1]{fontenc}
\usepackage{amsmath, amsthm, amssymb}
\usepackage[ansinew]{inputenc}
\usepackage{xcolor}
\usepackage{bbm}
\usepackage{mathrsfs}
\usepackage{cite}
\usepackage{longtable}
\usepackage{xparse}
\usepackage{fullpage}
\usepackage{bbm}
\usepackage{pdflscape}
\usepackage{stmaryrd}
\usepackage{multirow}
\usepackage{lscape}
\usepackage{graphicx} 
\usepackage[colorlinks=true,allcolors=red]{hyperref}
\usepackage{accents}

\newtheorem{theorem}{Theorem}[section]
\newtheorem{lemma}[theorem]{Lemma}
\newtheorem{prop}[theorem]{Proposition}

\theoremstyle{definition}

\theoremstyle{remark}
\newtheorem{remark}[theorem]{Remark}

\allowdisplaybreaks[4]

\numberwithin{equation}{section}

\usepackage{scalerel,stackengine}
\stackMath
\newcommand\reallywidehat[1]{
\savestack{\tmpbox}{\stretchto{
  \scaleto{
    \scalerel*[\widthof{\ensuremath{#1}}]{\kern.1pt\mathchar"0362\kern.1pt}
    {\rule{0ex}{\textheight}}
  }{\textheight}
}{2.ex}}
\stackon[-7pt]{#1}{\tmpbox}
}
\makeatletter
\newcommand*{\rom}[1]{\expandafter\@slowromancap\romannumeral #1@}
\makeatother

\begin{document}

\title{On the Fourier analysis of the Einstein-Klein-Gordon system: Growth and Decay of the Fourier constants} 

\author{Athanasios Chatzikaleas}

\address{Westf\"alische Wilhelms-Universit\"at M\"unster, Mathematisches Institut, Einsteinstrasse 62, 48149 M\"unster Germany}
\email{achatzik@uni-muenster.de}

\dedicatory{}

\begin{abstract}
We consider the $(1+3)$-dimensional Einstein equations with negative cosmological constant coupled to a spherically-symmetric, massless scalar field and study perturbations around the Anti-de Sitter spacetime. We derive the resonant systems, pick out vanishing secular terms and discuss issues related to small divisors. Most importantly, we rigorously establish (sharp, in most of the cases) asymptotic behaviour for all the interaction coefficients. The latter is based on uniform estimates for the eigenfunctions associated to the linearized operator as well as on some oscillatory integrals.  
\end{abstract}

\maketitle
\tableofcontents 
\addtocontents{toc}{\protect\setcounter{tocdepth}{1}} 
\noindent
\section{Introduction}
\noindent
In this work, we initiate a rigorous study of the dynamics in Fourier space 
of spherically symmetric solutions to the coupled Einstein-Klein-Gordon system, 
\begin{align}\label{arxiko}
\begin{dcases}
	\text{Ric}(g) -\frac{1}{2}g  \text{R}(g) + \Lambda g = 8 \pi \left ( d \phi \otimes d \phi - \frac{1}{2} g |d \phi|_{g}^2 \right ), \quad 
	 \Lambda<0, \\
	\Box_{g} \phi = 0,
\end{dcases}
\end{align} 
around the Anti-de Sitter solution. Here, $\text{Ric}(g)$ and $R(g)$ denote the Ricci and scalar curvatures respectively, both with respect to the spacetime metric $g$, $\Box_{g} $ stands for the wave operator and $\Lambda$ is the cosmological constant. 
\subsection{(In-)stability of Anti-de Sitter}
The Ant-de Sitter (AdS) spacetime is the unique (up to quotients by isometry subgroups) maximally symmetric solution to the Einstein equations with negative cosmological constant,
\begin{align*}
	\text{Ric}(g)=    -\Lambda g, \quad \Lambda<0.
\end{align*}
In local compactified coordinates $(t,x,\omega) \in  \mathcal{M}:=  \mathbb{R} \times \left[0,\pi /2 \right) \times \mathbb{S}^{d-1}$, it reads
\begin{align*}
	g_{\text{AdS}}(t,x,\omega) = \frac{1}{\cos ^2 (x)} \left( - dt^2 + dx^2 + \sin ^2 (x) d \omega ^2 \right) 
\end{align*}
implying that $(\mathcal{M},g_{\text{AdS}})$ is conformal to half of the Einstein static universe. One can see that, altough the spatial distance from any point $(t,x)$ with $0 \leq x <\pi/2$ to $\mathcal{I}:=\{x=\pi/2\}$ is infinite, null geodesics reach the conformal spatial infinity $\mathcal{I} $ in finite time with respect to time as measured by an inertial observer at the center. The particular characteristic of the AdS solution as well as of all asymptotically AdS spacetimes (that are spacetimes that approach the AdS solution at infinity in the spatial directions fast enough and share the same conformal boundary) is that the conformal spatial infinity is a time-like cylinder $\mathbb{R} \times \mathbb{S}^{d-1}$. Consequently, the AdS metric is not globally hyperbolic and hence in order to study the evolution of the field $\phi$ on such backgrounds one has to prescribe boundary conditions also on $\mathcal{I}$ in addition to the initial data on the $\{t=0\}$ slice. \\ \\
It is well known that the Minkowski space is a ground state among asymptotically flat spacetimes \cite{MR0612249,MR0626707}. The AdS spacetime also enjoys a similar variational characterization due to the positive energy theorem which states that for solutions to the general Einstein field equations which are globally regular and satisfy a reasonable energy condition, the AdS space is a ground state among asymptotically AdS spacetimes \cite{MR0701918}. As far as the initial-boundary value problem is concerned, Smulevici-Holzegel \cite{MR2913628} (for Dirichlet boundary conditions) and Warnick-Holzegel \cite{MR3369103} (for more general boundary conditions) proved its local well-posedness in the class of spherically symmetric solutions. We also note that, if one removes the symmetry assumptions, Friedrich's proof of local well-posedness for the vacuum equations yields local well-posedness for the Einstein-Klein-Gordon system provided that the latter has the conformal value of the mass. Once the local well-posedness is established, an important question for the AdS solution (as for any ground state) is whether it is stable or not, meaning whether small perturbations of the solution on the $\{t = 0\}$ slice remain small for all future times or not. For the Minkowski spacetime such a question has been answered by Christodoulou-Klainerman \cite{MR1316662} and for the de-Sitter spacetime by Friedrich \cite{MR0868737} who proved its stability. 
\\ \\
The main mechanism responsible for the stability of the Minkowski spacetime is the dissipation of energy by dispersion. In the case of AdS solution, such a mechanism is no longer present. For ``reflective'' boundary conditions on the conformal infinity, waves which start at any point inside the region $\{0\leq x < \pi /2 \}$ and propagate outwards are reflected on $\mathcal{I}$ and return back to into the region from where they started \cite{MR3205859}. Such boundary conditions are confining enough forcing the AdS solution to act as a closed universe (in terms of its fields inside). 
\\ \\ 
Although the conjecture on the instability of the AdS spacetime was first announced by Dafermos \cite{DafermosTalk} and Dafermos-Holzegel \cite{DafermosHolzegel} in 2006, the first work in this direction was the seminal paper of Bizo\'{n}-Rostworowski \cite{PhysRevLett107031102}. Specifically, Bizo\'{n}-Rostworowski \cite{PhysRevLett107031102} considered the spherically symmetric Einstein massless scalar field equations with negative cosmological constant in $(1+3)$-dimensions and established strong numerical (as well as analytical) results indicating that the AdS solution to the Einstein equations (although linearly stable) is in fact nonlinearly unstable against the formation of a black hole under arbitrarily small and generic perturbations. In their work \cite{PhysRevLett107031102}, they used specific Gaussian-type initial data and concluded that such initial data evolve to a wave which, as it propagates in time, collapses quickly and an apparent horizon appears. On top of their numerical findings, Bizo\'{n}-Rostworowski \cite{PhysRevLett107031102} also proposed the resonant mode mixing as the mechanism responsible for the AdS instability. The authors expanded the dynamical variables in terms of the eigenfunctions $\{e_n: n \geq 0\}$ to the associated linearized operator. Subsequently, this resulted in the construction of a perturbative series expansion bifurcating from initial data dominated only by a fixed number of modes. Inserting this expansion in the non-linear system, one obtains an infinite system of harmonic oscillators in which secular\footnote{These are terms of the form $\tau^{\alpha} \sin( \beta \tau)$ and $\tau^{\alpha} \cos( \beta \tau)$ for some $\alpha>0$ and $\beta \in \mathbb{R}$.} terms appear naturally, see also Section \ref{SectionSecularTerms}. According to \cite{PhysRevLett107031102}, if these terms are not removed, they should eventually become responsible for exponential instability and a possible breakdown of the solutions. For example, if one starts with initial data dominated by the first 2-modes, namely $e_0 +  e_1$, then a secular term occurs that cannot be removed, for more details see \cite{PhysRevLett107031102} and page 20 in \cite{2011PhRvL107c1102B}. Their important work then triggered a substantial amount of numerical and heuristic studies. Dias-Horowitz-Santos \cite{MR2978943} considered pure gravity with a negative cosmological constant and provided additional support strengthening the evidence that the AdS spacetime might be nonlinearly unstable. Moreover, similar results have been obtained by Ja{\l}mu{\.z}na-Rostworowski-Bizo\'{n} \cite{PhysRevD84085021} and Buchel-Lehner-Liebling \cite{86123011} for higher dimensions as well as by Choptuik \cite{PhysRevLett709} who studied the mechanism of the spherically symmetric collapse of a scalar field with a general time and radial spatial dependent metric and for several families of initial data.\\ \\
Most importantly, Bizo\'{n}-Rostworowski \cite{PhysRevLett107031102} also conjectured that there may exist specific initial data (called islands of stability) for which the evolution of small perturbations around the AdS solution remains globally regular in time. As a matter of fact, Maliborski-Rostworowski \cite{PhysRevLett111051102} considered the spherically symmetric Einstein-massless scalar field equations with negative cosmological constant in $1+d$ dimensions with $d\geq 2$ and provided reliable numerical evidence indicating that actually time-periodic solutions may exist for non-generic initial data. They were able to construct these solutions using both nonlinear perturbative expansions and fully nonlinear numerical methods. Specifically, they expanded the dynamical variables in terms of the eigenfunctions $\{e_n: n \geq 0\}$ to the associated linearized operator and constructed a perturbative series expansion bifurcating from initial data dominated only by 1-modes. Then, they inserted this in the non-linear system and obtained a recurrence relation where one can fine tune the initial data order-by-order in such a way so that all secular terms appearing in the resonant system can be removed. Consequently, they are left with a secular free perturbative series that presumably converges obtaining a time-periodic solution from non-generic initial data. However, their analysis suggests that all secular terms can be removed provided that one can verify an infinite system of both linear and non-linear conditions where the so called Fourier coefficients play a decisive role in its validity. These are oscillatory integrals defined by the mode couplings induced by the non-linearities of the system and are essentially weighted integrals of products of the eigenfunctions $\{e_n: n \geq 0\}$ to the associated linearized operator, see Sections \ref{SectionPerturbation1DefinitionFourier} and \ref{SectionPerturbation2DefinitionFourier} as well as \cite{MR3435560,PhysRevLett107031102,PhysRevLett111051102}. In this work, we consider the coupled Einstein-Klein-Gordon system in spherical symmetry, track down various cancellations occurring in the potentially resonant terms and establish robust bounds for the Fourier coefficients using two types of perturbations. \\ \\
Moreover, similar conjectures were made by Dias-Horowitz-Marolf-Santos \cite{MR3002881} who argued that many asymptotically AdS solutions are nonlinearly stable (including geons, boson stars, and black holes) and by Buchel-Liebling-Lehner \cite{87123006} who considered boson stars in global AdS spacetime and study their stability. Finally, a very interesting non-perturbative approach was developed by Moschidis who rigorously proved the AdS instability conjecture for generic initial data and for the Einstein--massless Vlasov system and Einstein--null dust system in his celebrated works \cite{MR4150259, MR4542703, MR3828498}.  

\subsection{Spherical symmetric ansatz}
We consider the coupled Einstein-Klein-Gordon system \eqref{arxiko} in spherical symmetry and for simplicity we fix the spatial dimension $d=3$. Following the work of Bizo\'{n}-Rostworowski \cite{PhysRevLett107031102} we parametrize the spacetime metric by the spherically symmetric ansatz
\begin{align}\label{symmetricansatz}
	g(t,x,\omega) = \frac{1}{\cos ^2 (x)} \left( -   A(t,x)e^{-2\delta(t,x)} dt^2 + A^{-1}(t,x) dx^2+ \sin ^2 (x) d \omega ^2 \right),
\end{align}
for $(t,x,\omega) \in \mathcal{M}$. Under this ansatz, the wave equation in spherical symmetry becomes
\begin{align} \label{originalwaveequation}
	\partial_{t} \left( A^{-1}(t,x)e^{\delta(t,x)} \partial_{t} \phi(t,x) \right) = \frac{1}{\tan^2(x)} \partial_{x} \left( \tan^2(x) A(t,x)e^{-\delta(t,x)} \partial_{x} \phi(t,x) \right). 
\end{align}
We note that the gauge condition used here includes the normalization $\delta=0$ and $A=1$ at $x=0$. We transform the second order partial differential equation \eqref{originalwaveequation} into a first order system by setting
\begin{align*}
	\Phi (t,x) = \partial_{x} \phi (t,x),~~~\Pi (t,x) = \frac{1}{A(t,x)e^{-\delta (t,x)} } \partial_{t} \phi(t,x).
\end{align*}
Then, \eqref{originalwaveequation} reads
\begin{align*}
\begin{dcases}
	 \partial_{t} \Phi(t,x) = \partial_{x} \left( A(t,x)e^{-\delta(t,x)} \Pi(t,x) \right), \\ 
	 \partial_{t} \Pi(t,x) = - \reallywidehat{L} \left(  A(t,x)e^{-\delta (t,x)} \Phi (t,x) \right),
\end{dcases}
\end{align*}
where
\begin{align*}
	\reallywidehat{L}[g](x):= - \frac{1}{\tan^2(x)} \partial_{x} ( \tan^2(x) g(x)),
\end{align*}
coupled to the Einstein equations
\begin{align*}
\begin{dcases}
	(1-A(t,x))e^{-\delta(t,x)}= \frac{\cos ^3 (x)}{\sin(x)} \int_{0}^{x} e^{-\delta(t,y)} \left(\Phi^2(t,y) + \Pi^2(t,y) \right) (\tan(y))^2 dy, \\
	-\delta(t,x) = \int_{0}^{x} \left(\Phi^2(t,y) + \Pi^2(t,y) \right) \sin(y) \cos(y) dy.
\end{dcases}
\end{align*}
\subsection{The linearized operator}
From the Einstein equations, one can derive an additional equation, namely the momentum constraint
\begin{align*}
	\partial_{t} A(t,x) = - 2 \sin(x)\cos(x) A(t,x) \partial_{x}\phi(t,x) \partial_{t}\phi(t,x)
\end{align*}
and now \eqref{originalwaveequation} can be written as
\begin{align*}
	\partial_{t}^2 \phi(t,x) +   L\left[\phi (t,x)\right] &=  \frac{1}{2} \partial_{x} \big( \big( A(t,x)e^{-\delta(t,x)} \big)^2 \big) \partial_{x} \phi(t,x) + \big(1- \big(A(t,x)e^{-\delta(t,x)} \big)^2  \big)  L\left[\phi(t,x) \right]\\
	&- 2 \sin(x) \cos(x) \partial_{x} \phi(t,x) \left(\partial_{t} \phi(t,x) \right)^2 - \partial_{t} \delta(t,x) \partial_{t} \phi(t,x),
\end{align*}
where 
\begin{align*}
	L[f](x):= -\frac{1}{\tan^2(x)} \partial_{x} ( \tan^2(x) \partial_{x} f(x))
\end{align*}
is the operator which governs linearized perturbations of AdS solution. The solutions to the eigenvalue problem $L[f] = \omega^2 f$ subject to Dirichlet boundary conditions on the conformal boundary $\mathcal{I}=\{x=\pi /2 \}$ 
fall into the hypergeometric class and hence can be found explicitly. For a rigorous definition of the spectrum as well as for the definition of the domain in which the linearized operator is self-adjoint, see the work of Bachelot \cite{MR2430631}. Specifically, the eigenvalues read 
\begin{align*}
	\omega^2_{j}:=(3+2j)^2 
\end{align*}
and eigenfunctions are weighted Jacobi polynomials,
\begin{align*}
	e_{j}(x):=2\frac{\sqrt{j! (j+2)!}}{\Gamma (j+\frac{3}{2})} \cos^3(x) P^{\frac{1}{2},\frac{3}{2}}_{j}(\cos(2x)),  
\end{align*}
for all integers $j=0,1,2,\dots$ and $x\in[0,\pi/2]$. For the definition, basic properties and an introduction to the Jacobi polynomials see Chapter 4, page 48 in Szeg\"{o}'s book \cite{MR0372517}. In addition, the linearized operator $L$ is self-adjoint with respect to the weighted inner product
\begin{align} \label{innerproduct}
	(f|g):=\int _{0}^{\frac{\pi}{2}} f(x)g(x) \tan^{2}(x) dx.
\end{align}
Finally, note that the eigenvalues are strictly positive and hence the linear problem is orbitally stable. 
\section{Main result and preliminaries}
We consider the spherically symmetric Einstein-massless scalar field equations with negative cosmological constant under the spherically symmetric ansatz \eqref{symmetricansatz},
\begin{align}
	& \partial_{t} \Phi(t,x) = \partial_{x} \left( A(t,x)e^{-\delta(t,x)} \Pi(t,x) \right), \label{EKG1} \\ 
	&  \partial_{t} \Pi(t,x) = - \reallywidehat{L} \left(  A(t,x)e^{-\delta (t,x)} \Phi (t,x) \right), \label{EKG2}\\
	& (1-A(t,x))e^{-\delta(t,x)}= \frac{\cos ^3 (x)}{\sin(x)} \int_{0}^{x} e^{-\delta(t,y)} \left(\Phi^2(t,y) + \Pi^2(t,y) \right) (\tan(y))^2 dy,\label{EKG3} \\
	& \delta(t,x) = -\int_{0}^{x} \left(\Phi^2(t,y) + \Pi^2(t,y) \right) \sin(y) \cos(y) dy, \label{EKG4}
\end{align} 
and we are mainly interested in the asymptotic behaviour of the Fourier constants which appear in the analysis of perturbations around the AdS solution $(\Phi, \Pi, A, \delta)=(0,0,1,0)$.
\subsection{Statement of the main result}
Specifically, we consider two types of perturbations. On the one hand, in light of recent work Maliborski-Rostworowski \cite{PhysRevLett111051102}, although the series may not converge, we seek a solution of the form
\begin{align}
	&\Phi(t,x) =\sum_{\lambda=0}^{\infty} \psi_{2\lambda+1} (\tau,x)\epsilon^{2\lambda+1}  , \quad  \Pi(t,x)=\sum_{\lambda=0}^{\infty} \sigma_{2\lambda+1} (\tau,x)\epsilon^{2\lambda+1} , \label{series1} \\
	& A(t,x) e^{-\delta(t,x)} =\sum_{\lambda=0}^{\infty} \xi_{2\lambda} (\tau,x)\epsilon^{2\lambda}, \quad  e^{-\delta(t,x)}=\sum_{\lambda=0}^{\infty} \zeta_{2\lambda} (\tau,x)\epsilon^{2\lambda}  ,\label{series3}  \\ 
	& \tau = \Omega_{\gamma} t, \quad
	\Omega_{\gamma} = \sum_{\lambda=0}^{\infty} \omega_{\gamma,2\lambda} \epsilon^{2\lambda} ,\label{series4}
\end{align}
where $\psi_{2\lambda+1}$,$\sigma_{2\lambda+1}$,$\xi_{2\lambda}$ and $\zeta_{2\lambda}$ are all periodic in time and in particular $\xi_{0} = \zeta_{0}  = 1 $. Here, $\psi_{1},\sigma_{1},\xi_{2},\zeta_{2},\omega_{\gamma,0},\gamma$ will be chosen later. Furthermore, with a slight abuse of notation, we use the same letters to denote the variables with respect to the $(\tau,x)$ and $(t,x)$.  
On the other hand, we still assume that $(\Phi,\Pi,A,\delta)$ are all close to the AdS solution $(0,0,1,0)$ but expand them using a finite sum
\begin{align}
	& \Phi(t,x) = \Phi_{1}(\tau,x) \epsilon + \Psi(\tau,x) \epsilon ^3 , \quad 
	\Pi(t,x) = \Pi_{1}(\tau,x) \epsilon + \Sigma(\tau,x) \epsilon ^3 
	\label{finitesum1} \\ 
&A(t,x) = 1-A_{2}(\tau,x) \epsilon^2 - B(\tau,x) \epsilon^4 , \quad 
e^{-\delta (t,x)} = 1 - \delta_{2} (\tau,x) \epsilon^2 - \Theta (\tau,x) \epsilon^4 ,\label{finitesum3} \\
& \tau = (\omega_{\gamma} + \epsilon^2 \theta_{\gamma} +\epsilon^4 \eta_{\gamma})t, \label{finitesum5}
\end{align}
for some error terms $\Psi, \Sigma, B,\Theta,\eta_{\gamma}$ where $\Phi_{1},\Pi_{1},A_{2},\delta_{2}$ are all explicit periodic expressions in time and will be chosen later together with $\gamma$.  \\ \\
The aim of this paper is to provide a better understanding of the asymptotic behaviour of the Fourier coefficients that play a crusial role to the work of Maliborski-Rostworowski \cite{PhysRevLett111051102} and the conjectured existence of time-periodic solutions to the Eistein-Klein-Gordon equation in spherical symmetry. Is it also motivated by the work of Hunik-Kostyra-Rostworowski \cite{zbMATH07227949} who established interesting recurrence relations for the interaction coefficients for the $5-$dimensional Einstein equations in vacuum with negative cosmological constant within cohomogenity-two biaxial Bianchi IX ansatz. To place our results in the context of the physics literature, we refer the reader to \cite{PhysRevLett113071601, MR4016456, PhysRevLett115081103, MR3270350, PhysRevD92084001}. In our main result, we track down the various cancellations occurring in the potentially resonant terms and establish robust bounds for the Fourier coefficients in the two settings described above. 
\begin{theorem}[Rough version]\label{maintheorem}
We establish robust bounds for the interaction coefficients of the various Fourier components in the perturbative expansions around the AdS:
	\begin{itemize}
		\item Perturbation 1: the series \eqref{series1}-\eqref{series4}
		\item Perturbation 2: the finite sum \eqref{finitesum1}-\eqref{finitesum5}.
	\end{itemize}
For a precise version, see Propositions \ref{Result1} and \ref{A1}-\ref{B2} respectively.
\end{theorem}
\subsection{Acknowledgments}
The author would like to thank Professor Jacques Smulevici for very useful  comments and insights and gratefully acknowledges the support of the ERC grant 714408 GEOWAKI under the European Union's Horizon 2020 research and innovation program. The author also acknowledges funding by the Alexander von Humboldt Foundation endowed by the Federal Ministry of Education and Research, by the ERC Consolidator Grant 772249 as well as by the Germany's Excellence Strategy EXC 2044 390685587, Mathematics Munster: Dynamics-Geometry-Structure.

\subsection{Preliminaries}
As we will see, there are subspaces from which these Fourier constants decay and subspaces from which they grow. On the one hand, in order to establish the decay estimates, we will use
\begin{itemize}
\item the leading order terms of the eigenfunctions and their weighted derivatives (Lemma \ref{ClosedformulasFore} and Remark \ref{lot})
\item the asymptotic behaviour of specific oscillatory integrals (Lemma \ref{OscillatoryIntegrals})
\item integration by parts (Lemma \ref{byparts1} and Lemma \ref{byparts2})
\end{itemize}
The second is based on 
\begin{itemize}
	\item Dirichlet-Kernel-type identities (Lemma \ref{Dirichlet})
	\item carefully chosen anti-derivatives
\end{itemize}
On the other hand, for the growth estimates, we will only use
\begin{itemize}
	\item Holder's inequality
	\item $L^{\infty}-$bounds for quantities related to the eigenfunctions (Lemma \ref{Linftyboundse})
\end{itemize}
To begin with, we prove the first auxiliary result.
\begin{lemma}[Closed formulas]\label{ClosedformulasFore}
For all $x \in [0, \pi/2 ]$ and $i=0,1,\dots$, we have 
\begin{align*}
	& e_{i}(x) = \frac{2}{\sqrt{\pi}} \frac{ 1}{\sqrt{\omega_{i}^2-1}}  \left( \omega_{i} \frac{\sin\left( \omega_{i} x \right)}{\tan(x)} - \cos \left( \omega_{i}x \right) \right), \\
	& \frac{e_{i}^{\prime}(x)}{\omega_{i}} = \frac{2}{\sqrt{\pi}} \frac{ 1}{\sqrt{\omega_{i}^2-1}} \left(  \omega_{i}\frac{\cos\left( \omega_{i} x \right)}{\tan(x)} - \frac{\sin \left( \omega_{i}x \right)}{\tan^2(x)}  \right),
\end{align*}
Furthermore, both $\{e_{i}:i=0,1,2,\dots\} $ and $\{e_{i}^{\prime} / \omega_{i}:i=0,1,2,\dots \} $ form an orthogonal basis for $L^{2}\left( \left[0,\pi/2  \right] \right)$ with respect to the weighted inner product \eqref{innerproduct}, namely
\begin{align*}
	(e_{i}|e_{j}) = \delta_{i,j},\quad (e_{i}^{\prime}|e_{j}^{\prime}) = \omega_{i}^2 \delta_{i,j}
\end{align*}
and, for all $i,j=0,1,\dots,$. Here, $\delta_{i,j}$ stands for the Kronecker's delta. 
\end{lemma}
\begin{proof}
For the first part, we make use of the facts 
\begin{align*}
& P^{\frac{1}{2},\frac{3}{2}}_{j} (z) 
=  \frac{2\Gamma(j+2)}{\Gamma(j+3)} \frac{d}{d z} P^{-\frac{1}{2},\frac{1}{2}}_{j+1} (z), \quad P^{-\frac{1}{2},\frac{1}{2}}_{j+1} (z) = 
\frac{1}{\sqrt{\pi}} \frac{\Gamma(j+\frac{3}{2})}{\Gamma(j+2) }
\frac{\cos(\omega_{j} x)}{\cos(x)}
\end{align*}
where $x=x(z)=\frac{1}{2} \arccos(z)$ and $j=0,1,\dots$. These identities can be found in Chapter 4, page 60, equation (4.1.8) and Chapter 4, page 63, equation (4.21.7) in Szeg\"{o}'s book \cite{MR0372517}. Then, the closed formula for $e_{j}$ follows by the chain rule. Indeed, we define $z=\cos(2x)$ and compute
\begin{align*}
	P^{\frac{1}{2},\frac{3}{2}}_{j} (z) 
& =  \frac{2\Gamma(j+2)}{\Gamma(j+3)}
\frac{1}{\sqrt{\pi}} \frac{\Gamma(j+\frac{3}{2})}{\Gamma(j+2) }
\frac{d}{d z} \left( \frac{\cos(\omega_{j} x)}{\cos(x)}\right) \\
& =\frac{2}{\sqrt{\pi}} \frac{\Gamma(j+\frac{3}{2})}{\Gamma(j+3)}
\frac{dx}{d z} \frac{d}{d x} \left( \frac{\cos(\omega_{j} x)}{\cos(x)} \right) \\
& =\frac{2}{\sqrt{\pi}} \frac{\Gamma(j+\frac{3}{2})}{\Gamma(j+3)}
\frac{-1}{2 \sin(2x)} \frac{-\omega_{j} \cos(x) \sin(\omega_{j}x)+\cos(\omega_{j}x) \sin(x)}{\cos^2(x)} \\
& =\frac{1}{\sqrt{\pi}} \frac{\Gamma(j+\frac{3}{2})}{\Gamma(j+3)}
\left( \frac{\omega_{j}}{2\sin(x) \cos^2(x)} \sin(\omega_{j}x) 
-\frac{1}{2 \cos^3(x)} \cos(\omega_{j}x)
\right)
\end{align*}
and so
\begin{align*}
e_{j}(x) & :=2\frac{\sqrt{j! (j+2)!}}{\Gamma (j+\frac{3}{2})} \cos^3(x) P^{\frac{1}{2},\frac{3}{2}}_{j}(\cos(2x)) \\
	& =\frac{\sqrt{j! (j+2)!}}{\Gamma (j+\frac{3}{2})}  \frac{1}{\sqrt{\pi}} \frac{\Gamma(j+\frac{3}{2})}{\Gamma(j+3)}
\left(\omega_{j} \frac{\cos(x)}{\sin(x)} \sin(\omega_{j}x) 
- \cos(\omega_{j}x) 
\right) \\
& = \frac{1}{\sqrt{\pi}} \frac{\sqrt{j! (j+2)!}}{\Gamma(j+3) } 
\left(\omega_{j} \frac{\sin(\omega_{j}x) }{\tan(x)}
- \cos(\omega_{j}x) 
\right).
\end{align*}
Finally, since $j$ is an integer, we conclude
\begin{align*}
	\frac{\sqrt{j! (j+2)!}}{\Gamma(j+3) } 
	= \frac{\sqrt{j! (j+2)!}}{ (j+2)! } 
	= \sqrt{\frac{j!}{(j+2)!}}
	= \sqrt{\frac{1}{(j+1)(j+2)}} 
	= \frac{2}{\sqrt{\omega_{i}^2-1}}
\end{align*}
The closed formula for $e_{j}^{\prime}$ follows by differentiating the closed formula for $e_{j}$.  
Using these formulas, the orthogonality properties are straightforward. For the fact that the set $\{e_{j}:j=0,1,2,\dots \}$ forms a basis for $L^{2}\left( \left[0,\frac{\pi}{2} \right] \right)$ with respect to the weighted inner product \eqref{innerproduct}, see the Appendix in the work of Bachelot \cite{MR2430631}. In order to show that $\{ \frac{e^{\prime}_{j}}{\omega_{j}}:j=0,1,2,\dots \}$ also forms a basis for the same function space, one has to prove that
\begin{align*}
	\left(f \Big| \frac{e^{\prime}_{j}}{\omega_{j}} \right) = 0,\quad \forall j=0,1,2,\dots \Longrightarrow f=0.
\end{align*}
To this end, we define
\begin{align*}
	F(x):= \int_{x}^{\frac{\pi}{2}} f(y) dy
\end{align*}
and use the fact that
\begin{align*}
	- (\tan^2(x) e_{j}^{\prime}(x))^{\prime} = \omega_{j}^2 \tan^2(x) e_{j}(x)
\end{align*}
which follows from $L[e_{j}]=\omega_{j}^2 e_{j}$ together with integration by parts to compute
\begin{align*}
	0=\left(f \Big| \frac{e^{\prime}_{j}}{\omega_{j}} \right) &=
	\int_{0}^{\frac{\pi}{2}} f(x) \frac{e^{\prime}_{j}(x)}{\omega_{j}} \tan^2(x) dx
	=-\int_{0}^{\frac{\pi}{2}} F^{\prime}(x) \frac{e^{\prime}_{j}(x)}{\omega_{j}} \tan^2(x) dx \\
	&=\int_{0}^{\frac{\pi}{2}} F(x) \left( \frac{e^{\prime}_{j}(x)}{\omega_{j}} \tan^2(x) \right)^{\prime} dx 
	=\int_{0}^{\frac{\pi}{2}} F(x) \left( \frac{e^{\prime}_{j}(x)}{\omega_{j}} \tan^2(x) \right)^{\prime} dx \\
	&=-\omega_{j} \int_{0}^{\frac{\pi}{2}} F(x) e_{j}(x) dx
\end{align*}
for all $j=0,1,2,\dots$. Now, we use the fact that $\{e_{j}:j=0,1,2,\dots \}$ forms a basis to get $F = 0$ which in turn implies $f=0$.
\end{proof}
\begin{remark}\label{lot}
We find the leading order terms
	\begin{align*}
	& e_{i}(x) = \frac{2}{\sqrt{\pi}} \left( \frac{ \omega_{i}}{\sqrt{\omega_{i}^2-1}} \frac{\sin\left( \omega_{i} x \right)}{\tan(x)} -\frac{ 1}{\sqrt{\omega_{i}^2-1}}  \cos \left( \omega_{i}x \right) \right) \simeq \frac{2}{\sqrt{\pi}} \frac{\sin\left( \omega_{i} x \right)}{\tan(x)},\\
	& \frac{e_{i}^{\prime}(x)}{\omega_{i}} = \frac{2}{\sqrt{\pi}} \left( \frac{ \omega_{i}}{\sqrt{\omega_{i}^2-1}} \frac{\cos\left( \omega_{i} x \right)}{\tan(x)} - \frac{1}{\sqrt{\omega_{i}^2-1}} \frac{\sin \left( \omega_{i}x \right)}{\tan^2(x)}  \right) \simeq \frac{2}{\sqrt{\pi}} \frac{\cos \left( \omega_{i} x \right)}{\tan(x)},
\end{align*}
as $i \longrightarrow \infty$ and for all $x\in \left[ 0, \frac{\pi}{2} \right]$. These estimates are uniform with respect to the weighted $L^2-$norm
\begin{align*}
	\| f \|
	:=\left( \int_{0}^{\frac{\pi}{2}} f^2(x) \tan^2 (x) dx \right)^{\frac{1}{2}}.
\end{align*}
Indeed, for large $i$, we estimate
\begin{align*}
	\left \| e_{i} - \frac{2}{\sqrt{\pi}} \frac{\sin\left( \omega_{i} \cdot \right)}{\tan} \right \|
	&\leq \frac{2}{\sqrt{\pi}} \left( \frac{ \omega_{i}}{\sqrt{\omega_{i}^2-1}} -1 \right) \left \| \frac{\sin(\omega_{i} \cdot)}{\tan} \right \|  +\frac{2}{\sqrt{\pi}} \frac{1}{\sqrt{\omega_{i}^2-1}} \left \| \cos(\omega_{i} \cdot) \right \| \\
	& \lesssim  \frac{ \omega_{i}}{\sqrt{\omega_{i}^2-1}} -1  +\sqrt{ \frac{\omega_{i}}{\omega_{i}^2-1} 	} \lesssim \frac{1}{\sqrt{\omega_{i}}}, \\
	\left \| \frac{e_{i}^{\prime}}{\omega_{i}} - \frac{2}{\sqrt{\pi}} \frac{\cos\left( \omega_{i} \cdot \right)}{\tan} \right \|
	& \leq \frac{2}{\sqrt{\pi}} \left(\frac{ \omega_{i}}{\sqrt{\omega_{i}^2-1}} -1 \right) \left \| \frac{\cos(\omega_{i} \cdot)}{\tan} \right \|  + \frac{2}{\sqrt{\pi}} \frac{1}{\sqrt{\omega_{i}^2-1}} \left \| \frac{\sin(\omega_{i} \cdot)}{\tan^2} \right \| \\
	& \lesssim  \frac{ \omega_{i}}{\sqrt{\omega_{i}^2-1}} -1  +\sqrt{ \frac{\omega_{i}}{\omega_{i}^2-1} 	} \lesssim \frac{1}{\sqrt{\omega_{i}}},
\end{align*}
since
\begin{align}
     & \left \| \frac{\sin(\omega_{i} \cdot)}{\tan} \right \|^2 
     = \int_{0}^{\frac{\pi}{2}} \sin^2 (\omega_{i}x) dx = \frac{\pi}{4}, \quad 
     \left \| \frac{\cos(\omega_{i} \cdot)}{\tan} \right \|^2 
	 = \int_{0}^{\frac{\pi}{2}} \cos^2 (\omega_{i}x) dx = \frac{\pi}{4}, \nonumber \\
	 &\left \| \cos(\omega_{i} \cdot) \right \|^2
     = \int_{0}^{\frac{\pi}{2}} \cos^2 (\omega_{i}x) \tan^2(x)dx  \lesssim \omega_{i}, \quad  \left \| \frac{\sin(\omega_{i} \cdot)}{\tan^2} \right \|^2 =
	 \int_{0}^{\frac{\pi}{2}} \frac{\sin^2 (\omega_{i}x)}{\tan^2(x)} dx \lesssim \omega_{i}. \label{AuxiliaryIntegralsBounds}
\end{align} 
The proof of \eqref{AuxiliaryIntegralsBounds} is given in the Appendix, see Lemma \ref{lemmaAppendix}. 
\end{remark}
Next, we prove $L^{\infty}-$bounds for quantities related to the eigenfunctions.
\begin{lemma}[$L^{\infty}$ bounds] \label{Linftyboundse}
For all $j=0,1,2,\dots$, we have
\begin{align*}
 \sup_{x \in \left[0,\frac{\pi}{2} \right]} \left| e_{j}(x) \right| \leq \frac{2}{\sqrt{\pi}} ~\omega_{j},  \quad   
 \sup_{x \in \left[0,\frac{\pi}{2} \right]} \left| 
	\tan(x) \frac{e_{j}^{\prime}(x)}{\omega_{j}} \right| &\leq \frac{4}{\sqrt{\pi}},\\ \sup_{x \in \left[0,\frac{\pi}{2} \right]} \left| \int_{x}^{\frac{\pi}{2}} e_{j}(y) \sin(y) \cos(y) dy \right| &\leq \frac{2}{\omega_{j}}.
\end{align*}
\end{lemma}
\begin{proof}
For the first estimate, we define the oscillating part of $e_{j}(x)$, namely
\begin{align*}
	f_{j}(x):= \omega_{j} \frac{\sin(\omega_{j} x)}{\tan(x)} - \cos(\omega_{j} x),
\end{align*}	
for all $x\in [0,\frac{\pi}{2}]$ and $j=0,1,2,\dots$ and compute its derivative
\begin{align*}
	f_{j}^{\prime}(x) = \frac{\omega_{j}}{\tan(x)} \left( \omega_{j} \cos(\omega_{j} x) - \frac{\sin(\omega_{j}x)}{\tan(x)}  \right) = 
\omega_{j}	\frac{\cos(\omega_{j}x)}{\tan(x)} \left( \omega_{j}  - \frac{\tan(\omega_{j}x)}{\tan(x)}  \right).
\end{align*}
Hence, for all $x\in (0,\frac{\pi}{2})$ and $j=0,1,2,\dots$, the equation
\begin{align*}
	f_{j}^{\prime}(x) = 0 \Longleftrightarrow  \tan(\omega_{j} x) = \omega_{j} \tan(x)
\end{align*}
has countable many solutions, say $x=x_{j}^{\star}\in (0,\frac{\pi}{2})$. Then,
\begin{align*}
	f_{j}(x_{j}^{\star}) &= \cos(\omega_{j} x_{j}^{\star}) \Bigg( \omega_{j} \frac{\tan(\omega_{j} x_{j}^{\star}) }{\tan(x_{j}^{\star}) } -1  \Bigg) = \cos(\omega_{j} x_{j}^{\star}) \Bigg( \omega_{j} \frac{\omega_{j} \tan( x_{j}^{\star}) }{\tan(x_{j}^{\star}) } -1  \Bigg) = \cos(\omega_{j} x_{j}^{\star}) \big( \omega_{j}^2-1 \big).
\end{align*}
Now, since
\begin{align*}
	f_{j}(0) = \lim_{x \rightarrow 0}  f_{j}(x) = \omega_{j}^2-1,\quad \quad f_{j}\left(\frac{\pi}{2} \right) = \lim_{x \rightarrow \frac{\pi}{2}}  f_{j}(x) = 0,
\end{align*}
we get
\begin{align*}
	\sup_{x \in \left[0,\frac{\pi}{2} \right]} \left| f_{j}(x) \right| = \max \left\{ \left| f_{j}(0) \right|,\left| f_{j}(x_{j}^{\star}) \right|,\left| f_{j} \left( \frac{\pi}{2} \right) \right| \right\} \leq \omega_{j}^2-1
\end{align*}
and finally
\begin{align*}
	\sup_{x \in \left[0,\frac{\pi}{2} \right]} \left| e_{j}(x) \right|  = \frac{2}{\sqrt{\pi}} \frac{1}{\sqrt{\omega_{j}^2-1}} \sup_{x \in \left[0,\frac{\pi}{2} \right]} \left| f_{j}(x) \right| = \frac{2}{\sqrt{\pi}} \sqrt{\omega_{j}^2-1} \leq \frac{2}{\sqrt{\pi}}\omega_{j}.
\end{align*}
For the second estimate, we observe that
\begin{align*}
	\tan(x) \frac{e_{j}^{\prime}(x)}{\omega_{j}} & =\frac{2}{\sqrt{\pi}} \frac{1}{\sqrt{\omega_{j}^2-1}} \left(  \omega_{j} \cos(\omega_{j} x) - \frac{\sin(\omega_{j}x)}{\tan(x)} \right) \\
& =\frac{2}{\sqrt{\pi}} \frac{1}{\sqrt{\omega_{j}^2-1}} \left(  \omega_{j} \cos(\omega_{j} x) - \frac{1}{\omega_{j}}\omega_{j} \frac{\sin(\omega_{j}x)}{\tan(x)} + \frac{1}{\omega_{j}} \cos(\omega_{j}x) -\frac{1}{\omega_{j}} \cos(\omega_{j}x) \right)\\
& =\frac{2}{\sqrt{\pi}} \frac{1}{\sqrt{\omega_{j}^2-1}} \left(  \omega_{j} \cos(\omega_{j} x) - \frac{1}{\omega_{j}}
\left( 
\omega_{j} \frac{\sin(\omega_{j}x)}{\tan(x)} -  \cos(\omega_{j}x) \right) -\frac{1}{\omega_{j}} \cos(\omega_{j}x) 
\right)\\
& =\frac{2}{\sqrt{\pi}} \frac{1}{\sqrt{\omega_{j}^2-1}} 
\left(
\left(  \omega_{j} -\frac{1}{\omega_{j}} \right) \cos(\omega_{j} x) - \frac{1}{\omega_{j}} f_{j}(x)  \right).
\end{align*} 
Hence, by triangular inequality, 
\begin{align*}
	\sup_{x \in \left[0,\frac{\pi}{2} \right]} \left| 
	\tan(x) \frac{e_{j}^{\prime}(x)}{\omega_{j}} \right| 
	& \leq
	\frac{2}{\sqrt{\pi}} \frac{1}{\sqrt{\omega_{j}^2-1}} 
\left(
\left(  \omega_{j} -\frac{1}{\omega_{j}} \right) + \frac{1}{\omega_{j}} \sup_{x \in \left[0,\frac{\pi}{2} \right]} \left| f_{j}(x) \right|  \right) \\
& \leq
	\frac{2}{\sqrt{\pi}} \frac{1}{\sqrt{\omega_{j}^2-1}} 
\left(
\left(  \omega_{j} -\frac{1}{\omega_{j}} \right) + \frac{\omega_{j}^2-1}{\omega_{j}}  \right) 
 \leq \frac{4}{\sqrt{\pi}},
\end{align*}
for all $j=0,1,2,\dots$. For the last estimate, we define
\begin{align*}
	g_{j}(x):=\int_{x}^{\frac{\pi}{2}} e_{j}(y) \sin(y) \cos(y) dy
\end{align*}
for all $x\in \left[0,\frac{\pi}{2} \right]$ and $j=0,1,2,\dots$ and compute its derivative
\begin{align*}
	g_{j}^{\prime}(x) = -e_{j}(x)\sin(x) \cos(x). 
\end{align*}
As above, for all $x\in (0,\frac{\pi}{2})$ and $j=0,1,2,\dots$, the equation
\begin{align*}
	g_{j}^{\prime}(x) = 0 \Longleftrightarrow \omega_{j} \sin(\omega_{j}x) = \cos(\omega_{j}x) \tan(x)
\end{align*}
has countable many solutions, say $x=x_{j}^{\prime}\in (0,\frac{\pi}{2})$. Then,
\begin{align*}
	g_{j}(x_{j}^{\prime}) &=
	\int_{x_{j}^{\prime}}^{\frac{\pi}{2}} e_{j}(y) \sin(y)\cos(y) dy \\
	&= \frac{2}{\sqrt{\pi}} \frac{1}{\sqrt{\omega_{j}^2-1}} \left( 
	\omega_{j} \int_{x_{j}^{\prime}}^{\frac{\pi}{2}} \sin(\omega_{j}y) \cos^2(y) dy - \int_{x_{j}^{\prime}}^{\frac{\pi}{2}} \cos(\omega_{j}y) \sin(y)\cos(y) dy 
	\right) \\
&  = \frac{2}{\sqrt{\pi}} \frac{1}{\sqrt{\omega_{j}^2-1}} 
	\frac{\cos(\omega_{j}x_{j}^{\prime})}{\omega_{j}^2-4} \Big( 
	  - 1-2 \sin^2(x_{j}^{\prime}) + \omega_{j}^2 \cos^2(x_{j}^{\prime})
	+ 3  \sin^2(x_{j}^{\prime}) 
	\Big)	\\	
&  = \frac{2}{\sqrt{\pi}} \frac{1}{\sqrt{\omega_{j}^2-1}} 
	\frac{\cos(\omega_{j}x_{j}^{\prime})}{\omega_{j}^2-4} (\omega_{j}^2-1) \cos^2(x_{j}^{\prime}) 
= \frac{2}{\sqrt{\pi}} \frac{\sqrt{\omega_{j}^2-1} }{\omega_{j}^2-4} \cos(\omega_{j}x_{j}^{\prime})  \cos^2(x_{j}^{\prime})		
\end{align*}
and since
\begin{align*}
	g_{j}(0) = \lim_{x \rightarrow 0}  g_{j}(x) =  \frac{2}{\sqrt{\pi}} \frac{\sqrt{\omega_{j}^2-1} }{\omega_{j}^2-4} ,\quad \quad g_{j}\left(\frac{\pi}{2} \right) = \lim_{x \rightarrow \frac{\pi}{2}}  g_{j}(x) = 0,
\end{align*}
we get
\begin{align*}
	\sup_{x \in \left[0,\frac{\pi}{2} \right]} \left| g_{j}(x) \right| 
	\leq  \frac{2}{\sqrt{\pi}} \frac{\sqrt{\omega_{j}^2-1} }{\omega_{j}^2-4} \leq \frac{2}{\omega_{j}},
\end{align*}
valid for all $j=0,1,2,\dots$, which concludes the proof.
\end{proof}
For future reference, we also prove the following Dirichlet-Kernel-type identities.
\begin{lemma}[Dirichlet-Kernel-type identities]\label{Dirichlet}
For any $n \in \mathbb{N}$ and $x\in\mathbb{R}$, we have
\begin{align*}
& \frac{\sin((2n+1)x)}{\sin(x)} = 1+ 2\sum_{\mu=1}^{n} \cos(2 \mu x), \\
& \frac{\cos((2n+1)x)}{\cos(x)} =(-1)^{n}\left( 1+ 2\sum_{\mu=1}^{n} (-1)^{\mu} \cos(2 \mu x) \right), \\
& \frac{\sin(2 n x)}{\tan(x)} = 1+ \cos(2nx) + 2\sum_{\mu=1}^{n-1} \cos(2 \mu x).
\end{align*}
\begin{proof}
	The first result is well-known (Dirichlet Kernel). For the second, we use the first one and just replace $x$ by $x+\pi /2$,
	\begin{align*}
	 & 1+ 2\sum_{\mu=1}^{n} (-1)^{\mu} \cos(2 \mu x)  =
	1+ 2\sum_{\mu=1}^{n} \cos(2 \mu x+\pi \mu) = 1+ 2\sum_{\mu=1}^{n} \cos \left(2 \mu \left( x+ \frac{\pi}{2} \right)\right ) \\ 
		&= \frac{\sin \left( (2n+1) \left(x+\frac{\pi}{2} \right) \right)}{\sin \left(x+\frac{\pi}{2} \right)} 
		 = \frac{\sin((2n+1)x)\cos \left( (2n+1) \frac{\pi}{2} \right)+\cos((2n+1)x)\sin \left( (2n+1) \frac{\pi}{2} \right)}{\cos(x)} \\
		&= (-1)^n \frac{\cos((2n+1)x)}{\cos(x)}.
	\end{align*}
	For the third, we observe that
\begin{align*}
	& \frac{\sin(2nx)}{\tan(x)} +\cos(2nx)  = \frac{\sin(2nx) \cos(x) + \cos(2nx) \sin(x)}{\sin(x)} \\
	& = \frac{\sin(2n+1)x}{\sin(x)}  = 1+ 2\sum_{\mu=1}^{n} \cos(2 \mu x) 
	 = 1+ 2\sum_{\mu=1}^{n-1} \cos(2 \mu x) + 2 \cos(2 n x),
\end{align*}
from which the result follows. 
\end{proof}
\end{lemma}
Finally, we establish the asymptotic behaviour of specific oscillatory integrals which will appear later.
\begin{lemma}[Oscillatory integrals]\label{OscillatoryIntegrals}
	For any $N \in \mathbb{N}$, we have
	\begin{align*}
		 \int_{0}^{\frac{\pi}{2}} \frac{\sin((2a+1)x)}{\tan(x)} dx 
		&= \frac{\pi}{2} + \mathcal{O} \left( \frac{1}{a^{2}} \right),\\
		 \int_{0}^{\frac{\pi}{2}}  \cos^2(x) \frac{\sin(2 a x )}{\tan(x)} dx& = \frac{\pi}{2} + \mathcal{O} \left( \frac{1}{a^{N}} \right), \\
		 \int_{0}^{\frac{\pi}{2}}  \cos^3(x) \frac{\sin((2a+1)x)}{\tan(x)} dx &= \frac{\pi}{2} + \mathcal{O} \left( \frac{1}{a^{N}} \right),\\
		 \int_{\frac{\pi}{a}}^{\frac{\pi}{2}}   \frac{ \cos(2 a x)}{\tan^2(x)} dx  &= c a + \mathcal{O}  \left( \frac{1}{a^{3}} \right),
	\end{align*} 
as $a \longrightarrow \infty$ with $a \in \mathbb{N}$. Here,
\begin{align*}
	c:= \frac{1}{\pi} - \pi + 2 \int_{0}^{1} \frac{\sin(2 \pi y )}{y}dy \simeq 0.01302.
\end{align*}
\end{lemma}
\begin{proof}
For the first integral, we use Lemma \ref{Dirichlet} to infer
\begin{align*}
	\int_{0}^{\frac{\pi}{2}} \frac{\sin((2a+1)x)}{\tan(x)} dx & = 
\int_{0}^{\frac{\pi}{2}} \cos(x) \frac{\sin((2a+1)x)}{\sin(x)} dx = \int_{0}^{\frac{\pi}{2}} \cos(x) \left( 1+ 2 \sum_{\mu=1}^{a} \cos(2 \mu x)  \right) dx \\
& = \int_{0}^{\frac{\pi}{2}} \cos(x) dx + 2 \sum_{\mu=1}^{a} \int_{0}^{\frac{\pi}{2}} \cos(x) \cos(2 \mu x)  dx \\
& = 1 + 2 \sum_{\mu=1}^{a} \frac{(-1)^{\mu}}{1-4 \mu^2}  = 1 + 2 \sum_{\mu=1}^{\infty} \frac{(-1)^{\mu}}{1-4 \mu^2} 
- 2 \sum_{\mu=a+1}^{\infty} \frac{(-1)^{\mu}}{1-4 \mu^2} \\
& = 1 + 2 \frac{\pi-2}{4} 
- 2 \sum_{\mu=a+1}^{\infty} \frac{(-1)^{\mu}}{1-4 \mu^2} \\
& = \frac{\pi}{2}
+ 2 \sum_{\mu=a+1}^{\infty} \frac{(-1)^{\mu}}{4 \mu^2-1} 
= \frac{\pi}{2} + \mathcal{O} \left( \frac{1}{a^{2}} \right).
\end{align*}
The second integral follows similarly. Lemma \ref{Dirichlet} implies
\begin{align*}
	& \int_{0}^{\frac{\pi}{2}}  \cos^2(x) \frac{\sin(2 a x)}{\tan(x)} dx  = \int_{0}^{\frac{\pi}{2}}  \cos^2(x) \left( 
	1+ \cos(2 a x) + 2 \sum_{\mu =1}^{a-1} \cos(2 \mu x)
	\right) dx \\
	& = \int_{0}^{\frac{\pi}{2}}  \cos^2(x) dx + \int_{0}^{\frac{\pi}{2}}  \cos^2(x) \cos(2 a x) dx + 2 \sum_{\mu=1}^{a-1} \int_{0}^{\frac{\pi}{2}}  \cos^2(x)\cos(2 \mu x) 
	dx \\
& = \frac{\pi}{4}+0+2 \frac{\pi}{8}  = \frac{\pi}{2},	
\end{align*}
for all $a \geq 3$, since
\begin{align*}
 \int_{0}^{\frac{\pi}{2}}  \cos^2(x) dx = \frac{\pi}{4},\quad \int_{0}^{\frac{\pi}{2}}  \cos^2(x) \cos(2 \lambda x) dx =  
\begin{dcases}
\frac{\pi}{8}, &  \lambda = 1 \\
0, &  \lambda \geq 2. 
\end{dcases} 
\end{align*}
Next, for the third integral, Lemma \ref{Dirichlet} also yields
\begin{align*}
	& \int_{0}^{\frac{\pi}{2}}  \cos^3(x) \frac{\sin((2a+1)x)}{\tan(x)} dx  = \int_{0}^{\frac{\pi}{2}} \cos^4(x) \frac{\sin((2a+1)x)}{\sin(x)} dx \\
	& = \int_{0}^{\frac{\pi}{2}}  \cos^4(x) \Bigg( 1+ 2\sum_{\mu=1}^{a} \cos(2 \mu x)\Bigg) dx  = \int_{0}^{\frac{\pi}{2}} \cos^4(x) dx + 
2\sum_{\mu=1}^{a} 	 \int_{0}^{\frac{\pi}{2}}  \cos^4(x) \cos(2 \mu x)  dx \\
	& = \frac{3 \pi }{16} + 2 \frac{\pi }{8} + 2 \frac{ \pi }{32} = \frac{ \pi }{2},
\end{align*}
for all $a \geq 3$, since
\begin{align*}
 \int_{0}^{\frac{\pi}{2}}  \cos^4(x) dx = \frac{3\pi}{16},\quad \int_{0}^{\frac{\pi}{2}}  \cos^2(x) \cos(2 \lambda x) dx =  
\begin{dcases}
\frac{\pi}{8}, &  \lambda = 1 \\
\frac{\pi}{32}, &  \lambda = 2 \\
0, &  \lambda \geq 3. 
\end{dcases}  
\end{align*}
Finally, we conclude with the fourth integral. First, observe that
\begin{align*}
	2 \frac{\cos(2 a x)}{\tan^2(x)} + 4 a \frac{\sin(2 a x)}{\tan(x)} + \Bigg( 2 \frac{\cos(2ax)}{\tan(x)} + \frac{\sin(2ax)}{a} \Bigg)^{\prime} = 0
\end{align*}
and therefore, for any $\epsilon >0$,
\begin{align*}
	\int_{\epsilon}^{\frac{\pi}{2}}  2 \frac{ \cos(2 a x)}{\tan^2(x)} dx & = \int_{\epsilon}^{\frac{\pi}{2}}  \left( -\Bigg( 2 \frac{\cos(2ax)}{\tan(x)} + \frac{\sin(2ax)}{a} \Bigg)^{\prime} - 4 a \frac{\sin(2 a x)}{\tan(x)} \right) dx \nonumber \\
	& = 2 \frac{\cos(2a \epsilon)}{\tan(\epsilon)} + \frac{\sin(2a \epsilon)}{a} - 4a \int_{\epsilon}^{\frac{\pi}{2}} \frac{\sin(2 a x)}{\tan(x)} dx.
\end{align*}
Second, we set $\epsilon = \frac{\pi}{a}$ and get $\cos(2a \epsilon) = 1, \sin(2a \epsilon) = 0$. Now, Lemma \ref{Dirichlet} shows that
 \begin{align*}
 	\int_{0}^{\frac{\pi}{2}} \frac{\sin(2 a x)}{\tan(x)} dx = 
  \frac{\pi}{2}+ \int_{0}^{\frac{\pi}{2}}  \cos(2ax) dx + 2\sum_{\mu=1}^{a-1} \int_{0}^{\frac{\pi}{2}}  \cos(2 \mu x) dx 
= 	\frac{\pi}{2}+0+0 = \frac{\pi}{2},
 \end{align*}
 valid for all $a \geq 1$, and we change variables $y=ax / \pi $ to infer
 \begin{align*}
 	\int_{\frac{\pi}{a}}^{\frac{\pi}{2}}  \frac{ \cos(2 a x)}{\tan^2(x)} dx &= 
 	\frac{1}{\tan(\frac{\pi}{a})} -2 a \int_{\frac{\pi}{a}}^{\frac{\pi}{2}} \frac{\sin(2 a x)}{\tan(x)} dx = \frac{1}{\tan(\frac{\pi}{a})} -2 a \left( \frac{\pi}{2} - \int_{0}^{\frac{\pi}{a}} \frac{\sin(2 a x)}{\tan(x)} dx 
\right)\\
&= \frac{1}{\tan(\frac{\pi}{a})} - a \pi + 2\pi \int_{0}^{1} \frac{\sin(2 \pi y )}{\tan(\frac{\pi y }{a})} dy,
 \end{align*}
for all $a \geq 2$. Now, for all $y \in [0,1]$ and $a \longrightarrow \infty$,
 \begin{align*}
 	\frac{1}{\tan(\frac{\pi y }{a})} = \frac{a}{\pi} \frac{1}{y} - \frac{\pi}{3 a} y + \mathcal{O} \left( \left( \frac{y}{a} \right)^{3} \right)
 \end{align*}
 and hence
 \begin{align*}
 	\int_{0}^{1} \frac{\sin(2 \pi y )}{\tan(\frac{\pi y }{a})} dy & =
 	 \frac{a}{\pi} \int_{0}^{1} \frac{\sin(2 \pi y )}{y}dy -\frac{\pi}{3 a}\int_{0}^{1}   y\sin(2 \pi y ) dy +\int_{0}^{1} \sin(2 \pi y ) \mathcal{O}\left( \left( \frac{y}{a} \right)^{3} \right) dy \\
 	 & = \frac{a}{\pi} \int_{0}^{1} \frac{\sin(2 \pi y )}{y}dy -\frac{\pi}{3 a} \frac{-1}{2 \pi} +\mathcal{O} \left( \frac{1}{a^{3}} \right) \\
 	 & =  \frac{a}{\pi} \int_{0}^{1} \frac{\sin(2 \pi y )}{y}dy +\frac{1}{6 a}  +\mathcal{O} \left( \frac{1}{a^{3}} \right)
 \end{align*}
from which we conclude
 \begin{align*}
 	\int_{\frac{\pi}{a}}^{\frac{\pi}{2}}  \frac{ \cos(2 a x)}{\tan^2(x)} dx &= \frac{1}{\tan(\frac{\pi}{a})} - a \pi + 2\pi \int_{0}^{1} \frac{\sin(2 \pi y )}{\tan(\frac{\pi y }{a})} dy \\
 	&=\frac{1}{\tan(\frac{\pi}{a})} - a \pi+ 2a \int_{0}^{1} \frac{\sin(2 \pi y )}{y}dy +\frac{\pi}{3 a}  +\mathcal{O} \left( \frac{1}{a^{3}} \right) \\
 	& =\frac{a}{\pi}-\frac{\pi}{3 a}+\mathcal{O} \left( \frac{1}{a^{3}} \right) - a \pi+ 2a \int_{0}^{1} \frac{\sin(2 \pi y )}{y}dy +\frac{\pi}{3 a}  +\mathcal{O} \left( \frac{1}{a^{3}} \right) \\
 	& =a \left(
 	\frac{1}{\pi} - \pi + 2 \int_{0}^{1} \frac{\sin(2 \pi y )}{y}dy
 	\right) 
 	+\mathcal{O} \left( \frac{1}{a^{3}} \right), 
 \end{align*}
 that completes the proof.
\end{proof}
\begin{remark}\label{e0oscilating}
A straightforward computation yields 
\begin{align*}
	e_{0}(x) = 4 \sqrt{\frac{2}{\pi}} \cos^3(x)
\end{align*}
and hence the third integral of Lemma \ref{OscillatoryIntegrals} also gives
	\begin{align*}
		& \int_{0}^{\frac{\pi}{2}}  e_{0}(x) \frac{\sin((2a+1)x)}{\tan(x)} dx = 2 \sqrt{2 \pi} + \mathcal{O} \left( \frac{1}{a^{N}} \right), \\	
	\end{align*}
for $a \longrightarrow \infty$.	
\end{remark}
Finally, we will also use standard integration by parts.
\begin{lemma}[Integration by parts]\label{byparts1}
	Let $a \in \mathbb{N}$ and $N \in \mathbb{N}$. Assume that $F$ is differentiable in $\left (0,\frac{\pi}{2} \right)$ and continuous in $\left [0,\frac{\pi}{2} \right]$. Then,
	\begin{align*}
 \int_{0}^{\frac{\pi}{2}} F(x) \cos( a x )  dx   = 
&   \sum_{k=0}^{N} \frac{(-1)^{k}}{a^{2k+1}} \left( F^{(2k)}(x) \sin(a x) \right) \Big|_{x=0}^{x=\frac{\pi}{2}}   +  \sum_{k=0}^{N} \frac{(-1)^{k}}{a^{2k+2}} \left( F^{(2k+1)}(x) \cos(a x) \right) \Big|_{x=0}^{x=\frac{\pi}{2}} 
  \\
 & + \frac{(-1)^{N-1}}{a^{2N+2}} \int_{0}^{\frac{\pi}{2}} F^{(2N+2)}(x) \cos(a x) dx, \\
 \int_{0}^{\frac{\pi}{2}} F(x) \sin( a x )  dx  = 
 &  \sum_{k=0}^{N} \frac{(-1)^{k+1}}{a^{2k+1}} \left( F^{(2k)}(x) \cos(a x) \right) \Big|_{x=0}^{x=\frac{\pi}{2}}  +  \sum_{k=0}^{N} \frac{(-1)^{k}}{a^{2k+2}} \left( F^{(2k+1)}(x) \sin(a x) \right) \Big|_{x=0}^{x=\frac{\pi}{2}} 
  \\
 & + \frac{(-1)^{N-1}}{a^{2N+2}} \int_{0}^{\frac{\pi}{2}} F^{(2N+2)}(x) \sin(a x) dx.
\end{align*} 
\end{lemma}
\begin{proof}
The proof is a straight forward application of integration by parts.
\end{proof}
\noindent
The following result is a direct consequence of Lemma \ref{byparts1}.
\begin{lemma} \label{byparts2}
Let $b \in \mathbb{N}$ and $N \in \mathbb{N}$. Assume that $F$ is differentiable in $\left (0,\frac{\pi}{2} \right)$, continuous in $\left [0,\frac{\pi}{2} \right]$ with uniformly bounded derivatives in $\left [0,\frac{\pi}{2} \right]$. Then, as $b \longrightarrow \infty$, we have
\begin{align*}
&	\int_{0}^{\frac{\pi}{2}} F(x) \cos( 2 b x )  dx   = 
  \sum_{k=0}^{N} \frac{(-1)^{k}}{(2b)^{2k+2}} \left( (-1)^{b} F^{(2k+1)} \left(  \frac{\pi}{2} \right) -F^{(2k+1)} \left( 0 \right)  \right) + \mathcal{O} \left(  \frac{1}{b^{2N+2}} \right), \\
&	\int_{0}^{\frac{\pi}{2}} F(x) \sin( 2 b x )  dx   = 
 \sum_{k=0}^{N} \frac{(-1)^{k+1}}{(2b)^{2k+1}} \left( (-1)^{b} F^{(2k)} \left(  \frac{\pi}{2} \right) -F^{(2k)} \left( 0 \right)  \right) + \mathcal{O} \left(  \frac{1}{b^{2N+2}} \right), \\
& \int_{0}^{\frac{\pi}{2}} F(x) \cos( (2 b+1) x )  dx   =  \sum_{k=0}^{N} \frac{(-1)^{k+b}}{(2b+1)^{2k+1}} F^{(2k)} \left(  \frac{\pi}{2} \right) 
+ \sum_{k=0}^{N} \frac{(-1)^{k+1}}{(2b+1)^{2k+2}} F^{(2k+1)} \left( 0 \right) 
  + \mathcal{O} \left(  \frac{1}{b^{2N+2}} \right), \\
& \int_{0}^{\frac{\pi}{2}} F(x) \sin( (2 b+1) x )  dx   = 
\sum_{k=0}^{N} \frac{(-1)^{k}}{(2 b+1)^{2k+1}} F^{(2k)} \left(  0 \right) 
+ \sum_{k=0}^{N} \frac{(-1)^{k+b}}{(2 b+1)^{2k+2}} F^{(2k+1)} \left( \frac{\pi}{2} \right) 
  + \mathcal{O} \left(  \frac{1}{b^{2N+2}} \right).
\end{align*}
\end{lemma}
\begin{proof}
All these asymptotic expansions follow from Lemma \ref{byparts1} just by computing the boundary terms. Note that if $a$ is even, namely $a = 2b$ for some $b \in \mathbb{N}$, then $\sin(a \frac{\pi}{2}) = 0$ and $ \cos(a \frac{\pi}{2}) = (-1)^{b}$ whereas if $a$ is odd, namely $a = 2b+1$ for some $b \in \mathbb{N}$, then $\sin(a \frac{\pi}{2}) = (-1)^b$ and $ \cos(a \frac{\pi}{2}) = 0$. For example, Lemma \ref{byparts1} yields 
\begin{align*} 
\int_{0}^{\frac{\pi}{2}} F(x) \cos( 2 b x )  dx  & = 
 \sum_{k=0}^{N} \frac{(-1)^{k}}{a^{2k+2}} \left( (-1)^{b} F^{(2k+1)} \left(  \frac{\pi}{2} \right) -F^{(2k+1)} \left( 0 \right)  \right) \\
& + \frac{(-1)^{N-1}}{a^{2N+2}} \int_{0}^{\frac{\pi}{2}} F^{(2N+2)}(x) \cos(a x) dx.
\end{align*}
The other asymptotic expansions follow similarly.
\end{proof}
\noindent
With these auxiliary results at hand we proceed to the main analysis of the Fourier constants.
\section{Perturbations 1: Series}\label{SectionPerturbation1}
As mentioned above, Maliborski-Rostworowski \cite{PhysRevLett111051102} considered \eqref{EKG1}-\eqref{EKG4} and provided strong numerical evidence indicating that time-periodic solutions exist for non-generic initial data. To explain their approach and see how the Fourier constants appear, first assume that $(\Phi, \Pi, A, \delta)$ are all close to $(0,0,1,0)$ and expand $(\Phi, \Pi, A, \delta)$ in terms of powers of $\epsilon$ using the series \eqref{series1}-\eqref{series4}. 
\subsection{Definition of the Fourier constants}\label{SectionPerturbation1DefinitionFourier}
First, we compute the density
\begin{align*}
	\Phi^2(t,x)+\Pi^2(t,x) = \sum_{\lambda=1}^{\infty} r_{2\lambda} \epsilon^{2\lambda}
\end{align*}
where, for all $\lambda = 0,1,2,\dots$,
\begin{align*}
	r_{2(\lambda+1)}(\tau,x) = \sum_{\substack{\mu,\nu=0 \\ \mu+\nu=\lambda }}^{\lambda}
	\left( \psi_{2\mu+1}(\tau,x)\psi_{2\nu+1}(\tau,x) +
	 \sigma_{2\mu+1}(\tau,x)\sigma_{2\nu+1} (\tau,x)
	\right).
\end{align*}
Next, we substitute these expressions into \eqref{EKG1}-\eqref{EKG4}, collect terms of the same order in $\epsilon$ and obtain a hierarchy of equations
\begin{align*} 
	\omega_{\gamma} \partial_{\tau} \psi_{2\lambda+1} (\tau,x) - \partial_{x} \sigma_{2\lambda+1} (\tau,x) &= 
	\sum_{\substack{\mu,\nu=0 \\ \mu+\nu=\lambda \\ (\mu,\nu) \neq (\lambda,0) }}^{\lambda}
	\Big (
	- \omega_{\gamma,2\nu} \partial_{\tau} \psi_{2\mu+1} (\tau,x) + \partial_{x} \left( \xi_{2\nu}(\tau,x) \sigma_{2\mu+1}(\tau,x) \right)
	 \Big ), 
	\\
	 \omega_{\gamma} \partial_{\tau} \sigma_{2\lambda+1} (\tau,x) 
	+ \reallywidehat{L}  \left[ \psi_{2\lambda+1}(\tau,x) \right]
	&= -  \sum_{\substack{\mu,\nu=0 \\ \mu+\nu=\lambda \\ (\mu,\nu) \neq (\lambda,0) }}^{\lambda} 
	\Big(
	\omega_{\gamma,2\nu}\partial_{\tau} \sigma_{2\mu+1}(\tau,x)+  \reallywidehat{L}  \left[ \xi_{2\nu}(\tau,x) \psi_{2\mu+1}(\tau,x) \right]
	\Big),\\
	  \xi_{2(\lambda+1)} (\tau,x) &= \zeta_{2(\lambda+1)} (\tau,x) - \frac{\cos^3(x)}{\sin(x)} 
	\sum_{\substack{\mu,\nu=0 \\ \mu+\nu=\lambda  }}^{\lambda} \int_{0}^{x} r_{2(\mu+1)}(\tau,y) \zeta_{2\nu}(\tau,y) \tan^2(y) dy,\\
	 \zeta_{2(\lambda+1)} (\tau,x) &=  
	\sum_{\substack{\mu,\nu=0 \\ \mu+\nu=\lambda }}^{\lambda} \int_{0}^{x} r_{2(\mu+1)}(\tau,y) \zeta_{2\nu}(\tau,y)\sin(y)\cos(y)dy,
\end{align*}
for all $\lambda=0,1,2,\dots$, together with 
\begin{align*}
     \omega_{\gamma,0} = \omega_{\gamma}, \quad
	 \xi_{0}(\tau,x) = 1,\quad
	 \zeta_{0}(\tau,x) = 1, \quad  
\begin{cases}
	\omega_{\gamma} \partial_{\tau} \psi_{1}(\tau,x) - \partial_{x} \sigma_{1} (\tau,x) = 0, \\
	\omega_{\gamma} \partial_{\tau} \sigma_{1}(\tau,x)+ \reallywidehat{L} \left[ \psi_{1}(\tau,x) \right]=0.
\end{cases}
\end{align*}
From the set of all eigenvalues $\{e_{i}\}_{i=0}^{\infty}$ to the linear operator, we choose a dominant mode $e_{\gamma}$ for some $\gamma \in \{0,1,2,\dots \}$ and pick
\begin{align*}
\begin{cases}
	\psi_{1} (\tau,x) = \cos(\tau) e^{\prime}_{\gamma}(x),\\
	\sigma_{1} (\tau,x) =  -\omega_{\gamma} \sin(\tau) e_{\gamma}(x).
\end{cases}
\end{align*}
Now, for each $\lambda=0,1,2,\dots$, we expand the coefficients $\psi_{2\lambda+1},\sigma_{2\lambda+1},\xi_{2\lambda},\zeta_{2\lambda}$ in terms of the eigenvalues of the linearized operator, namely
\begin{align*}
	& \psi_{2\lambda+1}(\tau,x)=\sum_{i=0}^{\infty} f_{2\lambda+1}^{(i)}(\tau) \frac{e^{\prime}_{i}(x)}{\omega_{i}}, \quad \sigma_{2\lambda+1}(\tau,x)=\sum_{i=0}^{\infty} g_{2\lambda+1}^{(i)} (\tau) e_{i}(x), \\
	& \xi_{2\lambda}(\tau,x)=\sum_{i=0}^{\infty} p_{2\lambda}^{(i)}(\tau) e_{i}(x), \quad \zeta_{2\lambda}(\tau,x)=\sum_{i=0}^{\infty} q_{2\lambda}^{(i)} (\tau) e_{i}(x),
\end{align*}
substitute these expressions into the recurrence relations above, take the inner product $(\cdot|e^{\prime}_{m})$ for the first equation, the $(\cdot|e_{m})$ for all the other equations and use their orthogonality properties $(e^{\prime}_{n}|e^{\prime}_{m})=\omega_{n}^2 \delta_{nm}$, $(e_{n},e_{m})=\delta_{nm}$. Using the notation
\begin{align*}
	\dot{f}(\tau) = \frac{d f(\tau)}{d \tau},\quad 
\ddot{f} (\tau) = \frac{d^2 f(\tau)}{d \tau^2},
\end{align*}
we find
\begin{align*}
	& \omega_{\gamma} \dot{f}_{2\lambda+1}^{(m)} (\tau) = \omega_{m} g_{2\lambda+1}^{(m)} (\tau) +\sum_{\substack{\mu,\nu=0 \\ \mu+\nu=\lambda \\ (\mu,\nu) \neq (\lambda,0) }}^{\lambda}
	\left(
	-\omega_{\gamma,2\nu} \dot{f}_{2\mu+1}^{(m)}(\tau)+ \omega_{m}
	\sum_{i,j=0}^{\infty} C_{ij}^{(m)} p_{2\nu}^{(i)}(\tau) g_{2\mu+1}^{(j)}(\tau)
	\right), \\
	&
	\omega_{\gamma} \dot{g}_{2\lambda+1}^{(m)} (\tau) = - \omega_{m} f_{2\lambda+1}^{(m)} (\tau) -  \sum_{\substack{\mu,\nu=0 \\ \mu+\nu=\lambda \\ (\mu,\nu) \neq (\lambda,0) }}^{\lambda}
	\left(
	\omega_{\gamma,2\nu}  \dot{g}_{2\mu+1}^{(m)}(\tau) + \omega_{m} \sum_{i,j=0}^{\infty}   \overline{C}_{ij}^{(m)} p_{2\nu}^{(i)}(\tau) f_{2\mu+1}^{(j)}(\tau) \right),  \\
	& p_{2(\lambda+1)}^{(m)}(\tau) = 	\sum_{\substack{\rho,k,\nu=0 \\ \rho+k+\nu=\lambda  }}^{\lambda} \sum_{i,j,l=0}^{\infty}
	\Bigg(
	\widetilde{A}_{ijl}^{(m)} f_{2\rho+1}^{(i)}(\tau)f_{2k+1}^{(j)}(\tau)+\widetilde{B}_{ijl}^{(m)} g_{2\rho+1}^{(i)}(\tau)g_{2k+1}^{(j)}(\tau)
	\Bigg) q_{2\nu}^{(l)}(\tau),\nonumber \\
	& q_{2(\lambda+1)}^{(m)}(\tau) = 
	\sum_{\substack{\rho,k,\nu=0 \\ \rho+k+\nu=\lambda  }}^{\lambda} \sum_{i,j,l=0}^{\infty}
	\Bigg(
	\frac{\overline{A}_{ijl}^{(m)}}{\omega_{m}}  f_{2\rho+1}^{(i)}(\tau)f_{2k+1}^{(j)}(\tau)+
	\frac{\overline{B}_{ijl}^{(m)}}{\omega_{m}} g_{2\rho+1}^{(i)}(\tau)g_{2k+1}^{(j)}(\tau)
	\Bigg) q_{2\nu}^{(l)}(\tau),\nonumber
\end{align*} 
where all the interactions with respect to the spatial variable $x \in \left[0,\frac{\pi}{2} \right]$ are included in the following Fourier constants
\begin{align*}
	C_{ij}^{(m)} & :=  \int_{0}^{\frac{\pi}{2}}  e_{i}(x) e_{j}(x) e_{m}(x) \tan^2(x) dx, \\
    \overline{C}_{ij}^{(m)} &:= 
     \int_{0}^{\frac{\pi}{2}}  e_{i}(x) \frac{e_{j}^{\prime}(x)}{\omega_{j}} \frac{e_{m}^{\prime}(x)}{\omega_{m}} \tan^2(x) dx, \\
    \overline{A}_{ijl}^{(m)} &:= 
     \int_{0}^{\frac{\pi}{2}}  \frac{ e_{i}^{\prime}(x)}{\omega_{i}} \frac{ e_{j}^{\prime}(x)}{\omega_{j}} e_{l}(x) \frac{ e_{m}^{\prime}(x)}{\omega_{m}}  \frac{\sin^3(x)}{\cos(x)} dx, \\
    \overline{B}_{ijl}^{(m)} &:= 
     \int_{0}^{\frac{\pi}{2}}  e_{i}(x) e_{j}(x) e_{l}(x) \frac{ e_{m}^{\prime}(x)}{\omega_{m}}  \frac{\sin^3(x)}{\cos(x)} dx, \\
     \widetilde{A}_{ijl}^{(m)} &:= \frac{\overline{A}_{ijl}^{(m)}}{\omega_{m}}  - 
     \int_{0}^{\frac{\pi}{2}} \frac{ e_{i}^{\prime}(x)}{\omega_{i}} \frac{e_{j}^{\prime}(x)}{\omega_{j}} e_{l}(x) 
     \left( 
     \int_{x}^{\frac{\pi}{2}} e_{m}(y) \sin(y)\cos(y) dy
     \right) \tan^2(x)dx, , \\
     \widetilde{B}_{ijl}^{(m)} &:= \frac{\overline{B}_{ijl}^{(m)}}{\omega_{m}}  - \int_{0}^{\frac{\pi}{2}} e_{i}(x) e_{j}(x) e_{l}(x) 
     \left( 
     \int_{x}^{\frac{\pi}{2}} e_{m}(y) \sin(y)\cos(y) dy
     \right) \tan^2(x)dx.
\end{align*}
We also find 
\begin{align*}
	f_{1}^{(m)} (\tau) =\omega_{\gamma} \cos(\tau) \delta_{\gamma}^{m}, \quad 
	g_{1}^{(m)} (\tau) = - \omega_{\gamma} \sin(\tau) \delta_{\gamma}^{m}, \quad
	p_{0}^{(m)} (\tau) = q_{0}^{(m)} (\tau) = (1|e_{m})
\end{align*}
and use Lemma \ref{ClosedformulasFore} to compute
\begin{align*}
(1|e_{m}):=	\int_{0}^{\frac{\pi}{2}} e_{m}(x) \tan^2(x) dx = \frac{2}{\sqrt{\pi}} \frac{(-1)^m}{\omega_{m}} \sqrt{\omega_{m}^2-1},
\end{align*}
for all $m=0,1,2,\dots$. In addition, we differentiate the first equation with respect to $\tau$ and use the second to obtain the harmonic oscillator equation
\begin{align} \label{HarmonicOscilator}
	\ddot{f}_{2\lambda+1}^{(m)} (\tau) + \left(\frac{\omega_{m}}{\omega_{\gamma}} \right)^2 f_{2\lambda+1}^{(m)} (\tau) = S_{2\lambda+1}^{(m)}(\tau)
\end{align}
where the source term is given by
\begin{align*}
	S_{2\lambda+1}^{(m)}(\tau) &:=\frac{ \omega_{m} }{ \omega_{\gamma}} \sum_{\substack{\mu,\nu=0 \\ \mu+\nu=\lambda \\ (\mu,\nu) \neq (\lambda,0) }}^{\lambda}
	\Bigg[
	-\frac{\omega_{\gamma,2\nu}}{\omega_{\gamma}} 
	\left(
	\dot{g}_{2\mu+1}^{(m)}(\tau) +\frac{ \omega_{\gamma} }{ \omega_{m}} \ddot{f}_{2\mu+1}^{(m)}(\tau)
	\right) \\
	& + 
	 \sum_{i,j=0}^{\infty} 
\left(
	C_{ij}^{(m)} 
	\frac{d}{d \tau} \left(p_{2\nu}^{(i)}(\tau) g_{2\mu+1}^{(j)}(\tau) \right)
-  \frac{\omega_{m}}{\omega_{\gamma}}  \overline{C}_{ij}^{(m)} p_{2\nu}^{(i)}(\tau) f_{2\mu+1}^{(j)}(\tau) 
\right) \Bigg].
\end{align*}	
Finally, we make use of the variation constants formula to solve  \eqref{HarmonicOscilator} and find
\begin{align*}
	f_{2\lambda+1}^{(m)} (\tau) &=
	f_{2\lambda+1}^{(m)} (0) \cos \left( \frac{\omega_{m}}{\omega_{\gamma}} \tau \right) + \frac{\omega_{\gamma}}{\omega_{m}} \dot{f}_{2\lambda+1}^{(m)} (0) \sin \left( \frac{\omega_{m}}{\omega_{\gamma}}\tau \right) + \frac{\omega_{\gamma}}{\omega_{m}}  \int_{0}^{\tau} \sin \left( \frac{\omega_{m}}{\omega_{\gamma}} (\tau-s) \right) S_{2\lambda+1}^{(m)}(s) ds.
\end{align*}
In conclusion, we get for all $m=0,1,2,\dots$ the following recurrence relations. For all $\lambda=1,2,3,\dots$,
\begin{align*}
& f_{1}^{(m)} (\tau) =\omega_{\gamma} \cos(\tau) \delta_{\gamma}^{m}, \\
& f_{2\lambda+1}^{(m)} (\tau)  = f_{2\lambda+1}^{(m)} (0) \cos \left( \frac{\omega_{m}}{\omega_{\gamma}} \tau \right) + \frac{\omega_{\gamma}}{\omega_{m}} \dot{f}_{2\lambda+1}^{(m)} (0) \sin \left( \frac{\omega_{m}}{\omega_{\gamma}}\tau \right) + \frac{\omega_{\gamma}}{\omega_{m}}  \int_{0}^{\tau} \sin \left( \frac{\omega_{m}}{\omega_{\gamma}} (\tau-s) \right) S_{2\lambda+1}^{(m)}(s) ds, \\ \\
& S_{2\lambda+1}^{(m)}(\tau) =\frac{ \omega_{m} }{ \omega_{\gamma}} \sum_{\substack{\mu,\nu=0 \\ \mu+\nu=\lambda \\ (\mu,\nu) \neq (\lambda,0) }}^{\lambda}
	\Bigg[
	-\frac{\omega_{\gamma,2\nu}}{\omega_{\gamma}} 
	\left(
	\dot{g}_{2\mu+1}^{(m)}(\tau) +\frac{ \omega_{\gamma} }{ \omega_{m}} \ddot{f}_{2\mu+1}^{(m)}(\tau)
	\right) \\
	&\qquad \qquad \qquad \qquad + 
	 \sum_{i,j=0}^{\infty} 
\left(
	C_{ij}^{(m)} 
	\frac{d}{d \tau} \left(p_{2\nu}^{(i)}(\tau) g_{2\mu+1}^{(j)}(\tau) \right)
-  \frac{\omega_{m}}{\omega_{\gamma}}  \overline{C}_{ij}^{(m)} p_{2\nu}^{(i)}(\tau) f_{2\mu+1}^{(j)}(\tau) 
\right) \Bigg]
\\ \\
		&g_{1}^{(m)} (\tau)  = - \omega_{\gamma} \sin(\tau) \delta_{\gamma}^{m}, \\
		&g_{2\lambda+1}^{(m)} (\tau)  =
	\frac{\omega_{\gamma}}{\omega_{m}} \dot{f}_{2\lambda+1}^{(m)} (\tau)
	+\sum_{\substack{\mu,\nu=0 \\ \mu+\nu=\lambda \\ (\mu,\nu) \neq (\lambda,0)}}^{\lambda}
	\Bigg[
	\frac{ \omega_{\gamma,2\nu}}{\omega_{m}} 
	\dot{f}_{2\mu+1}^{(m)}(\tau) - \sum_{i,j=0}^{\infty} C_{ij}^{(m)} p_{2\nu}^{(i)}(\tau) g_{2\mu+1}^{(j)}(\tau)\Bigg], \\ \\
    & p_{0}^{(m)}(\tau)  = \frac{2}{\sqrt{\pi}} \frac{(-1)^m}{\omega_{m}} \sqrt{\omega_{m}^2-1},  \\
    & p_{2(\lambda+1)}^{(m)}(\tau) = 	\sum_{\substack{\rho,k,\nu=0 \\ \rho+k+\nu=\lambda \\  }}^{\lambda} \sum_{i,j,l=0}^{\infty}
		\Bigg[
	\widetilde{A}_{ijl}^{(m)}
	f_{2\rho+1}^{(i)}(\tau)f_{2k+1}^{(j)}(\tau)+
	\widetilde{B}_{ijl}^{(m)}
	g_{2\rho+1}^{(i)}(\tau)g_{2k+1}^{(j)}(\tau)
		\Bigg] q_{2\nu}^{(l)}(\tau), \\ \\
    &q_{0}^{(m)}(\tau)  = \frac{2}{\sqrt{\pi}} \frac{(-1)^m}{\omega_{m}} \sqrt{\omega_{m}^2-1}, \\
	&q_{2(\lambda+1)}^{(m)}(\tau) = 
	\sum_{\substack{\rho,k,\nu=0 \\ \rho+k+\nu=\lambda   }}^{\lambda} \sum_{i,j,l=0}^{\infty}
	\Bigg[
	\frac{\overline{A}_{ijl}^{(m)}}{\omega_{m}}  f_{2\rho+1}^{(i)}(\tau)f_{2k+1}^{(j)}(\tau)+
	\frac{\overline{B}_{ijl}^{(m)}}{\omega_{m}} g_{2\rho+1}^{(i)}(\tau)g_{2k+1}^{(j)}(\tau)
		\Bigg] q_{2\nu}^{(l)}(\tau).	
\end{align*}
\subsection{time-periodic solutions and secular terms}\label{SectionSecularTerms}
As pointed out in \cite{PhysRevLett111051102}, non-periodic terms appear naturally when the source term $S_{2\lambda+1}^{(m)}(\tau)$ has terms of the form $\cos (  \omega_{m}  \tau / \omega_{\gamma} )$ or $\sin (  \omega_{m}  \tau / \omega_{\gamma} )$ in its Fourier expansion. Indeed, we assume that, for some $\lambda=1,2,3,\dots$, 
\begin{align*}
	S_{2\lambda+1}^{(m)}(\tau) = \sum_{a \in I_{\lambda}} S_{1,2\lambda+1,a}^{(m)} \cos(a x)+\sum_{b \in J_{\lambda}} S_{2,2\lambda+1,b}^{(m)} \sin(b x),
\end{align*}
and in addition there exists an index $m =0,1,2,\dots$ such that $\omega_{m}/\omega_{\gamma}:=a \in I_{\lambda} $.
Then, the integral
\begin{align*}
	\int_{0}^{\tau} \sin \left( \frac{\omega_{m}}{\omega_{\gamma}} (\tau-s) \right) S_{2\lambda+1}^{(m)}(s) ds
\end{align*}
produces a non-periodic term since
\begin{align*}
	\int_{0}^{\tau} \sin \left( \frac{\omega_{m}}{\omega_{\gamma}} (\tau-s) \right) \cos \left( a s \right) ds & = 
	\int_{0}^{\tau}\sin \left( \frac{\omega_{m}}{\omega_{\gamma}} (\tau-s) \right) \cos \left( \frac{\omega_{m}}{\omega_{\gamma}} s \right) ds = \frac{1}{2} \tau \sin \left( \frac{\omega_{m}}{\omega_{\gamma}} \tau \right).
\end{align*}
Such secular terms are also produced when there exists an $m =0,1,2,\dots$ such that $\omega_{m}/ \omega_{\gamma}:=b \in J_{\lambda}$.
In other words,
\begin{align*}
	\forall \lambda =0,1,2,\dots,~ \exists \text{~a set~} \mathcal{N}_{\lambda}:~\forall m \in  \mathcal{N}_{\lambda}, ~f_{2\lambda+1}^{(m)}(\tau) \text{~contains non-periodic terms}. 
\end{align*}
Maliborski and Rostworowski \cite{PhysRevLett111051102} were able to numerically cancel these secular terms by prescribing the initial data $(f_{2\lambda+1}^{(m)}(0),\dot{f}_{2\lambda+1}^{(m)}(0))$. To explain their approach, we take $f_{1}^{(\gamma)}(0)=1$ and $f_{2\lambda+1}^{(\gamma)}(0)=0$ for $\lambda=1,2,3,\dots$. First, they choose $\dot{f}_{2\lambda+1}^{(m)}(0)=0$ for all $\lambda=0,1,2,\dots$ and all $m=0,1,2,\dots$ to ensure that the source term $ S_{2\lambda+1}^{(m)}(\tau)$ is a series only of cosines. Then, they observed that the fixed index $\gamma $ belongs in $ \mathcal{N}_{\lambda}$ for all $\lambda=0,1,2,\dots$ and there is only one secular term in $f_{2\lambda+1}^{(\gamma)}(\tau)$ which can be removed by choosing the frequency shift $\omega_{\gamma,2(\lambda-1)}$. Furthermore, for all $m \in \mathcal{N}_{\lambda} \setminus \{\gamma\}$, there are some secular terms which cancel by the structure of the equations, some secular terms cancel by choosing some initial data but some initial data remain free variables at this stage. They choose these free variables together with $\omega_{\gamma,2\lambda}$ to cancel the secular terms in the $f_{2(\lambda+1)+1}(\tau)$.  For more details see \cite{PhysRevLett111051102}. However, there is no proof based on rigorous arguments ensuring that this procedure works for all $\lambda$.  
\subsection{Choice of the initial data}
For example, we fix $\gamma=0$ and choose
\begin{align*}
	\dot{f}_{2\lambda+1}^{(m)}(0) = 0,\quad \forall \lambda \geq 0, \quad \forall m \geq 0.
\end{align*}
First, we use the recurrence relation above and find periodic expressions for $p_{2}^{(m)}(\tau)$ and $q_{2}^{(m)}(\tau)$ due to the periodicity of $f_{1}^{(m)}(\tau)$ and $g_{1}^{(m)}(\tau)$. Second, we compute
\begin{align*}
	S_{3}^{(m)}(\tau)= A_{3}^{(m)} \cos(\tau)+B_{3}^{(m)} \cos(3\tau),
\end{align*}
for some sequences $\{A_{3}^{(m)}\}_{m=0,1,\dots}$ and $\{B_{3}^{(m)}\}_{m=0,1,\dots}$. Then, the equation
\begin{align*}
	\frac{\omega_{m}}{\omega_{0}}=\frac{3+2m}{3}=1+\frac{2}{3} m \in \{1,3 \}
\end{align*}
has two solutions $m \in \mathcal{N}_{3}:= \{0,3\}$. Based on the discussion above, we get two secular terms in the list $\{f_{3}^{(m)}(\tau): m=0,1,2,\dots\}$, one for $m=0$ and one for $m=3$. We see that the secular term for $m=0$ can be canceled by choosing the frequency shift $\omega_{0,2}$ whereas the secular term for $m=3$ cancels by the structure of the equations meaning $B_{3}^{(3)}=0$ without any choice of the initial data $\{f_{3}^{(m)}(0): m=0,1,2,\dots\}$. Hence, all these are free variables at this stage (i.e. for $\lambda=3$) and will be chosen later to cancel all the secular terms in $f_{5}^{(m)}(\tau)$. Specifically, we get
\begin{align*}
		& f_{3}^{(0)}(\tau)=
		\left( \frac{765}{128 \pi}+f_{3}^{(0)}(0) \right) \cos(\tau) 
	- \frac{765}{128 \pi} \cos(3\tau) 
	+ \left(\omega_{0,2} - \frac{153}{4 \pi} \right) \tau \sin(\tau), \\
		& f_{3}^{(1)}(\tau)=
		\frac{6183}{256 \pi}\sqrt{3} \cos(\tau) 
	+ \frac{765}{256 \pi}\sqrt{3} \cos(3\tau) 
	+ \left(f_{3}^{(1)}(0) - \frac{1737}{64 \pi}\sqrt{3} \right) \cos \left(\frac{5}{3}\tau \right), \\
	& f_{3}^{(2)}(\tau)=
		\frac{3717}{3200 \pi}\sqrt{\frac{3}{2}} \cos(\tau) 
	+ \frac{441}{128 \pi}\sqrt{\frac{3}{2}} \cos(3\tau) 
	+ \left(f_{3}^{(2)}(0) - \frac{7371}{1600 \pi}\sqrt{\frac{3}{2}} \right) \cos \left(\frac{7}{3}\tau \right),  
\end{align*}
and so forth. We choose 
\begin{align*}
\omega_{0,2}=\frac{153}{4\pi}
\end{align*}
to ensure the periodicity of $f_{3}^{(0)}(\tau)$. Once all secular terms in $f_{3}^{(m)}(\tau)$ are removed, the periodic expression for $f_{3}^{(m)}(\tau)$ implies a periodic expression also for $g_{3}^{(m)}(\tau)$. 
\subsection{Growth and decay of the Fourier constants}
In this section, we focus on the asymptotic behaviour of all the Fourier constants which appear using this approach. We shall use the notation
\begin{align*}
	\sum_{\pm} f(a\pm b \pm c) = f(a + b + c) + f(a + b - c) + f(a - b + c) + f(a - b - c),
\end{align*}
that is summation with respect to all possible combinations of plus and minus. Furthermore, expressions like $\omega_{i} \pm \omega_{j} \pm \omega_{m}$ stand not only for $\omega_{i} + \omega_{j} + \omega_{m}$ and $\omega_{i} - \omega_{j} - \omega_{m}$ but also for $\omega_{i} + \omega_{j} - \omega_{m}$ and $\omega_{i} - \omega_{j} + \omega_{m}$, that is considering all possible combinations of plus and minus. Specifically, we prove the following result the proof of which is based on the leading order terms (Remark \ref{lot}) together with the asymptotic behavior of the oscillatory integrals (Lemma \ref{OscillatoryIntegrals}), the orthogonality properties (Lemma \ref{ClosedformulasFore}) and the $L^{\infty}-$bounds (Lemma \ref{Linftyboundse}). 
\begin{prop}[Growth and decay estimates for the Fourier constants: Perturbation 1]\label{Result1}
	Let $N\in \mathbb{N}$. The following growth and decay estimates hold. 
\begin{center}
\begin{longtable}{ |l|l|l|}
\hline
\multicolumn{3}{ |c| }{Growth and decay estimates for $C_{ij}^{(m)}$ and $\overline{C}_{ij}^{(m)}$ as $i,j,m \longrightarrow + \infty$} \\
\hline
\quad  Constant  & $\exists ~~\omega_{i} \pm \omega_{j} \pm \omega_{m} \longarrownot\longrightarrow \infty $ &   \quad \quad$\forall~~ \omega_{i} \pm \omega_{j} \pm \omega_{m} \longrightarrow  \infty $  \\ \hline
\multirow{1}{*}{\quad\quad \(\displaystyle C_{ij}^{(m)} \)}
 &\quad \quad \quad$ \mathcal{O} \left( \omega_{i} \right)$
  &\quad \quad \(\displaystyle  \sum_{\pm}\mathcal{O} \left( \frac{1 }{ (\omega_{i} \pm \omega_{j} \pm \omega_{m})^{2}} \right) \)  \\ \hline 
\multirow{1}{*}{\quad\quad \(\displaystyle \overline{C}_{ij}^{(m)} \)}
 &\quad \quad \quad $\mathcal{O} \left( \omega_{i} \right)$
  &~~  \(\displaystyle \frac{4}{\sqrt{\pi}} + \sum_{\pm}\mathcal{O} \left( \frac{1}{(\omega_{i} \pm \omega_{j} \pm  \omega_{m} )^2} \right) \)  \\ \hline 
\end{longtable}
\end{center} 
\begin{center}
\begin{longtable}{ |l|l|l|}
\hline
\multicolumn{3}{ |c| }{Growth and decay estimates for $A_{ijl}^{(m)},B_{ijl}^{(m)},\overline{A}_{ijl}^{(m)}$ and $ \overline{B}_{ijl}^{(m)}$ as $i,j,l,m \longrightarrow + \infty$} \\
\hline
\quad  Constant  & $\exists ~~\omega_{i} \pm \omega_{j} \pm \omega_{l} \pm \omega_{m} \longarrownot\longrightarrow \infty $ & \quad \quad \quad $ \forall~~ \omega_{i} \pm \omega_{j} \pm \omega_{l}\pm \omega_{m} \longrightarrow  \infty $  \\ \hline
\multirow{1}{*}{ \(\displaystyle \overline{A}_{ijl}^{(m)},\overline{B}_{ijl}^{(m)} \)}
 & \quad \quad\quad\quad \(\displaystyle  \mathcal{O} \left(\omega_{l}\right) \)
  &\quad  \(\displaystyle  \sum_{\pm} \mathcal{O} \left( \frac{1 }{ (\omega_{i} \pm \omega_{j} \pm \omega_{l}\pm \omega_{m})^{N}} \right) \)  \\ \hline 
\multirow{1}{*}{ \(\displaystyle A_{ijl}^{(m)},B_{ijl}^{(m)} \)}
 & \quad \quad \quad\(\displaystyle   \mathcal{O} \left( \frac{\omega_{l}}{ \omega_{m}} \right)\)
  &\(\displaystyle ~~~\frac{1}{\omega_{m}} \sum_{\pm} \mathcal{O} \left( \frac{1}{\left( \omega_{i} \pm \omega_{j} \pm \omega_{l} \pm \omega_{m}   \right)^{N}} \right) \)  \\ \hline 
\end{longtable}
\end{center}	
\end{prop}
\begin{proof}
	First, observe that $\omega_{i} \pm \omega_{j} \pm \omega_{m}$ are all odd, namely
\begin{align*}
	\omega_{i} - \omega_{j} + \omega_{m} & = 2(i-j+m+1)+1, \quad 
	\omega_{i} - \omega_{j} - \omega_{m}  = 2(i-j-m-2)+1, \\
	\omega_{i} + \omega_{j} + \omega_{m} & = 2(i+j+m+4)+1,   \quad 
	\omega_{i} + \omega_{j} - \omega_{m}  = 2(i+j-m+1)+1.
\end{align*}
For large values of $i,j,m$ and in the case where all $\omega_{i} \pm \omega_{j} \pm \omega_{m} \longrightarrow \infty$, we obtain
	\begin{align*}
		C_{ij}^{(m)} & :=  \int_{0}^{\frac{\pi}{2}}  e_{i}(x) e_{j}(x) e_{m}(x) \tan^2(x) dx 
		 \simeq \left( \frac{2}{\sqrt{\pi}} \right)^3 \int_{0}^{\frac{\pi}{2}}  \frac{\sin(\omega_{i}x)\sin(\omega_{j}x)\sin(\omega_{m}x)}{\tan(x)} dx \\
		& =\frac{1}{4} \left( \frac{2}{\sqrt{\pi}} \right)^3 \Bigg(
		 \int_{0}^{\frac{\pi}{2}}  \frac{\sin((\omega_{i}-\omega_{j}+\omega_{m})x)}{\tan(x)} dx -  \int_{0}^{\frac{\pi}{2}}  \frac{\sin((\omega_{i}-\omega_{j}-\omega_{m})x)}{\tan(x)} dx \\
		&\quad \quad \quad\quad \quad\quad - \int_{0}^{\frac{\pi}{2}}  \frac{\sin((\omega_{i}+\omega_{j}+\omega_{m})x)}{\tan(x)} dx + 
		  \int_{0}^{\frac{\pi}{2}}  \frac{\sin((\omega_{i}+\omega_{j}-\omega_{m})x)}{\tan(x)} dx
		 \Bigg) \\
		 & =\frac{1}{4} \left( \frac{2}{\sqrt{\pi}} \right)^3 \left(
		\frac{\pi}{2}-\frac{\pi}{2}-\frac{\pi}{2}+\frac{\pi}{2} + \sum_{\pm} \mathcal{O} \left( \frac{1}{(\omega_{i} \pm \omega_{j} \pm \omega_{m})^2} \right)
		 \right) \\
		 & =\sum_{\pm} \mathcal{O} \left( \frac{1}{(\omega_{i} \pm \omega_{j} \pm \omega_{m})^2} \right). 
	\end{align*}
On the other hand, for large values of $i,j,m$ such that some $\omega_{i} \pm \omega_{j} \pm \omega_{m} \longarrownot\longrightarrow \infty$, Holder's inequality implies
\begin{align*}
	\left| C_{ij}^{(m)} \right| & =\left|  \int_{0}^{\frac{\pi}{2}}  e_{i}(x) e_{j}(x) e_{m}(x) \tan^2(x) dx\right| 
	 \leq 	
	\left\|
	e_{i}
	\right \|_{L^{\infty}\left[0,\frac{\pi}{2}\right]}
	\left \| 
	e_{j}\tan
	\right \|_{L^{2}\left[0,\frac{\pi}{2}\right]}
	\left \| 
	e_{m}\tan
	\right \|_{L^{2}\left[0,\frac{\pi}{2}\right]} 	\lesssim \omega_{i}.
\end{align*}
Similarly, for large values of $i,j,m$ and in the case where all $\omega_{i} \pm \omega_{j} \pm \omega_{m} \longrightarrow \infty$, we obtain
\begin{align*}
		\overline{C}_{ij}^{(m)} & :=  \int_{0}^{\frac{\pi}{2}}  e_{i}(x) \frac{e_{j}^{\prime}(x)}{\omega_{j}} \frac{e_{m}^{\prime}(x)}{\omega_{m}}\tan^2(x) dx  
		 \simeq \left( \frac{2}{\sqrt{\pi}} \right)^3 \int_{0}^{\frac{\pi}{2}}  \frac{\sin(\omega_{i}x)\cos(\omega_{j}x)\cos(\omega_{m}x)}{\tan(x)} dx \\
		& =\frac{1}{4} \left( \frac{2}{\sqrt{\pi}} \right)^3 \Bigg(
		 \int_{0}^{\frac{\pi}{2}}  \frac{\sin((\omega_{i}+\omega_{j}+\omega_{m})x)}{\tan(x)} dx +  \int_{0}^{\frac{\pi}{2}}  \frac{\sin((\omega_{i}+\omega_{j}-\omega_{m})x)}{\tan(x)} dx \\
		&\quad \quad \quad\quad \quad\quad + \int_{0}^{\frac{\pi}{2}}  \frac{\sin((\omega_{i}-\omega_{j}+\omega_{m})x)}{\tan(x)} dx + 
		  \int_{0}^{\frac{\pi}{2}}  \frac{\sin((\omega_{i}-\omega_{j}-\omega_{m})x)}{\tan(x)} dx
		 \Bigg) \\
		 & =\frac{1}{4} \left( \frac{2}{\sqrt{\pi}} \right)^3 \left(
		\frac{\pi}{2}+\frac{\pi}{2}+\frac{\pi}{2}+\frac{\pi}{2} + \sum_{\pm} \mathcal{O} \left( \frac{1}{(\omega_{i} \pm \omega_{j} \pm \omega_{m})^2} \right)
		 \right) \\
		 & =\frac{4}{\sqrt{\pi}} + \sum_{\pm} \mathcal{O} \left( \frac{1}{(\omega_{i} \pm \omega_{j} \pm \omega_{m})^2} \right). 
	\end{align*}	
On the other hand, for large values of $i,j,m$ such that some $\omega_{i} \pm \omega_{j} \pm \omega_{m} \longarrownot\longrightarrow \infty$, Holder's inequality implies
\begin{align*}
	\left|\overline{C}_{ij}^{(m)}\right| & =\left|  \int_{0}^{\frac{\pi}{2}}  e_{i}(x) \frac{e_{j}^{\prime}(x)}{\omega_{j}} \frac{e_{m}^{\prime}(x)}{\omega_{m}}\tan^2(x) dx \right|\\
		 &
	 \leq 	
	\left\|
	e_{i}
	\right \|_{L^{\infty}\left[0,\frac{\pi}{2}\right]}
	\left \| 
	\frac{e_{j}^{\prime}(x)}{\omega_{j}}\tan
	\right \|_{L^{2}\left[0,\frac{\pi}{2}\right]}
	\left \| 
	\frac{e_{m}^{\prime}(x)}{\omega_{m}}\tan
	\right \|_{L^{2}\left[0,\frac{\pi}{2}\right]} 	\lesssim \omega_{i}.
\end{align*}
Next, observe that $\omega_{i} \pm \omega_{j} \pm \omega_{l} \pm \omega_{m}$ are all even, 
\begin{align*}
	\omega_{i} - \omega_{j} + \omega_{l} - \omega_{m} & = 2 (i - j + l - m),\quad 
	\omega_{i} + \omega_{j} - \omega_{l} - \omega_{m}   = 2 (i + j - l - m), \\
	\omega_{i} - \omega_{j} - \omega_{l} - \omega_{m}& = 2 (-3 + i - j - l - m), \quad
	\omega_{i} + \omega_{j} + \omega_{l} - \omega_{m}   = 2 (3 + i + j + l - m),\\
	\omega_{i} - \omega_{j} - \omega_{l} + \omega_{m} & = 2 (i - j - l + m), \quad
	\omega_{i} + \omega_{j} + \omega_{l} + \omega_{m}   = 2 (6 + i + j + l + m), \\
	\omega_{i} - \omega_{j} + \omega_{l} + \omega_{m} & = 2 (3 + i - j + l + m), \quad
	\omega_{i} + \omega_{j} - \omega_{l} + \omega_{m}   = 2 (3 + i + j - l + m).
\end{align*}
Now, for large values of $i,j,l,m$ and in the case where all $\omega_{i} \pm \omega_{j} \pm \omega_{l} \pm \omega_{m} \longrightarrow \infty$, we obtain
\begin{align*}
	\overline{A}_{ijl}^{(m)} &:= \int_{0}^{\frac{\pi}{2}} \frac{e_{i}^{\prime}(x)}{\omega_{i}} \frac{e_{j}^{\prime}(x)}{\omega_{j}} e_{l}(x) \frac{e_{m}^{\prime}(x)}{\omega_{m}} \frac{\sin^3(x)}{\cos(x)}  dx \\
	& \simeq \left(\frac{2}{\sqrt{\pi}} \right)^4 \int_{0}^{\frac{\pi}{2}}  \cos^2(x) \frac{ \cos(\omega_{i}x)  \cos(\omega_{j}x) \sin(\omega_{l}x) \cos(\omega_{m}x) }{\tan(x)} dx \\
	& = \frac{1}{8}\left(\frac{2}{\sqrt{\pi}} \right)^4\Bigg( 
	\int_{0}^{\frac{\pi}{2}}  \cos^2(x) \frac{ \sin \left(  \left( \omega_{i} - \omega_{j} + \omega_{l} + \omega_{m}   \right) x \right) }{\tan(x)} dx \\
	& +  \int_{0}^{\frac{\pi}{2}}  \cos^2(x) \frac{ \sin \left(  \left( \omega_{i} - \omega_{j} + \omega_{l} - \omega_{m}   \right) x \right) }{\tan(x)} dx \\
	& - 	\int_{0}^{\frac{\pi}{2}}  \cos^2(x) \frac{  \sin \left(  \left( \omega_{i} - \omega_{j} - \omega_{l} + \omega_{m}   \right) x \right) }{\tan(x)} dx 
	-  \int_{0}^{\frac{\pi}{2}}  \cos^2(x) \frac{  \sin \left(  \left( \omega_{i} - \omega_{j} - \omega_{l} - \omega_{m}   \right) x \right) }{\tan(x)} dx \\
	& + 	\int_{0}^{\frac{\pi}{2}}  \cos^2(x) \frac{  \sin \left(  \left( \omega_{i} + \omega_{j} + \omega_{l} + \omega_{m}   \right) x \right) }{\tan(x)} dx 
	+  \int_{0}^{\frac{\pi}{2}}  \cos^2(x) \frac{  \sin \left(  \left( \omega_{i} + \omega_{j} + \omega_{l} - \omega_{m}   \right) x \right) }{\tan(x)} dx \\
	& - 	\int_{0}^{\frac{\pi}{2}}  \cos^2(x) \frac{  \sin \left(  \left( \omega_{i} + \omega_{j} - \omega_{l} + \omega_{m}   \right) x \right) }{\tan(x)} dx 
	-  \int_{0}^{\frac{\pi}{2}}  \cos^2(x) \frac{  \sin \left(  \left( \omega_{i} + \omega_{j} - \omega_{l} - \omega_{m}   \right) x \right) }{\tan(x)} dx
	 \Bigg) \\
&=\frac{1}{8} \left(\frac{2}{\sqrt{\pi}} \right)^4 \left( 4 \left(  \frac{\pi}{2} - \frac{\pi}{2} \right) + \sum_{\pm} \mathcal{O} \left( \frac{1}{\left( \omega_{i} \pm \omega_{j} \pm \omega_{l} \pm \omega_{m}   \right)^{N}}
	 \right)  \right) \\
		 &
	  = \sum_{\pm} \mathcal{O} \left( \frac{1}{\left( \omega_{i} \pm \omega_{j} \pm \omega_{l} \pm \omega_{m}   \right)^{N}} \right),	  
\end{align*}
whereas, for large values of $i,j,l,m$ such that some $\omega_{i} \pm \omega_{j} \pm \omega_{l} \pm \omega_{m} \longarrownot\longrightarrow \infty$,  Holder's inequality implies
\begin{align*}
	\left| \overline{A}_{ijl}^{(m)} \right| &= \left| \int_{0}^{\frac{\pi}{2}} \frac{e_{i}^{\prime}(x)}{\omega_{i}} \frac{e_{j}^{\prime}(x)}{\omega_{j}} e_{l}(x) \frac{e_{m}^{\prime}(x)}{\omega_{m}} \frac{\sin^3(x)}{\cos(x)}  dx\right| \\
	&=\left| \int_{0}^{\frac{\pi}{2}} \frac{e_{i}^{\prime}(x)}{\omega_{i}}\tan(x) \cdot \frac{e_{j}^{\prime}(x)}{\omega_{j}}\tan(x)\cdot e_{l}(x)\cos^2(x) \cdot \frac{e_{m}^{\prime}(x)}{\omega_{m}}\tan(x)   dx\right| \\
	&  \leq 	
	\left \| 
	\frac{e_{i}^{\prime}}{\omega_{i}}\tan
	\right \|_{L^{2}\left[0,\frac{\pi}{2}\right]}
	\left \| 
	\frac{e_{j}^{\prime}}{\omega_{j}}\tan
	\right \|_{L^{2}\left[0,\frac{\pi}{2}\right]}
	\left \| 
	e_{l} \cos^2
	\right\|_{L^{\infty}\left[0,\frac{\pi}{2}\right]}
	 \left \| 
	\frac{e_{m}^{\prime}}{\omega_{m}} \tan
	\right \|_{L^{\infty}\left[0,\frac{\pi}{2}\right]} 	\lesssim \omega_{l}.
\end{align*}
Furthermore, for large values of $i,j,l,m$ and in the case where all $\omega_{i} \pm \omega_{j} \pm \omega_{l}\pm \omega_{m} \longrightarrow \infty$, 
\begin{align*}
	\overline{B}_{ijl}^{(m)} &:= \int_{0}^{\frac{\pi}{2}} 
	e_{i}(x) e_{j}(x) e_{l}(x) \frac{e_{m}^{\prime}(x)}{\omega_{m}} \frac{\sin^3(x)}{\cos(x)}  dx\\
		 &
	 \simeq \left(\frac{2}{\sqrt{\pi}} \right)^4 \int_{0}^{\frac{\pi}{2}}  \cos^2(x) \frac{ \sin(\omega_{i}x)  \sin(\omega_{j}x) \sin(\omega_{l}x) \cos(\omega_{m}x) }{\tan(x)} dx \\
	& =\frac{1}{8} \left(\frac{2}{\sqrt{\pi}} \right)^4\Bigg( 
	\int_{0}^{\frac{\pi}{2}}  \cos^2(x) \frac{ \sin \left(  \left( \omega_{i} - \omega_{j} + \omega_{l} + \omega_{m}   \right) x \right) }{\tan(x)} dx \\
		 &
	+  \int_{0}^{\frac{\pi}{2}}  \cos^2(x) \frac{ \sin \left(  \left( \omega_{i} - \omega_{j} + \omega_{l} - \omega_{m}   \right) x \right) }{\tan(x)} dx \\
	& - 	\int_{0}^{\frac{\pi}{2}}  \cos^2(x) \frac{  \sin \left(  \left( \omega_{i} - \omega_{j} - \omega_{l} + \omega_{m}   \right) x \right) }{\tan(x)} dx 
	-  \int_{0}^{\frac{\pi}{2}}  \cos^2(x) \frac{  \sin \left(  \left( \omega_{i} - \omega_{j} - \omega_{l} - \omega_{m}   \right) x \right) }{\tan(x)} dx \\
	& - 	\int_{0}^{\frac{\pi}{2}}  \cos^2(x) \frac{  \sin \left(  \left( \omega_{i} + \omega_{j} + \omega_{l} + \omega_{m}   \right) x \right) }{\tan(x)} dx 
	-  \int_{0}^{\frac{\pi}{2}}  \cos^2(x) \frac{  \sin \left(  \left( \omega_{i} + \omega_{j} + \omega_{l} - \omega_{m}   \right) x \right) }{\tan(x)} dx \\
	& + 	\int_{0}^{\frac{\pi}{2}}  \cos^2(x) \frac{  \sin \left(  \left( \omega_{i} + \omega_{j} - \omega_{l} + \omega_{m}   \right) x \right) }{\tan(x)} dx 
	+  \int_{0}^{\frac{\pi}{2}}  \cos^2(x) \frac{  \sin \left(  \left( \omega_{i} + \omega_{j} - \omega_{l} - \omega_{m}   \right) x \right) }{\tan(x)} dx
	 \Bigg) \\
&=\frac{1}{8} \left(\frac{2}{\sqrt{\pi}} \right)^4 \left( 4 \left(  \frac{\pi}{2} - \frac{\pi}{2} \right) + \sum_{\pm} \mathcal{O} \left( \frac{1}{\left( \omega_{i} \pm \omega_{j} \pm \omega_{l} \pm \omega_{m}   \right)^{N}}
	 \right)  \right)\\
		 &
	  = \sum_{\pm} \mathcal{O} \left( \frac{1}{\left( \omega_{i} \pm \omega_{j} \pm \omega_{l} \pm \omega_{m}   \right)^{N}} \right),	  
\end{align*}
whereas, for large values of $i,j,l,m$ such that some $\omega_{i} \pm \omega_{j} \pm \omega_{l} \pm \omega_{m} \longarrownot\longrightarrow \infty$,  Holder's inequality implies
\begin{align*}
	\left| \overline{B}_{ijl}^{(m)} \right| &= \left| \int_{0}^{\frac{\pi}{2}} e_{i}(x) e_{j}(x)  e_{l}(x) \frac{e_{m}^{\prime}(x)}{\omega_{m}} \frac{\sin^3(x)}{\cos(x)}  dx\right| \\
	&=\left| \int_{0}^{\frac{\pi}{2}} e_{i}(x)\tan(x)  \cdot e_{j}(x) \tan(x)\cdot e_{l}(x)\cos^2(x) \cdot \frac{e_{m}^{\prime}(x)}{\omega_{m}}\tan(x)   dx\right| \\
	&  \leq 	
	\left \| 
	e_{i}\tan
	\right \|_{L^{2}\left[0,\frac{\pi}{2}\right]}
	\left \| 
	e_{j}\tan
	\right \|_{L^{2}\left[0,\frac{\pi}{2}\right]}
	\left \| 
	e_{l} \cos^2
	\right\|_{L^{\infty}\left[0,\frac{\pi}{2}\right]}
	 \left \| 
	\frac{e_{m}^{\prime}}{\omega_{m}} \tan
	\right \|_{L^{\infty}\left[0,\frac{\pi}{2}\right]} 	\lesssim \omega_{l}.
\end{align*}
Next, we use once more Remark \ref{lot} to compute
\begin{align*}
	& \int_{x}^{\frac{\pi}{2}} e_{m}(y) \cos(y) \sin(y) dy  \simeq \frac{2}{\sqrt{\pi}} \int_{x}^{\frac{\pi}{2}} \sin(\omega_{m}y) \cos^2(y) dy \\
	 &= \frac{2}{\sqrt{\pi}} \left(
	 \frac{\sin(\omega_{m}x)\sin(2x)}{\omega_{m}^2-4}
	 -\frac{2\cos(\omega_{m}x)}{\omega_{m}(\omega_{m}^2-4)}
	 + \frac{\omega_{m}\cos^2(x)\cos(\omega_{m}x)}{\omega_{m}^2-4}
	 \right) 
	  \simeq \frac{2}{\sqrt{\pi}} 
	 \frac{1}{\omega_{m}} \cos^2(x)\cos(\omega_{m}x),
\end{align*}
as $m \longrightarrow \infty$. Hence, for large values of $i,j,l,m$ and in the case where all $\omega_{i} \pm \omega_{j} \pm \omega_{l} \pm \omega_{m} \longrightarrow \infty$, 
\begin{align*}
	A_{ijl}^{(m)} &:= 
     \int_{0}^{\frac{\pi}{2}} \frac{ e_{i}^{\prime}(x)}{\omega_{i}} \frac{e_{j}^{\prime}(x)}{\omega_{j}} e_{l}(x) 
     \left( 
     \int_{x}^{\frac{\pi}{2}} e_{m}(y) \sin(y)\cos(y) dy
     \right) \tan^2(x)dx\\
	& \simeq \left(\frac{2}{\sqrt{\pi}} \right)^4 \frac{1}{\omega_{m}} \int_{0}^{\frac{\pi}{2}}  \cos^2(x) \frac{ \cos(\omega_{i}x)  \cos(\omega_{j}x) \sin(\omega_{l}x) \cos(\omega_{m}x) }{\tan(x)} dx \\
	& = \frac{1}{8}\left(\frac{2}{\sqrt{\pi}} \right)^4 \frac{1}{\omega_{m}} \Bigg( 
	\int_{0}^{\frac{\pi}{2}}  \cos^2(x) \frac{ \sin \left(  \left( \omega_{i} - \omega_{j} + \omega_{l} + \omega_{m}   \right) x \right) }{\tan(x)} dx \\
	 &+  \int_{0}^{\frac{\pi}{2}}  \cos^2(x) \frac{ \sin \left(  \left( \omega_{i} - \omega_{j} + \omega_{l} - \omega_{m}   \right) x \right) }{\tan(x)} dx \\
	& - 	\int_{0}^{\frac{\pi}{2}}  \cos^2(x) \frac{  \sin \left(  \left( \omega_{i} - \omega_{j} - \omega_{l} + \omega_{m}   \right) x \right) }{\tan(x)} dx 
	-  \int_{0}^{\frac{\pi}{2}}  \cos^2(x) \frac{  \sin \left(  \left( \omega_{i} - \omega_{j} - \omega_{l} - \omega_{m}   \right) x \right) }{\tan(x)} dx \\
	& + 	\int_{0}^{\frac{\pi}{2}}  \cos^2(x) \frac{  \sin \left(  \left( \omega_{i} + \omega_{j} + \omega_{l} + \omega_{m}   \right) x \right) }{\tan(x)} dx 
	+  \int_{0}^{\frac{\pi}{2}}  \cos^2(x) \frac{  \sin \left(  \left( \omega_{i} + \omega_{j} + \omega_{l} - \omega_{m}   \right) x \right) }{\tan(x)} dx \\
	& - 	\int_{0}^{\frac{\pi}{2}}  \cos^2(x) \frac{  \sin \left(  \left( \omega_{i} + \omega_{j} - \omega_{l} + \omega_{m}   \right) x \right) }{\tan(x)} dx 
	-  \int_{0}^{\frac{\pi}{2}}  \cos^2(x) \frac{  \sin \left(  \left( \omega_{i} + \omega_{j} - \omega_{l} - \omega_{m}   \right) x \right) }{\tan(x)} dx
	 \Bigg) \\
&=\frac{1}{8\omega_{m}} \left(\frac{2}{\sqrt{\pi}} \right)^4 \left( 4 \left(  \frac{\pi}{2} - \frac{\pi}{2} \right) + \sum_{\pm} \mathcal{O} \left( \frac{1}{\left( \omega_{i} \pm \omega_{j} \pm \omega_{l} \pm \omega_{m}   \right)^{N}}
	 \right)  \right) \\
	 & =\frac{1}{\omega_{m}} \sum_{\pm} \mathcal{O} \left( \frac{1}{\left( \omega_{i} \pm \omega_{j} \pm \omega_{l} \pm \omega_{m}   \right)^{N}} \right),	 
\end{align*}
whereas, for large values of $i,j,l,m$ such that some $\omega_{i} \pm \omega_{j}\pm \omega_{l} \pm \omega_{m} \longarrownot\longrightarrow \infty$,  Holder's inequality implies
\begin{align*}
	\left| A_{ijl}^{(m)} \right| &= \left| \int_{0}^{\frac{\pi}{2}} \frac{ e_{i}^{\prime}(x)}{\omega_{i}} \frac{e_{j}^{\prime}(x)}{\omega_{j}} e_{l}(x) 
     \left( 
     \int_{x}^{\frac{\pi}{2}} e_{m}(y) \sin(y)\cos(y) dy
     \right) \tan^2(x)dx\right| \\
	&=\left| \int_{0}^{\frac{\pi}{2}} \frac{ e_{i}^{\prime}(x)}{\omega_{i}} \tan(x) \frac{e_{j}^{\prime}(x)}{\omega_{j}} \tan(x) e_{l}(x) 
     \left( 
     \int_{x}^{\frac{\pi}{2}} e_{m}(y) \sin(y)\cos(y) dy
     \right) dx\right|\\
	&  \leq 
	 \left \| 
	\frac{e_{i}^{\prime}}{\omega_{i}} \tan
	\right \|_{L^{2}\left[0,\frac{\pi}{2}\right]}	
	 \left \| 
	\frac{e_{j}^{\prime}}{\omega_{j}} \tan
	\right \|_{L^{2}\left[0,\frac{\pi}{2}\right]}
	\left \| 
	e_{l} 
	\right\|_{L^{\infty}\left[0,\frac{\pi}{2}\right]}
	\left \| 
	 \int_{\cdot}^{\frac{\pi}{2}} e_{m}(y) \sin(y)\cos(y) dy
	\right \|_{L^{\infty}\left[0,\frac{\pi}{2}\right]} 	\lesssim \frac{\omega_{l}}{\omega_{m}}.
\end{align*}
Finally, for large values of $i,j,l,m$ and in the case where all $\omega_{i} \pm \omega_{j}\pm \omega_{l} \pm \omega_{m} \longrightarrow \infty$, 
\begin{align*}
	B_{ijl}^{(m)} &:= 
     \int_{0}^{\frac{\pi}{2}} e_{i}(x) e_{j}(x) e_{l}(x) 
     \left( 
     \int_{x}^{\frac{\pi}{2}} e_{m}(y) \sin(y)\cos(y) dy
     \right) \tan^2(x)dx\\
	& \simeq \left(\frac{2}{\sqrt{\pi}} \right)^4 \frac{1}{\omega_{m}} \int_{0}^{\frac{\pi}{2}}  \cos^2(x) \frac{ \sin(\omega_{i}x)  \sin(\omega_{j}x) \sin(\omega_{l}x) \cos(\omega_{m}x) }{\tan(x)} dx \\
	& = \frac{1}{8}\left(\frac{2}{\sqrt{\pi}} \right)^4 \frac{1}{\omega_{m}} \Bigg( 
	\int_{0}^{\frac{\pi}{2}}  \cos^2(x) \frac{ \sin \left(  \left( \omega_{i} - \omega_{j} + \omega_{l} + \omega_{m}   \right) x \right) }{\tan(x)} dx \\
	 &+  \int_{0}^{\frac{\pi}{2}}  \cos^2(x) \frac{ \sin \left(  \left( \omega_{i} - \omega_{j} + \omega_{l} - \omega_{m}   \right) x \right) }{\tan(x)} dx \\
	& - 	\int_{0}^{\frac{\pi}{2}}  \cos^2(x) \frac{  \sin \left(  \left( \omega_{i} - \omega_{j} - \omega_{l} + \omega_{m}   \right) x \right) }{\tan(x)} dx 
	-  \int_{0}^{\frac{\pi}{2}}  \cos^2(x) \frac{  \sin \left(  \left( \omega_{i} - \omega_{j} - \omega_{l} - \omega_{m}   \right) x \right) }{\tan(x)} dx \\
	& + 	\int_{0}^{\frac{\pi}{2}}  \cos^2(x) \frac{  \sin \left(  \left( \omega_{i} + \omega_{j} + \omega_{l} + \omega_{m}   \right) x \right) }{\tan(x)} dx 
	+  \int_{0}^{\frac{\pi}{2}}  \cos^2(x) \frac{  \sin \left(  \left( \omega_{i} + \omega_{j} + \omega_{l} - \omega_{m}   \right) x \right) }{\tan(x)} dx \\
	& - 	\int_{0}^{\frac{\pi}{2}}  \cos^2(x) \frac{  \sin \left(  \left( \omega_{i} + \omega_{j} - \omega_{l} + \omega_{m}   \right) x \right) }{\tan(x)} dx 
	-  \int_{0}^{\frac{\pi}{2}}  \cos^2(x) \frac{  \sin \left(  \left( \omega_{i} + \omega_{j} - \omega_{l} - \omega_{m}   \right) x \right) }{\tan(x)} dx
	 \Bigg) \\
&=\frac{1}{8\omega_{m}} \left(\frac{2}{\sqrt{\pi}} \right)^4 \left( 4 \left(  \frac{\pi}{2} - \frac{\pi}{2} \right) + \sum_{\pm} \mathcal{O} \left( \frac{1}{\left( \omega_{i} \pm \omega_{j} \pm \omega_{l} \pm \omega_{m}   \right)^{N}}
	 \right)  \right) \\
	 & =\frac{1}{\omega_{m}} \sum_{\pm} \mathcal{O} \left( \frac{1}{\left( \omega_{i} \pm \omega_{j} \pm \omega_{l} \pm \omega_{m}   \right)^{N}} \right),	 
\end{align*}
whereas, for large values of $i,j,l,m$ such that some $\omega_{i} \pm \omega_{j}\pm \omega_{l} \pm \omega_{m} \longarrownot\longrightarrow \infty$,  Holder's inequality implies
\begin{align*}
	\left| B_{ijl}^{(m)} \right| &= \left| \int_{0}^{\frac{\pi}{2}} e_{i}(x) e_{j}(x) e_{l}(x) 
     \left( 
     \int_{x}^{\frac{\pi}{2}} e_{m}(y) \sin(y)\cos(y) dy
     \right) \tan^2(x)dx\right| \\
	&=\left| \int_{0}^{\frac{\pi}{2}} e_{i}(x) \tan(x) e_{j}(x) \tan(x) e_{l}(x) 
     \left( 
     \int_{x}^{\frac{\pi}{2}} e_{m}(y) \sin(y)\cos(y) dy
     \right) dx\right|\\
	&  \leq 
	 \left \| 
	e_{i} \tan
	\right \|_{L^{2}\left[0,\frac{\pi}{2}\right]}	
	 \left \| 
	e_{j}  \tan
	\right \|_{L^{2}\left[0,\frac{\pi}{2}\right]}
	\left \| 
	e_{l} 
	\right\|_{L^{\infty}\left[0,\frac{\pi}{2}\right]}
	\left \| 
	 \int_{\cdot}^{\frac{\pi}{2}} e_{m}(y) \sin(y)\cos(y) dy
	\right \|_{L^{\infty}\left[0,\frac{\pi}{2}\right]} 	\lesssim \frac{\omega_{l}}{\omega_{m}},
\end{align*}
which completes the proof.
\end{proof}
\section{Perturbations 2: Finite sum with an error term}\label{SectionPerturbation2}
Since the series \eqref{series1}-\eqref{series4} may not converge, we can still assume that $(\Phi,\Pi,A,\delta)$ are all close to the AdS solution $(0,0,1,0)$ and expand $(\Phi,\Pi,A,\delta)$ using \eqref{finitesum1}-\eqref{finitesum5}. First, we solve at the linear level and obtain the periodic expressions
\begin{align*}
\begin{dcases}
    \Phi_{1}(\tau,x):=  \cos(\tau) e^{\prime}_{\gamma}(x), \\
    \Pi_{1}(\tau,x) := - \omega_{\gamma} \sin(\tau) e_{\gamma}(x),
     \\ 
	A_{2}(\tau,x) :=
	\cos^2(\tau) \Gamma_{1}(x)+
	\sin^2(\tau) \Gamma_{2}(x),\\
	 \delta_{2}(\tau,x):=
	-\cos^2(\tau) \Gamma_{3}(x) -
	\sin^2(\tau) \Gamma_{4}(x),
\end{dcases}
\end{align*}
where
\begin{align*}
& \Gamma_{1}(x):=\frac{\cos^3(x)}{\sin(x)}\int_{0}^{x}   \left( e^{\prime}_{0}(y)\right)^2 \tan^2(y) dy, \quad
 \Gamma_{2}(x):= \omega_{0}^2
	\frac{\cos^3(x)}{\sin(x)}\int_{0}^{x}  \left( e_{0}(y)\right)^2 \tan^2(y) dy, \\
&    \Gamma_{3}(x):= \int_{0}^{x}  \left( e^{\prime}_{0}(y)\right)^2 \sin(y)\cos(y) dy, \quad 
 \Gamma_{4}(x):=\omega_{0}^2
 \int_{0}^{x}  \left( e_{0}(y)\right)^2 \sin(y)\cos(y) dy.  
\end{align*}
For simplicity, we have fixed $\gamma = 0$ and use Lemma \ref{ClosedformulasFore} we compute
\begin{align*}
& \Gamma_{1}(x) = \frac{3 \cos^2(x)}{2 \pi \tan(x)} \Big( 12 x -3\sin(2x)-3\sin(4x)+\sin(6x)  \Big), \\
&	\Gamma_{2}(x)= - \frac{3 \cos^2(x)}{2 \pi \tan(x)} \Big( -12 x -3\sin(2x)+3\sin(4x)+\sin(6x)  \Big), \\
&    \Gamma_{3}(x) = \frac{3 \sin^4(x)}{2\pi} \Big( 25 +20\cos(x) +3\cos(4x) \Big), \\
&    \Gamma_{4}(x) = - \frac{9}{32\pi} \Big( 
-93 +56 \cos(2x) + 28 \cos(4x) +8 \cos(6x) +\cos(8x)
\Big).  
\end{align*}
For future reference, observe that
\begin{align}\label{LinftyGamma}
 	\left \| \Gamma_{a} \right \|_{L^{\infty}\left[0,\frac{\pi}{2} \right]} \lesssim 1, \quad \forall a \in \{1,2,3,4 \}.
\end{align}
Second, we get the following non-linear system for the error terms $(\Psi,\Sigma,B,\Theta)$,
\begin{align*}
	(\omega_{0} + \epsilon^{2} \theta_{0}+\epsilon^{4} \eta_{0})\partial_{\tau}\Psi  & =  ( \theta_{0}+\epsilon^{2} \eta_{0}) \sin(\tau) e_{0}^{\prime}(x)  + \partial_{x} \Bigg( \Sigma +\omega_{0}\sin(\tau)e_{0}(x)(A_{2}+\delta_{2})   \Bigg) \\
	& + \epsilon^{2} \partial_{x} \Bigg( -\Sigma(A_{2}+\delta_{2}) -\omega_{0} \sin(\tau) e_{0}(x) ( -B + \Theta + A_{2}\delta_{2} )  \Bigg) \\
	& + \epsilon^{4} \partial_{x} \Bigg( \Sigma (-B+\Theta+A_{2} \delta_{2}) -\omega_{0}\sin(\tau)e_{0} (-\Theta A_{2} +B \delta_{2} ) \Bigg) \\
	& + \epsilon^{6} \partial_{x} \Bigg( \Sigma (-\Theta A_{2} +B \delta_{2}) + \omega_{0}\sin(\tau)e_{0} B \Theta \Bigg) \\
	& - \epsilon^{8} \partial_{x} \Big( \Sigma \Theta B \Big),
\end{align*}
\begin{align*}
(\omega_{0} + \epsilon^{2}& \theta_{0}+\epsilon^{4} \eta_{0})\partial_{\tau}\Sigma  =   \omega_{0} (\theta_{0}+\epsilon^{2} \eta_{0}) \cos(\tau) e_{0} \\
&+ \Bigg( \frac{4}{\sin(2x)} \left( \Psi - \cos(\tau) e_{0}^{\prime}(A_{2}+\delta_{2}) \right) -\cos(\tau) \left( e_{0}^{\prime} (A_{2}+\delta_{2}) \right)^{\prime} + \partial_{x}\Psi \Bigg)  \\
	& + \epsilon^{2} \Bigg( \frac{4}{\sin(2x)} \left( -\Psi (A_{2}+\delta_{2}) +\cos(\tau) e_{0}^{\prime} (-B+\Theta + A_{2} \delta_{2}) \right) -\left( \Psi (A_{2}+\delta_{2}) \right)^{\prime} \\
	& + \cos(\tau) \left( e_{0}^{\prime} (-B+\Theta+A_{2}\delta_{2}) \right)^{\prime} \Bigg) \\
	& + \epsilon^{4} \Bigg( \frac{4}{\sin(2x)} \left( \Psi (-B+\Theta +A_{2}\delta_{2}) + \cos(\tau) e_{0}^{\prime} (-A_{2}\Theta + \delta_{2}B) \right)  \\
	& + \cos(\tau) \left( e_{0}^{\prime}(\delta_{2}B-A_{2}\Theta) \right)^{\prime} + \left( \Psi(-B+\Theta+A_{2}\delta_{2}) \right)^{\prime}  \Bigg) \\
	& + \epsilon^{6} \Bigg( \frac{4}{\sin(2x)} \left( \Psi (-\Theta A_{2} + \delta_{2}B) - \cos(\tau) e_{0}^{\prime} B \Theta \right)  - \cos(\tau) \left( e_{0}^{\prime} B \Theta \right)^{\prime} + \left(\delta_{2} \Psi B \right)^{\prime}   - \left(A_{2} \Psi \Theta \right)^{\prime}  \Bigg) \\
	& - \epsilon^{8} \Bigg(  \frac{4}{\sin(2x)} B \Theta \Psi + \left( B \Theta \Psi  \right)^{\prime}  \Bigg),
\end{align*}
\begin{align*}
	\partial_{x} B & + \frac{1+2 \sin^2(x)}{\sin(x)\cos(x)} B  =  \frac{\sin(2x)}{2} \Bigg( 2 \cos(\tau) e_{0}^{\prime} \Psi - 2 \omega_{0} \sin(\tau) e_{0} \Sigma - A_{2} \cos^2(\tau) (e_{0}^{\prime})^2 - \omega_{0}^2 A_{2} \sin^2(\tau) e_{0}^2 \Bigg)  \\
	& +\frac{\sin(2x)}{2} \epsilon^2 \Bigg( \Psi^2 + \Sigma^2 - 2 A_{2} \cos(\tau) e_{0}^{\prime} \Psi + 2 \omega_{0} A_{2} \sin(\tau)e_{0}  \Sigma - \cos^2(\tau)  (e_{0}^{\prime})^2 B - \omega_{0}^2 \sin^2(\tau) (e_{0} )^2 B \Bigg) \\
	& +\frac{\sin(2x)}{2} \epsilon^4 \Bigg( -A_{2} \Psi^2  -A_{2} \Sigma^2 -2 \cos(\tau) e_{0}^{\prime} B \Psi + 2 \omega_{0} \sin(\tau) e_{0}  B \Sigma \Bigg) \\
	& -\frac{\sin(2x)}{2} \epsilon^6 \Big( B \left( \Psi^2 + \Sigma^2 \right) \Big), 
\end{align*}
\begin{align*}
	\partial_{x}\Theta & = \frac{\sin(2x)}{2} \Bigg( 2 \cos(\tau) e_{0}^{\prime} \Psi - 2 \omega_{0}\sin(\tau) e_{0} \Sigma - \delta_{2} \cos^{2}(\tau) (e_{0}^{\prime})^2 - \omega_{0}^2 \delta_{2} \sin^{2} (\tau) e_{0}^2 \Bigg) \\
	& + \frac{\sin(2x)}{2} \epsilon^{2} \Bigg( \Psi^2+\Sigma^2 - 2\delta_{2} \cos(\tau) e_{0}^{\prime} \Psi + 2 \omega_{0} \sin(\tau) e_{0} \delta_{2} \Sigma + \cos^2(\tau) (e_{0}^{\prime})^2 \Theta + \omega_{0}^2 \sin^2(\tau) e_{0}^2   \Bigg) \\
	& + \frac{\sin(2x)}{2} \epsilon^{4} \Bigg( -\delta_{2} \Psi^2 - \delta_{2} \Sigma^2 + 2 \cos(\tau ) e_{0}^{\prime} \Theta \Psi - 2 \omega_{0} \sin(\tau) e_{0} \Theta \Sigma \Bigg) \\
	& +\frac{\sin(2x)}{2} \epsilon^{6} \Big( \Theta \left( \Psi^2 + \Sigma^2 \right) \Big).
\end{align*}
\subsection{Definition of the Fourier constants}\label{SectionPerturbation2DefinitionFourier}
As before, we expand the error terms $(\Psi, \Sigma, B, \Theta)$ in terms of the eigenfunctions of the linearized operator as follows
\begin{align*}
	\Psi(\tau,x) = \sum_{i=0}^{\infty} \psi_{i}(\tau)  \frac{e_{i}^{\prime}(x)}{\omega_{i}}, \quad  \Sigma(\tau,x) = \sum_{i=0}^{\infty} \sigma_{i}(\tau) e_{i}(x), \\
    B(\tau,x) = \sum_{i=0}^{\infty} b _{i}(\tau) e_{i}(x),\quad \Theta(\tau,x) = \sum_{i=0}^{\infty} \xi _{i}(\tau) e_{i}(x).
\end{align*}
After substituting these expressions into the equations above, we take the inner product $(\cdot | e_{j}^{\prime})$ (for the equation for $\Psi$) and $(\cdot | e_{j})$ (for the equations for $\Sigma,B$ and $\Theta$) in both sides. A long but straightforward  computation yields that all the interactions with respect the space variable $x$ are included in the following Fourier constants:
\begin{align*}
	\mathbb{C}_{abij}& := \int_{0}^{\frac{\pi}{2}} \left( \Gamma_{a}(x) - \Gamma_{b}(x) \right) e_{i}(x) e_{j}(x) \tan^2(x) dx,
	 \\
	\mathbb{D}_{abij}& := \int_{0}^{\frac{\pi}{2}}  \Gamma_{a}(x)  \Gamma_{b}(x) e_{i}(x) e_{j}(x) \tan^2(x) dx, \\
	\mathbb{E}_{1423ij}& := \int_{0}^{\frac{\pi}{2}} \left( \Gamma_{1}(x)\Gamma_{4}(x) - \Gamma_{2}(x)\Gamma_{3}(x) \right) e_{i}(x) e_{j}(x) \tan^2(x) dx \\
	\mathbb{G}_{jki}& :=  \int_{0}^{\frac{\pi}{2}}  e_{i}(x) e_{j}(x) e_{k}(x) \tan^2(x) dx, \\
\mathbb{F}_{ajki}& := \int_{0}^{\frac{\pi}{2}} \Gamma_{a}(x)  e_{i}(x) e_{j}(x) e_{k}(x) \tan^2(x) dx, \\
	\mathbb{H}_{ijkl}& :=  \int_{0}^{\frac{\pi}{2}}   e_{i}(x) e_{j}(x) e_{k}(x) e_{l}(x) \tan^2(x) dx,
\end{align*}
for the equation for $\Psi$, 
\begin{align*}
\overline{\mathbb{C}}_{abij}& := \int_{0}^{\frac{\pi}{2}} \left( \Gamma_{a}(x) - \Gamma_{b}(x) \right) \frac{e_{i}^{\prime}(x)}{\omega_{i}} \frac{e_{j}^{\prime}(x)}{\omega_{j}} \tan^2(x) dx, \\
\overline{\mathbb{D}}_{abij} &:= \int_{0}^{\frac{\pi}{2}}  \Gamma_{a}(x)  \Gamma_{b}(x) \frac{ e_{i}^{\prime}(x)}{\omega_{i}} \frac{e_{j}^{\prime}(x)}{\omega_{j}} \tan^2(x) dx ,\\
	\overline{\mathbb{E}}_{1423ij} &:= \int_{0}^{\frac{\pi}{2}} \left( \Gamma_{1}(x)\Gamma_{4}(x) - \Gamma_{2}(x)\Gamma_{3}(x) \right)  \frac{e_{i}^{\prime}(x)}{\omega_{i}} \frac{e_{j}^{\prime}(x)}{\omega_{j}} \tan^2(x) dx ,\\
	\overline{\mathbb{G}}_{kji} &:= \int_{0}^{\frac{\pi}{2}}  e_{j}(x) \frac{ e_{i}^{\prime}(x)}{\omega_{i}} \frac{e_{k}^{\prime}(x)}{\omega_{k}} \tan^2(x) dx ,\\
\overline{\mathbb{F}}_{akji} &:= \int_{0}^{\frac{\pi}{2}} \Gamma_{a}(x) e_{j}(x)  \frac{e_{i}^{\prime}(x)}{\omega_{i}} \frac{e_{k}^{\prime}(x)}{\omega_{k}} \tan^2(x) dx ,\\
	\overline{\mathbb{H}}_{kjli} &:= \int_{0}^{\frac{\pi}{2}} e_{l}(x) e_{j}(x) \frac{ e_{i}^{\prime}(x)}{\omega_{i}} \frac{e_{k}^{\prime}(x)}{\omega_{k}} \tan^2(x) dx, 
\end{align*}
for the equation for $\Sigma$, 
\begin{align*}
	\mathbb{J}_{jki} & :=\int_{0}^{\frac{\pi}{2}} \frac{e_{j}^{\prime}(x)}{\omega_{j}} \frac{e_{k}^{\prime}(x)}{\omega_{k}} \frac{e_{i}^{\prime}(x)}{\omega_{i}}  \frac{\sin^3(x)}{\cos(x)} dx ,\\
	\mathbb{I}_{jki} & :=\int_{0}^{\frac{\pi}{2}} e_{j}(x) e_{k}(x) \frac{e_{i}^{\prime}(x)}{\omega_{i}}  \frac{\sin^3(x)}{\cos(x)} dx ,\\
	\mathbb{P}_{ajki} & :=\int_{0}^{\frac{\pi}{2}} \Gamma_{a}(x) \frac{e_{j}^{\prime}(x)}{\omega_{j}} \frac{e_{k}^{\prime}(x)}{\omega_{k}} \frac{e_{i}^{\prime}(x)}{\omega_{i}}  \frac{\sin^3(x)}{\cos(x)} dx ,\\
	\mathbb{Q}_{ajki} & :=\int_{0}^{\frac{\pi}{2}} \Gamma_{a}(x) e_{j}(x) e_{k}(x) \frac{e_{i}^{\prime}(x)}{\omega_{i}}  \frac{\sin^3(x)}{\cos(x)} dx ,\\
		\mathbb{R}_{klji} & :=\int_{0}^{\frac{\pi}{2}} e_{j}(x) \frac{e_{k}^{\prime}(x)}{\omega_{k}}  \frac{ e_{l}^{\prime}(x)}{\omega_{l}} \frac{e_{i}^{\prime}(x)}{\omega_{i}} \frac{\sin^3(x)}{\cos(x)} dx ,\\
			\mathbb{S}_{jkli} & :=\int_{0}^{\frac{\pi}{2}} e_{j}(x) e_{k}(x) e_{l}(x) \frac{e_{i}^{\prime}(x)}{\omega_{i}} \frac{\sin^3(x)}{\cos(x)} dx,
\end{align*}
for the equation for $\Theta$, and finally
\begin{align*}
	\overline{\mathbb{J}}_{jki} & :=\int_{0}^{\frac{\pi}{2}} \frac{e_{j}^{\prime}(x)}{\omega_{j}} \frac{e_{k}^{\prime}(x)}{\omega_{k}} \left( \int_{x}^{\frac{\pi}{2}} e_{i}(y) \sin(y) \cos(y) dy \right) \tan^2(x) dx ,\\
	\overline{\mathbb{I}}_{jki} & :=\int_{0}^{\frac{\pi}{2}} e_{j}(x) e_{k}(x) \left( \int_{x}^{\frac{\pi}{2}} e_{i}(y) \sin(y) \cos(y) dy \right) \tan^2(x) dx ,\\
	 \overline{\mathbb{P}}_{ajki} & :=\int_{0}^{\frac{\pi}{2}} 
	\Gamma_{a} (x) \frac{ e_{j}^{\prime}(x)}{\omega_{j}}  \frac{e_{k}^{\prime}(x)}{\omega_{k}}   \left( \int_{x}^{\frac{\pi}{2}} e_{i}(y) \sin(y) \cos(y) dy \right) \tan^2(x) dx ,\\
	 \overline{\mathbb{Q}}_{ajki} & :=\int_{0}^{\frac{\pi}{2}} 
	\Gamma_{a} (x)  e_{j}(x)   e_{k}(x)  \left( \int_{x}^{\frac{\pi}{2}} e_{i}(y) \sin(y) \cos(y) dy \right) \tan^2(x) dx ,\\
		 \overline{\mathbb{R}}_{klji} & :=\int_{0}^{\frac{\pi}{2}} 
	 e_{j}(x)  \frac{ e_{k}^{\prime} (x)}{\omega_{k}} \frac{e_{l}^{\prime}  (x)}{\omega_{l}}  \left( \int_{x}^{\frac{\pi}{2}} e_{i}(y) \sin(y) \cos(y) dy \right) \tan^2(x) dx ,\\
	  \overline{\mathbb{S}}_{jkli} & :=\int_{0}^{\frac{\pi}{2}} 
	 e_{j}(x)   e_{k} (x) e_{l}  (x)  \left( \int_{x}^{\frac{\pi}{2}} e_{i}(y) \sin(y) \cos(y) dy \right) \tan^2(x) dx,
\end{align*}
for the equation for $B$. In addition, the non-linear system for the error terms boils down to 
\begin{align*}
	\frac{d}{d \tau} \psi_{i}(\tau) - \frac{\omega_{i}}{\omega_{0}+\theta_{0} \epsilon^2 +\eta_{0} \epsilon^4} \sigma_{i}(\tau) & = \frac{1}{\omega_{0}+\theta_{0} \epsilon^2+\eta_{0} \epsilon^4} \frac{ \prescript{\psi}{}{F}_{i}(\tau) }{\omega_{i}}\\
	& + \frac{1}{\omega_{0}+\theta_{0} \epsilon^2++\eta_{0} \epsilon^4}  \frac{\prescript{\psi}{}{\mathfrak{N}}_{i}(\psi(\tau),\sigma(\tau),b(\tau),\xi(\tau))}{\omega_{i}}, \\
	\frac{d}{d \tau} \sigma_{i}(\tau) + \frac{\omega_{i}}{\omega_{0}+\theta_{0} \epsilon^2+\eta_{0} \epsilon^4} \psi_{i}(\tau) & = \frac{1}{\omega_{0}+\theta_{0} \epsilon^2+\eta_{0} \epsilon^4}  \prescript{\sigma}{}{F}_{i}(\tau) \\
	& +  \frac{1}{\omega_{0}+\theta_{0} \epsilon^2+\eta_{0} \epsilon^4}\prescript{\sigma}{}{\mathfrak{N}}_{i}(\psi(\tau),\sigma(\tau),b(\tau),\xi(\tau)), 
\end{align*}
subject to the constraints
\begin{align*}
		b_{i} (\tau)& =  \prescript{b}{}{S}_{i}(\tau) + \prescript{b}{}{\mathfrak{N}}_{i}(\psi(\tau),\sigma(\tau),b(\tau),\xi(\tau)) ,\\
	\xi_{i} (\tau)& = \prescript{\xi}{}{S}_{i}(\tau) + \prescript{\xi}{}{\mathfrak{N}}_{i}(\psi(\tau),\sigma(\tau),b(\tau),\xi(\tau)).
\end{align*}
Here, the source terms are given explicitly in terms of the Fourier constants by 
\begin{align*}	 
	\frac{\prescript{\psi}{}{F}_{i}(\tau)}{\omega_{i}} & = \omega_{i} \Big( \frac{\omega_{0}^2\theta_{0}\delta_{0,i}}{\omega_{i}^2} \sin(\tau) + \omega_{0}  \mathbb{C}_{13 0 i} \sin(\tau) \cos^2(\tau) + \omega_{0} \mathbb{C}_{240 i} \sin^3(\tau)  \Big) \\
	&+ \omega_{i} \epsilon^2 \Big( \frac{\omega_{0}^2\eta_{0}\delta_{0,i}}{\omega_{i}^2} \sin(\tau)+ \omega_{0}  \mathbb{D}_{130 i }\sin(\tau) \cos^{4}(\tau) + \omega_{0}  \mathbb{D}_{24 0 i} \sin^{5} (\tau) + \omega_{0} \mathbb{E}_{1423 0 i} \sin^3(\tau) \cos^2(\tau) \Big),\\
	\prescript{\sigma}{}{F}_{i}(\tau) & = \omega_{0} \omega_{i}\Big( \frac{\theta_{0}\delta_{0,i}}{\omega_{i}} \cos(\tau) +  \overline{\mathbb{C}}_{13 0 i} \cos^3(\tau) + \overline{\mathbb{C}}_{240 i} \cos(\tau) \sin^2(\tau)  \Big) \\
	&+\omega_{0} \omega_{i} \epsilon^2 \Big( \frac{\eta_{0}\delta_{0,i}}{\omega_{i}} \cos(\tau)+  \overline{\mathbb{D}}_{130 i }\cos^{5}(\tau) + \overline{\mathbb{D}}_{24 0 i} \cos^{4}(\tau) \sin^{4} (\tau) + \overline{\mathbb{E}}_{1423 0 i} \cos^3(\tau) \sin^2(\tau) \Big),\\	
\prescript{b}{}{S}_{i}(\tau) & = \omega_{0}^2 \Big( - \cos^4(\tau)  \overline{\mathbb{P}}_{100i} - \sin^2(\tau) \cos^2(\tau) \overline{\mathbb{P}}_{200i} - \sin^2(\tau) \cos^2(\tau) \overline{\mathbb{Q}}_{100i} - \sin^4(\tau) \overline{\mathbb{Q}}_{200i} \Big), \\
	\prescript{\xi}{}{S}_{i}(\tau) & = \frac{\omega_{0}^2}{\omega_{i}} \Big( \cos^4(\tau) \mathbb{P}_{300i} + \cos^2(\tau)  \sin^2(\tau) \mathbb{P}_{400i} + \sin^4(\tau) \mathbb{Q}_{400i} + \cos^2(\tau)  \sin^2(\tau) \mathbb{Q}_{300i} \Big), 
\end{align*}
where $\delta_{0,i}$ stands for the Kronecker's delta, whereas the linear and non-linear terms are given by
\begin{align*}	
	&\prescript{\psi}{}{\mathfrak{N}}_{i}(\psi(\tau),\sigma(\tau),b(\tau),\xi(\tau)) = \epsilon^2 \Bigg( -\cos^2(\tau) \sum _{j=0}^{\infty} \omega_{i}^2 \mathbb{C}_{13ji} \sigma_{j}(\tau) -\sin^2(\tau) \sum _{j=0}^{\infty} \omega_{i}^2 \mathbb{ C}_{24ji} \sigma_{j}(\tau) \\
	& -\omega_{0} \sin(\tau) \sum _{j=0}^{\infty} \omega_{i}^2 \mathbb{G}_{0 ji} ( \xi_{j}(\tau)-b_{j}(\tau))    \Bigg) \\
	& + \epsilon^4 \Bigg( -\cos^{4}(\tau) \sum _{j=0}^{\infty} \omega_{i}^2 \mathbb{D}_{13ji} \sigma_{j}(\tau)
	 - \cos^{2}(\tau) \sin^{2}(\tau) \sum _{j=0}^{\infty} \omega_{i}^2 \mathbb{E}_{1423ji} \sigma_{j}(\tau)+ \sin^{4}(\tau) \sum _{j=0}^{\infty} \omega_{i}^2 \mathbb{D}_{24ji} \sigma_{j}(\tau) \\
	 & + \omega_{0} \sin(\tau) \cos^{2}(\tau) \sum _{j=0}^{\infty} \omega_{i}^2 \mathbb{F}_{10 ji} \xi_{j}(\tau) + \omega_{0} \sin^{3}(\tau) \sum _{j=0}^{\infty} \omega_{i}^2 \mathbb{F}_{20 ji} \xi_{j}(\tau) + \omega_{0} \sin(\tau) \cos^{2}(\tau) \sum _{j=0}^{\infty} \omega_{i}^2 \mathbb{F}_{3 0 ji} b_{j}(\tau) \\
	& + \omega_{0} \sin^{3}(\tau) \sum _{j=0}^{\infty} \omega_{i}^2 \mathbb{F}_{40 ji} b_{j}(\tau) + \sum _{j,k=0}^{\infty} \omega_{i}^2 \mathbb{G}_{jki} \sigma_{j}(\tau)(-b_{k}(\tau)+\xi_{k}(\tau))  \Bigg) \\
	& + \epsilon^6 \Bigg( \omega_{0} \sin(\tau) \sum _{j,k=0}^{\infty} \omega_{i}^2 \mathbb{H}_{0 jki} b_{j}(\tau) \xi_{k}(\tau) - \sin^{2}(\tau) \sum _{j,k=0}^{\infty} \omega_{i}^2 \mathbb{F}_{2jki} \xi_{j}(\tau)\sigma_{k}(\tau)  - \sin^{2}(\tau) \sum _{j,k=0}^{\infty} \omega_{i}^2 \mathbb{F}_{4jki} \sigma_{j}(\tau)b_{k}(\tau) \\
	& - \cos^{2}(\tau) \sum _{j,k=0}^{\infty} \omega_{i}^2 \mathbb{F}_{1jki} \sigma_{j}(\tau)\xi_{k}(\tau)  - \cos^{2}(\tau) \sum _{j,k=0}^{\infty} \omega_{i}^2 \mathbb{F}_{3jki} \sigma_{j}(\tau)b_{k}(\tau) \Bigg) \\
	& + \epsilon^8 \Bigg( - \sum_{j,k,l=0}^{\infty} \omega_{i}^2 \mathbb{H}_{jkli} \xi_{j}(\tau) \sigma_{k}(\tau) b_{l}(\tau) \Bigg),
\end{align*}
\begin{align*}	
	& \prescript{\sigma}{}{\mathfrak{N}}_{i}(\psi(\tau),\sigma(\tau),b(\tau),\xi(\tau)) = \epsilon^2 \Bigg( \cos^2(\tau) \sum _{j=0}^{\infty} \omega_{i}  \overline{\mathbb{C}}_{13ji} \psi_{j}(\tau) +\sin^2(\tau) \sum _{j=0}^{\infty}  \omega_{i} \overline{\mathbb{C}}_{24ji} \psi_{j}(\tau) \\
	&  - \cos(\tau) \sum _{j=0}^{\infty} \omega_{0} \omega_{i} \overline{\mathbb{G}}_{0 ji} (\xi_{j}(\tau)-b_{j}(\tau))     \Bigg) \\
	& + \epsilon^4 \Bigg( \cos^{4}(\tau) \sum _{j=0}^{\infty} \omega_{i} \overline{\mathbb{D}}_{13ji} \psi_{j}(\tau) + \cos^{2}(\tau) \sin^{2}(\tau) \sum _{j=0}^{\infty} \omega_{i}  \overline{\mathbb{E}}_{1423ji} \psi_{j}(\tau) + \sin^{4}(\tau) \sum _{j=0}^{\infty} \omega_{i}  \overline{\mathbb{D}}_{24ji} \psi_{j}(\tau) \\
	& + \sin^{2}(\tau) \cos(\tau) \sum _{j=0}^{\infty} \omega_{0} \omega_{i}\overline{\mathbb{F}}_{20 ji} \xi_{j}(\tau) +  \cos^{3}(\tau) \sum _{j=0}^{\infty} \omega_{0} \omega_{i}\overline{\mathbb{F}}_{10 ji} \xi_{j}(\tau) + \cos^{3}(\tau) \sum _{j=0}^{\infty} \omega_{0} \omega_{i} \overline{\mathbb{F}}_{3 0 ji} b_{j}(\tau) \\
	& + \cos(\tau) \sin^{2}(\tau) \sum _{j=0}^{\infty}\omega_{0} \omega_{i} \overline{\mathbb{F}}_{40 ji} b_{j}(\tau) + \sum _{j,k=0}^{\infty} \omega_{i} \overline{\mathbb{G}}_{jki} \psi_{j}(\tau)(b_{k}(\tau)-\xi_{k}(\tau))  \Bigg) \\
	& + \epsilon^6 \Bigg(  \cos(\tau) \sum _{j,k=0}^{\infty} \omega_{0} \omega_{i} \overline{\mathbb{H}}_{0 jki} b_{j}(\tau) \xi_{k}(\tau) + \cos^{2}(\tau) \sum _{j,k=0}^{\infty} \omega_{i}\overline{\mathbb{F}}_{1jki} \xi_{k}(\tau) \psi_{j}(\tau) + \cos^{2}(\tau) \sum _{j,k=0}^{\infty} \omega_{i} \overline{\mathbb{F}}_{3jki} \psi_{j}(\tau) b_{k}(\tau) \\
	& + \sin^{2}(\tau) \sum _{j,k=0}^{\infty} \omega_{i}\overline{\mathbb{F}}_{2jki} \psi_{j}(\tau) \xi_{k}(\tau) + \sin^{2}(\tau) \sum _{j,k=0}^{\infty}  \omega_{i} \overline{\mathbb{F}}_{4jki} \psi_{j}(\tau) b_{k}(\tau) \Bigg) \\
	& + \epsilon^8 \Bigg(  \sum_{j,k,l=0}^{\infty} \omega_{i} \overline{\mathbb{H}}_{ljki} \xi_{k}(\tau) \psi_{l}(\tau) b_{j}(\tau) \Bigg),
\end{align*}
\begin{align*}	
	& \prescript{b}{}{\mathfrak{N}}_{i}(\psi(\tau),\sigma(\tau),b(\tau),\xi(\tau))  = \Bigg( 2 \cos(\tau) \sum _{j=0}^{\infty} \omega_{0}\overline{\mathbb{J}}_{0ji} \psi_{j}(\tau) - 2 \omega_{0} \sin(\tau) \sum _{j=0}^{\infty} \overline{\mathbb{I}}_{0ji} \sigma_{j}(\tau) \Bigg) \\
	& + \epsilon^2 \Bigg( - 2 \cos^3(\tau) \sum _{j=0}^{\infty} \omega_{0}\overline{\mathbb{P}}_{10ji} \psi_{j}(\tau) - 2 \cos(\tau) \sin^2(\tau) \sum _{j=0}^{\infty} \omega_{0}  \overline{\mathbb{P}}_{20ji} \psi_{j}(\tau) \\
	& + 2  \omega_{0} \cos^2(\tau) \sin(\tau) \sum _{j=0}^{\infty} \overline{\mathbb{Q}}_{10ji} \sigma_{j}(\tau) + 2 \omega_{0} \sin^3(\tau) \sum _{j=0}^{\infty} \overline{\mathbb{Q}}_{20ji} \sigma_{j}(\tau) - \cos^2(\tau) \sum _{j=0}^{\infty} \omega_{0}^2 \overline{\mathbb{R}}_{00ji} b_{j}(\tau) \\
	& -  \sin^2(\tau) \sum _{j=0}^{\infty} \omega_{0}^2 \overline{\mathbb{S}}_{00ji} b_{j}(\tau) +  \sum _{j,k=0}^{\infty}  \overline{\mathbb{J}}_{jki} \psi_{j}(\tau) \psi_{k}(\tau) +  \sum _{j,k=0}^{\infty} \overline{\mathbb{I}}_{jki} \sigma_{j}(\tau)\sigma_{k}(\tau)  \Bigg) \\
	& + \epsilon^4 \Bigg( -\cos^2(\tau)  \sum _{j,k=0}^{\infty}  \overline{\mathbb{P}}_{1jki} \psi_{j}(\tau) \psi_{k}(\tau) -\sin^2(\tau)  \sum _{j,k=0}^{\infty} \overline{\mathbb{P}}_{2jki} \psi_{j}(\tau) \psi_{k}(\tau) \\
	& - 2 \cos (\tau)  \sum _{j,k=0}^{\infty} \omega_{0}  \overline{\mathbb{R}}_{0kji} \psi_{k}(\tau) b_{j}(\tau) -\cos^2(\tau)  \sum _{j,k=0}^{\infty} \overline{\mathbb{Q}}_{1jki} \sigma_{j}(\tau) \sigma_{k}(\tau) \\
	& -\sin^2(\tau)  \sum _{j,k=0}^{\infty} \overline{\mathbb{Q}}_{2jki} \sigma_{j}(\tau) \sigma_{k}(\tau) +2  \sin(\tau)  \sum _{j,k=0}^{\infty}\omega_{0} \overline{\mathbb{S}}_{0jki} \sigma_{k}(\tau)  b_{j}(\tau) \Bigg) \\
	& + \epsilon^6 \Bigg( -   \sum _{j,k,l=0}^{\infty}  \overline{\mathbb{R}}_{klji} \psi_{k}(\tau) \psi_{l}(\tau) b_{j} (\tau) -   \sum _{j,k,l=0}^{\infty} \overline{\mathbb{S}}_{jkli} \sigma_{k}(\tau) \sigma_{l}(\tau) b_{j} (\tau) \Bigg),
\end{align*}
\begin{align*}	
	& \prescript{\xi}{}{\mathfrak{N}}_{i}(\psi(\tau),\sigma(\tau),b(\tau),\xi(\tau))  = \Bigg( 2 \cos(\tau) \sum _{j=0}^{\infty} \omega_{0}  \frac{\mathbb{J}_{0ji}}{\omega_{i}} \psi_{j}(\tau) - 2 \omega_{0} \sin(\tau) \sum _{j=0}^{\infty} \frac{\mathbb{I}_{0ji}}{\omega_{i}} \sigma_{j}(\tau) \Bigg) \\
	& + \epsilon^2 \Bigg(  2 \cos^3(\tau) \sum _{j=0}^{\infty} \omega_{0}  \frac{\mathbb{P}_{30ji}}{\omega_{i}} \psi_{j}(\tau) + 2 \cos(\tau) \sin^2(\tau) \sum _{j=0}^{\infty} \omega_{0}  \frac{\mathbb{P}_{40ji}}{\omega_{i}} \psi_{j}(\tau) -2  \omega_{0} \sin^3(\tau) \sum _{j=0}^{\infty}  \frac{\mathbb{Q}_{40ji}}{\omega_{i}} \sigma_{j}(\tau) \\
	&- 2 \omega_{0} \cos^2(\tau )\sin(\tau) \sum _{j=0}^{\infty} \frac{\mathbb{Q}_{30ji}}{\omega_{i}} \sigma_{j}(\tau) + \cos^2(\tau) \sum _{j=0}^{\infty} \omega_{0}^2 \frac{\mathbb{R}_{00ji}}{\omega_{i}} \xi_{j}(\tau) + \omega_{0}^2 \sin^2(\tau) \sum _{j=0}^{\infty} \omega_{i}\frac{\mathbb{S}_{00ji}}{\omega_{i}^2} \xi_{j}(\tau) \\
	& +  \sum _{j,k=0}^{\infty}  \frac{\mathbb{J}_{jki}}{\omega_{i}} \psi_{j}(\tau) \psi_{k}(\tau) +  \sum _{j,k=0}^{\infty}  \frac{\mathbb{I}_{jki}}{\omega_{i}} \sigma_{j}(\tau)\sigma_{k}(\tau)  \Bigg) \\
	& + \epsilon^4 \Bigg( \cos^2(\tau)  \sum _{j,k=0}^{\infty}  \frac{\mathbb{P}_{3jki}}{\omega_{i}} \psi_{j}(\tau) \psi_{k}(\tau) + \sin^2(\tau)  \sum _{j,k=0}^{\infty}  \frac{\mathbb{P}_{4jki}}{\omega_{i}} \psi_{j}(\tau) \psi_{k}(\tau) + 2 \cos (\tau)  \sum _{j,k=0}^{\infty} \omega_{0}  \frac{\mathbb{R}_{0kji}}{\omega_{i}} \psi_{k}(\tau) \xi_{j}(\tau)\\
	& +\cos^2(\tau)  \sum _{j,k=0}^{\infty} \frac{\mathbb{Q}_{3jki}}{\omega_{i}} \sigma_{j}(\tau) \sigma_{k}(\tau) + \sin^2(\tau)  \sum _{j,k=0}^{\infty} \frac{\mathbb{Q}_{4jki}}{\omega_{i}} \sigma_{j}(\tau) \sigma_{k}(\tau) - 2 \omega_{0} \sin(\tau)  \sum _{j,k=0}^{\infty}  \frac{\mathbb{S}_{0jki}}{\omega_{i}} \sigma_{k}(\tau)  \xi{j}(\tau) \Bigg) \\
	& + \epsilon^6 \Bigg( -   \sum _{j,k,l=0}^{\infty}  \frac{\mathbb{R}_{klji}}{\omega_{i}} \psi_{k}(\tau)\psi_{l}(\tau) \xi_{j} (\tau) +   \sum _{j,k,l=0}^{\infty}  \frac{\mathbb{S}_{jkli}}{\omega_{i}} \sigma_{k}(\tau) \sigma_{l}(\tau) \xi_{j} (\tau) \Bigg).
\end{align*}
\subsection{Approximate periodic solution and small divisors}\label{giaAppendix}
Similarly to the first approach with the infinite sum, the linear and homogeneous part of the ordinary differential equation for $(\psi_{i},\sigma_{i})$ is simply the equation for the harmonic oscillator and hence we can use the variational constants formula to solve it. We find the fixed-point formulation
\begin{align*}
	\psi_{i}(\tau) = \prescript{\psi}{}{S}_{i}(\tau) + \int_{0}^{\tau} \Bigg( &  \frac{\cos\left( \frac{\omega_{i} (\tau-s)}{\omega_{0} +\theta_{0}\epsilon^2+ \eta_{0}\epsilon^4} \right)}{\omega_{0} +\theta_{0}\epsilon^2+ \eta_{0}\epsilon^4} \frac{\prescript{\psi}{}{\mathfrak{N}}_{i}(\psi(s),\sigma(s),b(s),\xi(s))}{\omega_{i}} \\
	& + \frac{\sin\left( \frac{\omega_{i} (\tau-s)}{\omega_{0} +\theta_{0}\epsilon^2+ \eta_{0}\epsilon^4} \right)}{\omega_{0} +\theta_{0}\epsilon^2+ \eta_{0}\epsilon^4} \prescript{\sigma}{}{\mathfrak{N}}_{i}(\psi(s),\sigma(s),b(s),\xi(s)) \Bigg) ds, \\
	\sigma_{i}(\tau) = \prescript{\sigma}{}{S}_{i}(\tau) + \int_{0}^{\tau} \Bigg( &  \frac{-\sin\left( \frac{\omega_{i} (\tau-s)}{\omega_{0} +\theta_{0}\epsilon^2+ \eta_{0}\epsilon^4} \right)}{\omega_{0} +\theta_{0}\epsilon^2+ \eta_{0}\epsilon^4} \frac{\prescript{\psi}{}{\mathfrak{N}}_{i}(\psi(s),\sigma(s),b(s),\xi(s))}{\omega_{i}} \\
	& + \frac{\cos\left( \frac{\omega_{i} (\tau-s)}{\omega_{0} +\theta_{0}\epsilon^2+ \eta_{0}\epsilon^4} \right)}{\omega_{0} +\theta_{0}\epsilon^2+ \eta_{0}\epsilon^4} \prescript{\sigma}{}{\mathfrak{N}}_{i}(\psi(s),\sigma(s),b(s),\xi(s)) \Bigg) ds, 
\end{align*}
subject to the constrain equations
\begin{align*}
	b_{i} (\tau)& =  \prescript{b}{}{S}_{i}(\tau) + \prescript{b}{}{\mathfrak{N}}_{i}(\psi(\tau),\sigma(\tau),b(\tau),\xi(\tau)) ,\\
	\xi_{i} (\tau)& = \prescript{\xi}{}{S}_{i}(\tau) + \prescript{\xi}{}{\mathfrak{N}}_{i}(\psi(\tau),\sigma(\tau),b(\tau),\xi(\tau))
\end{align*}
where
\begin{align*}
	\prescript{\psi}{}{S}_{i}(\tau) & = \cos\left( \frac{\omega_{i} \tau}{\omega_{0} +\theta_{0}\epsilon^2+ \eta_{0}\epsilon^4} \right) \psi_{i}(0) + \sin\left( \frac{\omega_{i} \tau}{\omega_{0} +\theta_{0}\epsilon^2+ \eta_{0}\epsilon^4} \right) \sigma_{i}(0) \\
	& + \int_{0}^{\tau} \Bigg( \frac{\cos\left( \frac{\omega_{i} (\tau-s)}{\omega_{0} +\theta_{0}\epsilon^2+ \eta_{0}\epsilon^4} \right)}{\omega_{0} +\theta_{0}\epsilon^2+ \eta_{0}\epsilon^4} \frac{\prescript{\psi}{}{F}_{i}(s)}{\omega_{i}} + \frac{\sin\left( \frac{\omega_{i} (\tau-s)}{\omega_{0} +\theta_{0}\epsilon^2+ \eta_{0}\epsilon^4} \right)}{\omega_{0} +\theta_{0}\epsilon^2+ \eta_{0}\epsilon^4} \prescript{\sigma}{}{F}_{i}(s) \Bigg) ds, \\
	\prescript{\sigma}{}{S}_{i}(\tau) & = -\sin\left( \frac{\omega_{i} \tau}{\omega_{0} +\theta_{0}\epsilon^2+ \eta_{0}\epsilon^4} \right) \psi_{i}(0) + \cos \left( \frac{\omega_{i} \tau}{\omega_{0} +\theta_{0}\epsilon^2+ \eta_{0}\epsilon^4} \right) \sigma_{i}(0) \\
	& + \int_{0}^{\tau} \Bigg( \frac{-\sin\left( \frac{\omega_{i} (\tau-s)}{\omega_{0} +\theta_{0}\epsilon^2+ \eta_{0}\epsilon^4} \right)}{\omega_{0} +\theta_{0}\epsilon^2+ \eta_{0}\epsilon^4} \frac{\prescript{\psi}{}{F}_{i}(s)}{\omega_{i}} + \frac{\cos\left( \frac{\omega_{i} (\tau-s)}{\omega_{0} +\theta_{0}\epsilon^2+ \eta_{0}\epsilon^4} \right)}{\omega_{0} +\theta_{0}\epsilon^2+ \eta_{0}\epsilon^4} \prescript{\sigma}{}{F}_{i}(s) \Bigg) ds. 
\end{align*}
Furthermore, one can use various trigonometric identities to write 
\begin{align*}
\prescript{\psi}{}{F}_{i}(s) & = K_{i}(\epsilon ^{2}) \sin(s) + \Lambda_{i}(\epsilon^{2}) \sin(3s) + M_{i}(\epsilon^{2}) \sin(5s), \\
	\prescript{\sigma}{}{F}_{i}(s) & = N_{i}(\epsilon^{2}) \cos(s) + \Xi_{i}(\epsilon^{2}) \cos(3s) + T_{i}(\epsilon^{2}) \cos(5s)
\end{align*}
where
\begin{align*}
	K_{i}(\epsilon^{2}) &:= \omega_{0} \omega_{i}^2  \Bigg(\frac{\delta_{i,0} \omega_{0} \theta_{0}}{\omega_{i}^2} + \frac{1}{4} \mathbb{C}_{130i} + \frac{3}{4}\mathbb{C}_{240i} +  \epsilon^2 \left(\frac{\delta_{i,0} \omega_{0} \eta_{0}}{\omega_{i}^2}+ \frac{1}{8} \mathbb{ D}_{130i} + \frac{5}{8}   \mathbb{ D}_{240i} + \frac{1}{8} \mathbb{E}_{14230i} \right) \Bigg), \\
	\Lambda_{i}(\epsilon^{2}) &:= \omega_{0}\omega_{i}^2 \Bigg( \frac{1}{4} \mathbb{C}_{130i} -\frac{1}{4} \mathbb{C}_{240i} + \epsilon^2 \left( \frac{3}{16} \mathbb{D}_{130i} -\frac{5}{16} \mathbb{D}_{240i} +\frac{1}{16} \mathbb{E}_{14230i} \right)   \Bigg), \\
	M_{i}(\epsilon^{2}) &:= \omega_{0}\omega_{i}^2 \epsilon^2 \Bigg( \frac{1}{16} \mathbb{D}_{130i} +\frac{1}{16} \mathbb{D}_{240i} - \frac{1}{16} \mathbb{E}_{14230i}  \Bigg), \\
	N_{i}(\epsilon^{2}) &:= \omega_{0} \omega_{i} \left(\frac{ \delta_{i,0} \theta_{0}}{\omega_{i}} + \frac{3}{4} \overline{\mathbb{C}}_{130i} + \frac{1}{4}  \overline{\mathbb{C}}_{240i} +  \epsilon^2 \left( \frac{\delta_{i,0} \eta_{0}}{\omega_{i}}+ \frac{5}{32} \overline{\mathbb{D}}_{130i} + \frac{63}{64}   \overline{\mathbb{D}}_{240i} + \frac{19}{32}  \overline{\mathbb{E}}_{14230i} \right) \right), \\
	\Xi_{i}(\epsilon^{2}) &:= \omega_{0} \omega_{i} \left( \frac{1}{4} \overline{\mathbb{C}}_{130i} -\frac{1}{4} \overline{\mathbb{C}}_{240i} + \epsilon^2 \left( \frac{5}{64} \overline{\mathbb{D}}_{130i} + \frac{27}{64} \overline{\mathbb{D}}_{240i} +\frac{11}{64} \overline{\mathbb{E}}_{14230i} \right) \right), \\
	T_{i}(\epsilon^{2}) &:=\omega_{0} \omega_{i} \epsilon^2 \Big( \frac{1}{16} \overline{\mathbb{D}}_{130i} +\frac{5}{64} \overline{\mathbb{D}}_{240i} - \frac{1}{64} \overline{\mathbb{E}}_{14230i}  \Big).
\end{align*}
Now, computing the integrals above we obtain
\begin{align*}
\prescript{\psi}{}{S}_{i}(\tau) & = \prescript{\psi}{}{\mathcal{K}}_{i}(\epsilon^{2}) \cos(\tau) + \prescript{\psi}{}{\mathcal{R}}_{i}(\epsilon^{2}) \cos(3 \tau) + \prescript{\psi}{}{\mathcal{M}}_{i}(\epsilon^{2}) \cos(5\tau) \\
&+ \Big( \psi_{i}(0) + \prescript{\psi}{}{\mathcal{N}}_{i}(\epsilon^{2}) \Big) \cos\left( \frac{\omega_{i} \tau}{\omega_{0} +\theta_{0}\epsilon^2+\eta_{0}\epsilon^4} \right) +  \sigma_{i}(0) \sin\left( \frac{\omega_{i} \tau}{\omega_{0} +\theta_{0}\epsilon^2+\eta_{0}\epsilon^4} \right), \\ 
\prescript{\sigma}{}{S}_{i}(\tau) & = \prescript{\sigma}{}{\mathcal{K}}_{i}(\epsilon^{2}) \sin(\tau) + \prescript{\sigma}{}{\mathcal{R}}_{i}(\epsilon^{2}) \sin(3 \tau) + \prescript{\sigma}{}{\mathcal{M}}_{i}(\epsilon^{2}) \sin(5\tau) \\
&+ \sigma_{i}(0) \cos\left( \frac{\omega_{i} \tau}{\omega_{0} +\theta_{0}\epsilon^2+\eta_{0}\epsilon^4} \right) -   \Big( \psi_{i}(0) - \prescript{\sigma}{}{\mathcal{N}}_{i}(\epsilon^{2}) \Big) \sin\left( \frac{\omega_{i} \tau}{\omega_{0} +\theta_{0}\epsilon^2+\eta_{0}\epsilon^4} \right),
\end{align*}
where
\begin{align*}
\prescript{\psi}{}{\mathcal{K}}_{i}(\epsilon^{2}) & = 
\frac{1}{\omega_{i}}
\frac{ (\omega_{0}+\epsilon^2 \theta_{0}+\epsilon^4 \eta_{0}) K_{i}(\epsilon^2) + \omega_{i}^2 N_{i}(\epsilon^2) }{
(\omega_{i}-\omega_{0}-\epsilon^2 \theta_{0} -\epsilon^4 \eta_{0})
(\omega_{i}+\omega_{0}+\epsilon^2 \theta_{0}+\epsilon^4 \eta_{0} )
}, \\
\prescript{\psi}{}{\mathcal{R}}_{i}(\epsilon^{2}) & = 
\frac{1}{\omega_{i}}
\frac{ 3(\omega_{0}+\epsilon^2 \theta_{0}+\epsilon^4 \eta_{0}) \Lambda_{i}(\epsilon^2) + \omega_{i}^2 \Xi_{i}(\epsilon^2) }{
(\omega_{i}-3\omega_{0}-3\epsilon^2 \theta_{0}-3\epsilon^4 \eta_{0} )
(\omega_{i}+3\omega_{0}+3\epsilon^2 \theta_{0} +3\epsilon^4 \eta_{0})
}, \\
\prescript{\psi}{}{\mathcal{M}}_{i}(\epsilon^{2}) & = 
\frac{1}{\omega_{i}}
\frac{ 5(\omega_{0}+\epsilon^2 \theta_{0}+\epsilon^4 \eta_{0}) M_{i}(\epsilon^2) + \omega_{i}^2 T_{i}(\epsilon^2) }{
(\omega_{i}-5\omega_{0}-5\epsilon^2 \theta_{0} -5\epsilon^4 \eta_{0})
(\omega_{i}+5\omega_{0}+5\epsilon^2 \theta_{0} +5\epsilon^4 \eta_{0})
}, \\
\prescript{\psi}{}{\mathcal{N}}_{i}(\epsilon^{2}) & = 
-\frac{1}{\omega_{i}}
\frac{ (\omega_{0}+\epsilon^2 \theta_{0}+\epsilon^4 \eta_{0}) K_{i}(\epsilon^2) + \omega_{i}^2 N_{i}(\epsilon^2) }{
(\omega_{i}-\omega_{0}-\epsilon^2 \theta_{0} -\epsilon^4 \eta_{0})
(\omega_{i}+\omega_{0}+\epsilon^2 \theta_{0} +\epsilon^4 \eta_{0})
} \\
&\quad -\frac{1}{\omega_{i}}
\frac{ 3(\omega_{0}+\epsilon^2 \theta_{0}+\epsilon^4 \eta_{0}) \Lambda_{i}(\epsilon^2) + \omega_{i}^2 \Xi_{i}(\epsilon^2) }{
(\omega_{i}-3\omega_{0}-3\epsilon^2 \theta_{0}-3\epsilon^4 \eta_{0} )
(\omega_{i}+3\omega_{0}+3\epsilon^2 \theta_{0}+3\epsilon^4 \eta_{0} )
} \\
&\quad -\frac{1}{\omega_{i}}
\frac{ 5(\omega_{0}+\epsilon^2 \theta_{0}+\epsilon^4 \eta_{0}) M_{i}(\epsilon^2) + \omega_{i}^2 T_{i}(\epsilon^2) }{
(\omega_{i}-5\omega_{0}-5\epsilon^2 \theta_{0}-5\epsilon^4 \eta_{0} )
(\omega_{i}+5\omega_{0}+5\epsilon^2 \theta_{0}+5\epsilon^4 \eta_{0} )
}
\end{align*}
and
\begin{align*}
\prescript{\sigma}{}{\mathcal{K}}_{i}(\epsilon^{2}) & = 
-
\frac{ (\omega_{0}+\epsilon^2 \theta_{0}+\epsilon^4 \eta_{0}) N_{i}(\epsilon^2) + K_{i}(\epsilon^2) }{
(\omega_{i}-\omega_{0}-\epsilon^2 \theta_{0} -\epsilon^4 \eta_{0})
(\omega_{i}+\omega_{0}+\epsilon^2 \theta_{0} +\epsilon^4 \eta_{0})
}, \\
\prescript{\sigma}{}{\mathcal{R}}_{i}(\epsilon^{2}) & = 
-
\frac{ 3(\omega_{0}+\epsilon^2 \theta_{0}+\epsilon^4 \eta_{0}) \Xi_{i}(\epsilon^2) +  \Lambda_{i}(\epsilon^2) }{
(\omega_{i}-3\omega_{0}-3\epsilon^2 \theta_{0} -3\epsilon^4 \eta_{0})
(\omega_{i}+3\omega_{0}+3\epsilon^2 \theta_{0}+3\epsilon^4 \eta_{0} )
}, \\
\prescript{\sigma}{}{\mathcal{M}}_{i}(\epsilon^{2}) & = 
-
\frac{ 5(\omega_{0}+\epsilon^2 \theta_{0}+\epsilon^4 \eta_{0}) T_{i}(\epsilon^2) +  M_{i}(\epsilon^2) }{
(\omega_{i}-5\omega_{0}-5\epsilon^2 \theta_{0} -5\epsilon^4 \eta_{0})
(\omega_{i}+5\omega_{0}+5\epsilon^2 \theta_{0}+5\epsilon^4 \eta_{0} )
}, \\
\prescript{\sigma}{}{\mathcal{N}}_{i}(\epsilon^{2}) & = 
 \frac{1}{\omega_{i}}
\frac{ (\omega_{0}+\epsilon^2 \theta_{0}+\epsilon^4 \eta_{0}) K_{i}(\epsilon^2) + \omega_{i}^2 N_{i}(\epsilon^2) }{
(\omega_{i}-\omega_{0}-\epsilon^2 \theta_{0} -\epsilon^4 \eta_{0})
(\omega_{i}+\omega_{0}+\epsilon^2 \theta_{0}+\epsilon^4 \eta_{0} )
} \\
& + \frac{1}{\omega_{i}}
\frac{ 3(\omega_{0}+\epsilon^2 \theta_{0}+\epsilon^4 \eta_{0}) \Lambda_{i}(\epsilon^2) + \omega_{i}^2 \Xi_{i}(\epsilon^2) }{
(\omega_{i}-3\omega_{0}-3\epsilon^2 \theta_{0}-3\epsilon^4 \eta_{0} )
(\omega_{i}+3\omega_{0}+3\epsilon^2 \theta_{0} +3\epsilon^4 \eta_{0})
} \\
& +\frac{1}{\omega_{i}}
\frac{ 5(\omega_{0}+\epsilon^2 \theta_{0}+\epsilon^4 \eta_{0}) M_{i}(\epsilon^2) + \omega_{i}^2 T_{i}(\epsilon^2) }{
(\omega_{i}-5\omega_{0}-5\epsilon^2 \theta_{0}-5\epsilon^4 \eta_{0} )
(\omega_{i}+5\omega_{0}+5\epsilon^2 \theta_{0} +5\epsilon^4 \eta_{0})
}.
\end{align*}
Notice that conditions like
\begin{align*}
& \omega_{i}\pm \omega_{0} \pm \epsilon^2 \theta_{0} \pm \epsilon^4 \eta_{0} \neq 0, \quad 
\omega_{i}\pm 3\omega_{0} \pm 3\epsilon^2 \theta_{0} \pm 3\epsilon^4 \eta_{0} \neq 0, \quad 
 \omega_{i}\pm 5\omega_{0} \pm 5\epsilon^2 \theta_{0} \pm 5\epsilon^4 \eta_{0} \neq 0,
\end{align*}
are closely related to small divisors and play an important role in KAM theory \cite{MR2345400, MR3097022, MR4062430, MR3867631, MR3569244, MR3502158}. Observe that the identity
\begin{align*}
	\prescript{\psi}{}{\mathcal{N}}_{i}(\epsilon^{2}) +\prescript{\sigma}{}{\mathcal{N}}_{i}(\epsilon^{2}) = 0
\end{align*}
holds true for all $i=0,1,\dots$ and all $\epsilon >0$. Here, all the constants involved 
\begin{align*}
	\prescript{\psi}{}{\mathcal{K}}_{i}(\epsilon^{2}),\prescript{\psi}{}{\mathcal{R}}_{i}(\epsilon^{2}),\prescript{\psi}{}{\mathcal{M}}_{i}(\epsilon^{2}),\prescript{\psi}{}{\mathcal{N}}_{i}(\epsilon^{2}), \prescript{\sigma}{}{\mathcal{K}}_{i}(\epsilon^{2}),\prescript{\sigma}{}{\mathcal{R}}_{i}(\epsilon^{2}),\prescript{\sigma}{}{\mathcal{M}}_{i}(\epsilon^{2}),\prescript{\sigma}{}{\mathcal{N}}_{i}(\epsilon^{2})
\end{align*}
depend explicitly on the Fourier constants defined above and most importantly depend only on one index. Due to this fact, we can compute them (see Lemma \ref{LemmaAppendixB} in Appendix B). The fact that these constants are given in closed forms has a numerous advantages. First, we get the asymptotic behaviour for large $i$ and fixed $\epsilon>0$,
\begin{align*}
&	\prescript{\psi}{}{\mathcal{K}}_{i}(\epsilon^{2}) \simeq \frac{1}{\omega_{i}^5},\prescript{\psi}{}{\mathcal{R}}_{i}(\epsilon^{2}) \simeq \frac{\epsilon^2}{\omega_{i}^5},\prescript{\psi}{}{\mathcal{M}}_{i}(\epsilon^{2})\simeq \frac{\epsilon^2}{\omega_{i}^5},\prescript{\psi}{}{\mathcal{N}}_{i}(\epsilon^{2})\simeq \frac{1}{\omega_{i}^5}, \text{~for~} i \longrightarrow \infty \\
&	\prescript{\sigma}{}{\mathcal{K}}_{i}(\epsilon^{2}) \simeq \frac{1}{\omega_{i}^6},\prescript{\sigma}{}{\mathcal{R}}_{i}(\epsilon^{2}) \simeq \frac{\epsilon^2}{\omega_{i}^6},\prescript{\sigma}{}{\mathcal{M}}_{i}(\epsilon^{2})\simeq \frac{\epsilon^2}{\omega_{i}^6},\prescript{\sigma}{}{\mathcal{N}}_{i}(\epsilon^{2})\simeq \frac{1}{\omega_{i}^5}, \text{~for~} i \longrightarrow \infty.
\end{align*}
Second, we get their asymptotic behaviour for sufficiently small $\epsilon$ and fixed $i=0,1,\dots$, 
\begin{align*}
\prescript{\psi}{}{\mathcal{K}}_{i}(\epsilon^{2}) 
   &= \left\{\begin{array}{lr}
        \left(-3+\frac{459}{4 \pi \theta_{0}} \right) \epsilon^{-2} + (...)1+(...)\epsilon^2, & \text{for } i = 0 \\ [10pt]
        (...)1+(...)\epsilon^2, & \text{for } i \neq 0
        \end{array}\right., \\
\prescript{\psi}{}{\mathcal{R}}_{i}(\epsilon^{2}) 
   &= \left\{\begin{array}{lr}
      \frac{\omega_{3}}{72 \theta_{0}} \left( \omega_{3} \mathbb{C}_{1303}-\omega_{3} \mathbb{C}_{2403}+3 \omega_{0} \overline{\mathbb{C}}_{1303}-3 \omega_{0}   \overline{\mathbb{C}}_{2403} \right)\epsilon^{-2}+ (...)1+(...)\epsilon^2, & \text{for } i = 3 \\ [10pt]
        (...)1+(...)\epsilon^2, & \text{for } i \neq 3
        \end{array}\right.,\\
\prescript{\psi}{}{\mathcal{M}}_{i}(\epsilon^{2}) 
   &= \left\{\begin{array}{lr}
     (...)1+(...)\epsilon^2, & \text{for } i = 6 \\ [10pt]
        (...)\epsilon^2, & \text{for } i \neq 6
        \end{array}\right., \\ 
\prescript{\psi}{}{\mathcal{N}}_{i}(\epsilon^{2}) 
   &= \left\{\begin{array}{lr}
       \left(-3+\frac{459}{4 \pi \theta_{0}} \right)\epsilon^{-2}+ (...)1+(...)\epsilon^2, & \text{for } i = 0 \\ [10pt]
         \frac{\omega_{3}}{72 \theta_{0}} \left( \omega_{3} \mathbb{C}_{1303}-\omega_{3} \mathbb{C}_{2403}+3 \omega_{0} \overline{\mathbb{C}}_{1303}-3 \omega_{0}   \overline{\mathbb{C}}_{2403} \right)\epsilon^{-2}+(...)1+(...)\epsilon^2, & \text{for } i = 3 \\ [10pt]
         (...)1+(...)\epsilon^2, & \text{for } i = 6 \\ [10pt]
         (...)1+(...)\epsilon^2, & \text{for } i \neq 0,3,6
        \end{array}\right. 
\end{align*} 
\begin{align*}
\prescript{\sigma}{}{\mathcal{K}}_{i}(\epsilon^{2}) 
   &= \left\{\begin{array}{lr}
        \left(-3+\frac{459}{4 \pi \theta_{0}} \right) \epsilon^{-2} + (...)1+(...)\epsilon^2, & \text{for } i = 0 \\ [10pt]
        (...)1+(...)\epsilon^2, & \text{for } i \neq 0
        \end{array}\right., \\
\prescript{\sigma}{}{\mathcal{R}}_{i}(\epsilon^{2}) 
   &= \left\{\begin{array}{lr}
      \frac{\omega_{3}}{72 \theta_{0}} \left( \omega_{3} \mathbb{C}_{1303}-\omega_{3} \mathbb{C}_{2403}+3 \omega_{0} \overline{\mathbb{C}}_{1303}-3 \omega_{0}   \overline{\mathbb{C}}_{2403} \right)\epsilon^{-2}+ (...)1+(...)\epsilon^2, & \text{for } i = 3 \\ [10pt]
        (...)1+(...)\epsilon^2, & \text{for } i \neq 3
        \end{array}\right.,\\
\prescript{\sigma}{}{\mathcal{M}}_{i}(\epsilon^{2}) 
   &= \left\{\begin{array}{lr}
     (...)1+(...)\epsilon^2, & \text{for } i = 6 \\ [10pt]
        (...)\epsilon^2, & \text{for } i \neq 6
        \end{array}\right., \\ 
\prescript{\sigma}{}{\mathcal{N}}_{i}(\epsilon^{2}) 
   &= \left\{\begin{array}{lr}
       \left(-3+\frac{459}{4 \pi \theta_{0}} \right)\epsilon^{-2}+ (...)1+(...)\epsilon^2, & \text{for } i = 0 \\ [10pt]
         \frac{\omega_{3}}{72 \theta_{0}} \left( \omega_{3} \mathbb{C}_{1303}-\omega_{3} \mathbb{C}_{2403}+3 \omega_{0} \overline{\mathbb{C}}_{1303}-3 \omega_{0}   \overline{\mathbb{C}}_{2403} \right)\epsilon^{-2}+(...)1+(...)\epsilon^2, & \text{for } i = 3 \\ [10pt]
         (...)1+(...)\epsilon^2, & \text{for } i = 6 \\ [10pt]
         (...)1+(...)\epsilon^2, & \text{for } i \neq 0,3,6
        \end{array}\right.
\end{align*}
We compute
\begin{align*}
 \mathbb{C}_{1303} = -  \frac{291}{560\sqrt{10}\pi},
 ~\mathbb{C}_{2403} = -  \frac{99}{140\sqrt{10}\pi},
 ~\overline{\mathbb{C}}_{1303} = - \frac{111}{140\sqrt{10}\pi},
 ~\overline{\mathbb{C}}_{2403} = - \frac{339}{560\sqrt{10}\pi}
\end{align*}
and hence
\begin{align*}
	\omega_{3} \mathbb{C}_{1303}-\omega_{3} \mathbb{C}_{2403}+3 \omega_{0} \overline{\mathbb{C}}_{1303}-3 \omega_{0}   \overline{\mathbb{C}}_{2403} = 0.
\end{align*}
Consequently, by the structure of the equations, $\prescript{\psi}{}{\mathcal{R}}_{3},\prescript{\psi}{}{\mathcal{N}}_{3},\prescript{\sigma}{}{\mathcal{R}}_{3}$ and $\prescript{\sigma}{}{\mathcal{N}}_{3}$ cannot blowup as $\epsilon$ does to zero. However, we choose
\begin{align*}
	\theta_{0}:=  \frac{153}{4\pi}
\end{align*}
to ensure that every component of the periodic parts $\prescript{\psi}{}{S}_{i}(\tau)$ and $\prescript{\sigma}{}{S}_{i}(\tau)$ of $\psi_{i}$ and $\sigma_{i}$ respectively are bounded as $\epsilon$ goes to zero. This choice coincides with the choice of $\theta_{2}$ from the first approach as well as with the numerical computations of Rostworowski-Maliborski \cite{PhysRevLett111051102}. 
Similarly, using various trigonometric identities, we get
\begin{align*}
	\prescript{b}{}{S}_{i}(\tau) & = \prescript{b}{}{\mathcal{K}}_{i} + \prescript{b}{}{\mathcal{R}}_{i} \cos(2 \tau)+ \prescript{b}{}{\mathcal{M}}_{i}\cos(4 \tau) , \\
	\prescript{\xi}{}{S}_{i}(\tau) & =  \prescript{\xi}{}{\mathcal{K}}_{i} + \prescript{\xi}{}{\mathcal{R}}_{i} \cos(2 \tau)+ \prescript{\xi}{}{\mathcal{M}}_{i}\cos(4 \tau),
\end{align*}
where
\begin{align*}
\prescript{b}{}{\mathcal{K}}_{i} & =- \frac{\omega_{0}^2}{8} \Big( 3  \overline{\mathbb{P}}_{100i} + \overline{\mathbb{P}}_{200i} +  \overline{\mathbb{Q}}_{100i} + 3 \overline{\mathbb{Q}}_{200i} \Big),\\
\prescript{b}{}{\mathcal{R}}_{i} &= \frac{\omega_{0}^2 }{2}  \Big( -\overline{\mathbb{P}}_{100i} +  \overline{\mathbb{Q}}_{200i} \Big),\\
\prescript{b}{}{\mathcal{M}}_{i} &= \frac{\omega_{0}^2}{8} \Big(-  \overline{\mathbb{P}}_{100i} +  \overline{\mathbb{P}}_{200i} +  \overline{\mathbb{Q}}_{100i} - \overline{\mathbb{Q}}_{200i}\Big),\\
\prescript{\xi}{}{\mathcal{K}}_{i} &= \frac{\omega_{0}^2}{8\omega_{i}^2} \Big( 3 \mathbb{P}_{300i} +  \mathbb{P}_{400i} +  \mathbb{Q}_{300i} + 3 \mathbb{Q}_{400i} \Big),\\
	\prescript{\xi}{}{\mathcal{R}}_{i}& =\frac{\omega_{0}^2}{2 \omega_{i}^2} \Big( \mathbb{P}_{300i} -\mathbb{Q}_{400i} \Big),\\
 \prescript{\xi}{}{\mathcal{M}}_{i} &= \frac{\omega_{0}^2}{8\omega_{i}^2} \Big( \mathbb{P}_{300i} -  \mathbb{P}_{400i} - \mathbb{Q}_{300i} + \mathbb{Q}_{400i} \Big).
\end{align*}
As before, all the constants involved 
\begin{align*}
	\prescript{b}{}{\mathcal{K}}_{i},\prescript{b}{}{\mathcal{R}}_{i},\prescript{b}{}{\mathcal{M}}_{i}, \prescript{\xi}{}{\mathcal{K}}_{i},\prescript{\xi}{}{\mathcal{R}}_{i},\prescript{\xi}{}{\mathcal{M}}_{i}
\end{align*}
depend explicitly on the Fourier constants defined above and most importantly depend only on one index. Due to this fact, we can compute them (see Lemma \ref{LemmaAppendixB} in Appendix B). We immediately get their asymptotic behaviour for large $i$,
\begin{align*}
	& \prescript{b}{}{\mathcal{K}}_{i}\simeq \frac{1}{\omega_{i}^{12}},\prescript{b}{}{\mathcal{R}}_{i}\simeq \frac{1}{\omega_{i}^{12}},\prescript{b}{}{\mathcal{M}}_{i} \simeq \frac{1}{\omega_{i}^{12}}, \text{~for~} i \longrightarrow \infty \\
	&	\prescript{\xi}{}{\mathcal{K}}_{i}  \simeq \frac{1}{\omega_{i}^{6}},\prescript{\xi}{}{\mathcal{R}}_{i}\simeq \frac{1}{\omega_{i}^{6}},\prescript{\xi}{}{\mathcal{M}}_{i} \simeq \frac{1}{\omega_{i}^{6}}, \text{~for~} i \longrightarrow \infty.
\end{align*}

\subsection{Choice of the initial data}
We choose
\begin{align*}
	 \psi_{i}(0):=  - \prescript{\psi}{}{\mathcal{N}}_{i}(\epsilon^{2}) = \prescript{\sigma}{}{\mathcal{N}}_{i}(\epsilon^{2}), \quad \sigma_{i}(0):=0
\end{align*}
so that
\begin{align*}
\prescript{\psi}{}{S}_{i}(\tau) & = \prescript{\psi}{}{\mathcal{K}}_{i}(\epsilon^{2}) \cos(\tau) + \prescript{\psi}{}{\mathcal{R}}_{i}(\epsilon^{2}) \cos(3 \tau) + \prescript{\psi}{}{\mathcal{M}}_{i}(\epsilon^{2}) \cos(5\tau), \\
\prescript{\sigma}{}{S}_{i}(\tau) & = \prescript{\sigma}{}{\mathcal{K}}_{i}(\epsilon^{2}) \sin(\tau) + \prescript{\sigma}{}{\mathcal{R}}_{i}(\epsilon^{2}) \sin(3 \tau) + \prescript{\sigma}{}{\mathcal{M}}_{i}(\epsilon^{2}) \sin(5\tau).
\end{align*}
This choice is motivated by the fact that the source terms $\prescript{\psi}{}{S}_{i}(\tau)$ and $\prescript{\sigma}{}{S}_{i}(\tau)$ of the solutions $\psi_{i}(\tau)$ and $\sigma_{i}(\tau)$ would give rise to a periodic term. Indeed,
\begin{align*}
	\Psi(\tau,x) &= \sum_{i=0}^{\infty} \psi_{i}(\tau) 
	 \frac{e_{i}^{\prime}(x)}{\omega_{i}}
	 = \sum_{i=0}^{\infty} \prescript{\psi}{}{S}_{i}(\tau) 
	 \frac{e_{i}^{\prime}(x)}{\omega_{i}} + \mathcal{O} \left(\epsilon^2 \right) \\
	 & = \Bigg( \sum_{i=0}^{\infty} \prescript{\psi}{}{\mathcal{K}}_{i}(\epsilon^{2})  \frac{e_{i}^{\prime}(x)}{\omega_{i}} \Bigg) \cos(\tau)+ \Bigg( \sum_{i=0}^{\infty}  \prescript{\psi}{}{\mathcal{R}}_{i}(\epsilon^{2})  \frac{e_{i}^{\prime}(x)}{\omega_{i}} \Bigg)\cos(3 \tau) \\
	 & + \Bigg( \sum_{i=0}^{\infty} \prescript{\psi}{}{\mathcal{M}}_{i}(\epsilon^{2})  \frac{e_{i}^{\prime}(x)}{\omega_{i}} \Bigg) \cos(5\tau) + \mathcal{O} \left(\epsilon^2 \right)
\end{align*}
and similarly for $\Sigma, B$ and $\Theta$. 
\subsection{Growth and decay of the Fourier constants}
We are interested in the asymptotic behaviour of all the Fourier constants which appear using this approach. To begin with, we split them into five groups as follows
\begin{align*}
	\mathcal{A}_{1}& := \Bigg \{ \omega_{i}\mathbb{C}_{13ji}, \omega_{i}\mathbb{C}_{24ji}, \omega_{i}\mathbb{D}_{13ji},\omega_{i}\mathbb{D}_{24ji}, \omega_{i}\mathbb{E}_{1423ji}, \omega_{i}\mathbb{\overline{C}}_{13ji}, \omega_{i}\mathbb{\overline{C}}_{24ji}, \omega_{i}\mathbb{\overline{D}}_{13ji}, \omega_{i}\mathbb{\overline{D}}_{24ji}, \\
	& \quad \quad \quad \omega_{i}\mathbb{\overline{E}}_{1423ji}, \omega_{i}\mathbb{F}_{10ji}, \omega_{i}\mathbb{F}_{20ji},\omega_{i}\mathbb{F}_{30ji}, \omega_{i}\mathbb{F}_{40ji},\omega_{i}\mathbb{F}_{1jki},\omega_{i}\mathbb{F}_{2jki},\omega_{i}\mathbb{F}_{3jki}, \omega_{i}\mathbb{F}_{4jki},\omega_{i}\mathbb{G}_{0ji}, \\
&\quad \quad \quad \frac{ \omega_{0} \mathbb{P}_{10ji} }{\omega_{i}}, \frac{\omega_{0} \mathbb{P}_{20ji} }{\omega_{i}}, \frac{\omega_{0} \mathbb{P}_{30ji} }{\omega_{i}},\frac{\omega_{0} \mathbb{P}_{40ji}}{\omega_{i}}, \frac{\omega_{0} \mathbb{J}_{0ji}}{\omega_{i}}, \omega_{0} \mathbb{ \overline{P} }_{10ji}, \omega_{0} \mathbb{\overline{P}}_{20ji}, \omega_{0} \mathbb{\overline{P}}_{30ji},\omega_{0} \mathbb{\overline{P}}_{40ji}, \omega_{0} \mathbb{\overline{J}}_{0ji}	
\Bigg \}, \\
\mathcal{A}_{2}& := \Bigg \{ \omega_{i} \mathbb{F}_{1jki},\omega_{i} \mathbb{F}_{2jki},\omega_{i} \mathbb{F}_{3jki},\omega_{i} \mathbb{F}_{4jki},\frac{\omega_{0} \mathbb{R}_{0kji}}{\omega_{i}},\omega_{i} \mathbb{\overline{F}}_{1jki},\omega_{i} \mathbb{\overline{F}}_{2jki},\omega_{i} \mathbb{\overline{F}}_{3jki},\omega_{i} \mathbb{\overline{F}}_{4jki},\omega_{0} \mathbb{\overline{R}}_{0kji} \Bigg \}, \\
\mathcal{A}_{3}& := \Bigg \{ \frac{\mathbb{ Q}_{10ji}}{\omega_{i}},\frac{\mathbb{ Q}_{20ji}}{\omega_{i}},\frac{\mathbb{ Q}_{30ji}}{\omega_{i}},\frac{\mathbb{ Q}_{40ji}}{\omega_{i}},\frac{\mathbb{ I}_{0ji}}{\omega_{i}}, \frac{\omega_{0}^2 \mathbb{ R}_{00ji}}{\omega_{i}}, \frac{\mathbb{ S}_{00ji}}{\omega_{i}}, \omega_{0} \omega_{i}\mathbb{ \overline{F}}_{10ji},\omega_{0} \omega_{i}\mathbb{\overline{F}}_{20ji},\omega_{0} \omega_{i}\mathbb{\overline{F}}_{30ji},\omega_{0} \omega_{i}\mathbb{\overline{F}}_{40ji}, \\
&\quad \quad \quad \omega_{0} \omega_{i}\mathbb{\overline{G}}_{0ji}, \mathbb{\overline{Q}}_{10ji},\mathbb{\overline{Q}}_{20ji},\mathbb{\overline{Q}}_{30ji},\mathbb{\overline{Q}}_{40ji},\mathbb{\overline{I}}_{0ji},\omega_{0}^2\mathbb{\overline{R}}_{00ji},\mathbb{\overline{S}}_{00ji} \Bigg \}, \\
\mathcal{A}_{4}& := \Bigg \{ 
\frac{\mathbb{I}_{jki}}{\omega_{i}},\frac{\mathbb{J}_{jki}}{\omega_{i}}, \frac{\mathbb{P}_{1jki}}{\omega_{i}},\frac{\mathbb{P}_{2jki}}{\omega_{i}},\frac{\mathbb{P}_{3jki}}{\omega_{i}},\frac{\mathbb{P}_{4jki}}{\omega_{i}},
\frac{\mathbb{Q}_{1jki}}{\omega_{i}},\frac{\mathbb{Q}_{2jki}}{\omega_{i}},\frac{\mathbb{Q}_{3jki}}{\omega_{i}},\frac{\mathbb{Q}_{4jki}}{\omega_{i}},
\frac{\mathbb{S}_{0jki}}{\omega_{i}}
  \\
&\quad \quad \quad 
\overline{\mathbb{I}}_{jki}, \overline{\mathbb{J}}_{jki},\overline{\mathbb{P}}_{1jki},\overline{\mathbb{P}}_{2jki},\overline{\mathbb{P}}_{3jki}, \overline{\mathbb{P}}_{4jki}, 
\overline{\mathbb{Q}}_{1jki},\overline{\mathbb{Q}}_{2jki}, \overline{\mathbb{Q}}_{3jki}, \overline{\mathbb{Q}}_{4jki}, \overline{\mathbb{S}}_{0jki},
\omega_{0} \omega_{i}\overline{\mathbb{H}}_{0jki} \Bigg \}, \\
\mathcal{B}& := \Bigg \{  \omega_{i} \mathbb{H}_{0jki},    \omega_{i}\mathbb{G}_{jki},\frac{ \mathbb{R}_{klji}}{\omega_{i}},\frac{ \mathbb{S}_{jkli}}{\omega_{i} }, \omega_{i}\mathbb{\overline{G}}_{jki},  \omega_{i}\mathbb{H}_{klji}, \omega_{i}\mathbb{\overline{H}}_{ikjl}, \mathbb{\overline{R}}_{klji}, \overline{ \mathbb{S} }_{jkli}   \Bigg \}. 
\end{align*}
As before, we shall use the notation
\begin{align*}
	\sum_{\pm} f(a\pm b \pm c) = f(a + b + c) + f(a + b - c) + f(a - b + c) + f(a - b - c),
\end{align*}
that is summation with respect to all possible combinations of plus and minus and expressions like $\omega_{i} \pm \omega_{j} \pm \omega_{m}$ stand not only for $\omega_{i} + \omega_{j} + \omega_{m}$ and $\omega_{i} - \omega_{j} - \omega_{m}$ but also for $\omega_{i} + \omega_{j} - \omega_{m}$ and $\omega_{i} - \omega_{j} + \omega_{m}$, that is considering all possible combinations of plus and minus. We will use the leading order terms (Remark \ref{lot}) together with the asymptotic behavior of the oscillatory integrals (Lemma \ref{OscillatoryIntegrals}), the orthogonality properties (Lemma \ref{ClosedformulasFore}), the $L^{\infty}-$bounds for quantities related to the eigenfunctions (Lemma \ref{Linftyboundse}) and the $L^{\infty}-$bounds of the weights $\Gamma_{a}$ (estimate \eqref{LinftyGamma}).
\subsubsection{Fourier constants in $\mathcal{A}_{1}$, $\mathcal{A}_{2}$, $\mathcal{A}_{3}$ and $\mathcal{A}_{4}$}
First, we focus on the elements of $\mathcal{A}_{1}$. 
\begin{prop}[Fourier constants in $\mathcal{A}_{1}$]\label{A1}
	The following growth and decay estimates hold.
\begin{center}
\begin{longtable}{ |l|l|l|l|l| }
\hline
\multicolumn{5}{ |c| }{Growth and decay estimates for the Fourier constants in $\mathcal{A}_{1}$  as $i,j \longrightarrow + \infty$} \\
\hline
\quad  \quad Constant & \quad \quad $F$ & 1st derivative $\neq 0$ & $ \omega_{i} - \omega_{j} \longarrownot\longrightarrow \infty $ & \quad $ \omega_{i} - \omega_{j} \longrightarrow \infty $  \\ \hline
\multirow{1}{*}{ \(\displaystyle \omega_{i} \mathbb{C}_{13ji}, \omega_{i}\mathbb{\overline{C}}_{13ji}  \)}
 &  \(\displaystyle \quad \Gamma_{1}-\Gamma_{3} \)  &  \(\displaystyle  \quad F^{(3)} \left( \frac{\pi}{2} \right) \neq 0  \) & \quad  $\mathcal{O}\left(\omega_{i}\right)$  & \(\displaystyle  \sum_{\pm}
  \mathcal{O} \left(
   \frac{\omega_{i} }{ (\omega_{i} \pm \omega_{j})^4} 
   \right)
   \)  \\ \hline
\multirow{1}{*}{ \(\displaystyle \omega_{i}\mathbb{C}_{24ji}, \omega_{i}\mathbb{\overline{C}}_{24ji}\)}
 &  \(\displaystyle \quad \Gamma_{2}-\Gamma_{4} \)  &  \(\displaystyle  \quad F^{(3)} \left( \frac{\pi}{2} \right) \neq 0  \) & \quad  $\mathcal{O}\left(\omega_{i}\right)$  & \(\displaystyle  \sum_{\pm} 
 \mathcal{O} \left( \frac{\omega_{i} }{ (\omega_{i} \pm \omega_{j})^4}\right) \)  \\ \hline
\multirow{1}{*}{ \(\displaystyle \omega_{i}\mathbb{D}_{13ji}, \omega_{i}\mathbb{\overline{D}}_{13ji}\)}
 &  \(\displaystyle \quad \Gamma_{1} \Gamma_{3} \)  &  \(\displaystyle  \quad F^{(3)} \left( \frac{\pi}{2} \right) \neq 0  \) & \quad  $\mathcal{O}\left(\omega_{i}\right)$
  & \(\displaystyle  \sum_{\pm} 
  \mathcal{O} \left(
  \frac{\omega_{i} }{ (\omega_{i} \pm \omega_{j})^4}
  \right)
   \)  \\ \hline
\multirow{1}{*}{ \(\displaystyle \omega_{i}\mathbb{D}_{24ji}, \omega_{i}\mathbb{\overline{D}}_{24ji}\)}
 &  \(\displaystyle \quad \Gamma_{2} \Gamma_{4} \)  &  \(\displaystyle \quad F^{(3)} \left( \frac{\pi}{2} \right) \neq 0  \) & \quad  $\mathcal{O}\left(\omega_{i}\right)$
  & \(\displaystyle  \sum_{\pm}
  \mathcal{O} \left(
   \frac{\omega_{i} }{ (\omega_{i} \pm \omega_{j})^4}\right) \)  \\ \hline
\multirow{1}{*}{ \(\displaystyle \omega_{i}\mathbb{E}_{1423ji}, \omega_{i}\mathbb{\overline{E}}_{1423ji}\)}
 &  \(\displaystyle   \Gamma_{1} \Gamma_{4}+ \Gamma_{2} \Gamma_{3} \)  &  \(\displaystyle  \quad F^{(3)} \left( \frac{\pi}{2} \right) \neq 0  \) & \quad  $\mathcal{O}\left(\omega_{i}\right)$
  & \(\displaystyle  \sum_{\pm} \mathcal{O} \left( \frac{\omega_{i} }{ (\omega_{i} \pm \omega_{j})^4}\right) \)  \\ \hline
\multirow{1}{*}{\quad\quad \(\displaystyle \omega_{i}\mathbb{G}_{0ji}\)}
 &  \(\displaystyle  \quad\quad e_{0} \)  &  \(\displaystyle  \quad F^{(3)} \left( \frac{\pi}{2} \right) \neq 0  \) & \quad  $\mathcal{O}\left(\omega_{i} \right)$
  & \(\displaystyle  \sum_{\pm}\mathcal{O} \left( \frac{\omega_{i} }{ (\omega_{i} \pm \omega_{j})^4}\right) \)  \\ \hline
\multirow{1}{*}{\quad\quad \(\displaystyle \omega_{i}\mathbb{F}_{10ji}\)}
 &  \(\displaystyle \quad\quad  \Gamma_{1} e_{0} \)  &  \(\displaystyle  \quad F^{(9)} \left( \frac{\pi}{2} \right) \neq 0  \) & \quad  $\mathcal{O}\left(\omega_{i}\right)$
  & \(\displaystyle  \sum_{\pm}\mathcal{O} \left( \frac{\omega_{i} }{ (\omega_{i} \pm \omega_{j})^{10}} \right)\)  \\ \hline
\multirow{1}{*}{\quad\quad \(\displaystyle  \omega_{i}\mathbb{F}_{20ji}\)}
 &  \(\displaystyle  \quad\quad \Gamma_{2} e_{0} \)  &  \(\displaystyle  \quad F^{(11)} \left( \frac{\pi}{2} \right) \neq 0  \) & \quad  $\mathcal{O}\left(\omega_{i}\right)$
  & \(\displaystyle  \sum_{\pm}\mathcal{O} \left( \frac{\omega_{i} }{ (\omega_{i} \pm \omega_{j})^{12}}\right) \)  \\ \hline
\multirow{1}{*}{\quad\quad \(\displaystyle  \omega_{i}\mathbb{F}_{30ji}\)}
 &  \(\displaystyle \quad\quad  \Gamma_{3} e_{0} \)  &  \(\displaystyle  \quad F^{(3)} \left( \frac{\pi}{2} \right) \neq 0  \) & \quad  $\mathcal{O}\left(\omega_{i}\right)$
  & \(\displaystyle  \sum_{\pm}\mathcal{O} \left( \frac{\omega_{i} }{ (\omega_{i} \pm \omega_{j})^{4}}\right) \)  \\ \hline
\multirow{1}{*}{\quad\quad \(\displaystyle  \omega_{i}\mathbb{F}_{40ji}\)}
 &  \(\displaystyle \quad\quad  \Gamma_{4} e_{0} \)  &  \(\displaystyle  \quad F^{(3)} \left( \frac{\pi}{2} \right) \neq 0  \) & \quad  $\mathcal{O}\left(\omega_{i}\right)$
  & \(\displaystyle  \sum_{\pm}\mathcal{O} \left( \frac{\omega_{i} }{ (\omega_{i} \pm \omega_{j})^{4}}\right) \)  \\ \hline
\multirow{1}{*}{ \(\displaystyle \omega_{0}\mathbb{\overline{P}}_{10ji}, \frac{\omega_{0}\mathbb{P}_{10ji}}{\omega_{i}}  \)}
 &  \(\displaystyle   \Gamma_{1} e_{0}^{\prime}\sin \cos \)  &  \(\displaystyle  \quad F^{(9)} \left( \frac{\pi}{2} \right) \neq 0  \) &\quad $\mathcal{O}\left(\omega_{i}^{-1} \right)$
  & \(\displaystyle  \frac{1}{\omega_{i}} \sum_{\pm}\mathcal{O} \left( \frac{1 }{ (\omega_{i} \pm \omega_{j})^{10}}\right) \)  \\ \hline
\multirow{1}{*}{ \(\displaystyle \omega_{0}\mathbb{\overline{P}}_{20ji},\frac{\omega_{0}\mathbb{P}_{20ji}}{\omega_{i}} \)}
 &  \(\displaystyle  \Gamma_{2} e_{0}^{\prime}\sin \cos \)  &  \(\displaystyle \quad F^{(11)} \left( \frac{\pi}{2} \right) \neq 0  \) & \quad $\mathcal{O}\left(\omega_{i}^{-1} \right)$
  & \(\displaystyle \frac{1}{\omega_{i}} \sum_{\pm}\mathcal{O} \left( \frac{1 }{ (\omega_{i} \pm \omega_{j})^{12}}\right) \)  \\ \hline
\multirow{1}{*}{ \(\displaystyle  \omega_{0}\mathbb{\overline{P}}_{30ji},\frac{\omega_{0}\mathbb{P}_{30ji}}{\omega_{i}} \)}
 &  \(\displaystyle  \Gamma_{3} e_{0}^{\prime}\sin \cos \)  &  \(\displaystyle   \quad F^{(3)} \left( \frac{\pi}{2} \right) \neq 0  \) &\quad $\mathcal{O}\left(\omega_{i}^{-1} \right)$
  & \(\displaystyle \frac{1}{\omega_{i}} \sum_{\pm}\mathcal{O} \left( \frac{1 }{ (\omega_{i} \pm \omega_{j})^{4}}\right) \)  \\ \hline
\multirow{1}{*}{ \(\displaystyle  \omega_{0}\mathbb{\overline{P}}_{40ji},\frac{\omega_{0}\mathbb{P}_{40ji}}{\omega_{i}} \)}
 &  \(\displaystyle  \Gamma_{4} e_{0}^{\prime}\sin \cos \)  &  \(\displaystyle  \quad F^{(3)} \left( \frac{\pi}{2} \right) \neq 0  \) & \quad $\mathcal{O}\left(\omega_{i}^{-1} \right)$
  & \(\displaystyle \frac{1}{\omega_{i}} \sum_{\pm}\mathcal{O} \left( \frac{1 }{ (\omega_{i} \pm \omega_{j})^{4}}\right) \)  \\ \hline
\multirow{1}{*}{\(\displaystyle \quad \omega_{0}\mathbb{\overline{J}}_{0ji}, \frac{\omega_{0}\mathbb{J}_{0ji}}{\omega_{i}}\)}
 &  \(\displaystyle \quad e_{0}^{\prime}\sin \cos \)  &  \(\displaystyle  \quad F^{(3)} \left( \frac{\pi}{2} \right) \neq 0  \) &\quad $\mathcal{O}\left(\omega_{i}^{-1} \right)$
  & \(\displaystyle \frac{1}{\omega_{i}} \sum_{\pm}\mathcal{O} \left( \frac{1 }{ (\omega_{i} \pm \omega_{j})^{4}} \right)\)  \\ \hline
\end{longtable}
\end{center}	
\end{prop}

\begin{proof}
	All these estimates follow directly from Lemma \ref{byparts2} and in particular from 
\begin{align*}
	\int_{0}^{\frac{\pi}{2}} F(x) \cos( 2 b x )  dx   = 
 \sum_{k=0}^{N} \frac{(-1)^{k}}{a^{2k+2}} \left( (-1)^{b} F^{(2k+1)} \left(  \frac{\pi}{2} \right) -F^{(2k+1)} \left( 0 \right)  \right) + \mathcal{O} \left(  \frac{1}{a^{2N+2}} \right),
\end{align*}
as $b \longrightarrow \infty$. However, we illustrate the proof only for the first constant, namely $\omega_{i}\mathbb{C}_{13ji}$. For large values of $i,j$ and in the case where both $\omega_{i} \pm \omega_{j} \longrightarrow \infty$ (equivalently when $\omega_{i} - \omega_{j} \longrightarrow \infty$), 
	\begin{align*}
		\mathbb{C}_{13ji}  &:= \int_{0}^{\frac{\pi}{2}} \left( \Gamma_{1}(x) - \Gamma_{3}(x) \right) e_{i}(x)e_{j}(x)\tan^2(x) dx
		 \simeq \int_{0}^{\frac{\pi}{2}} \left( \Gamma_{1}(x) - \Gamma_{3}(x) \right) \sin(\omega_{i}x) \sin(\omega_{j}x) dx \\
		& = \frac{1}{2} \int_{0}^{\frac{\pi}{2}} \left( \Gamma_{1}(x) - \Gamma_{3}(x) \right) \cos((\omega_{i}-\omega_{j})x) dx - \frac{1}{2} \int_{0}^{\frac{\pi}{2}} \left( \Gamma_{1}(x) - \Gamma_{3}(x) \right) \cos((\omega_{i}+\omega_{j})x) dx.
	\end{align*}
Observe that both $\omega_{i} + \omega_{j}$ and $\omega_{i} - \omega_{j}$ are even,
\begin{align*}
	\omega_{i} + \omega_{j} = 2(3 +i+j), \quad \omega_{i} - \omega_{j} = 2(i-j).  
\end{align*}
We define $F(x):=\Gamma_{1}(x) - \Gamma_{3}(x)$. If $\omega_{i} - \omega_{j} \longrightarrow \infty$, then Lemma \ref{byparts2} applies and since
\begin{align*}
	F^{\prime} \left( \frac{\pi}{2}\right) = F^{\prime} \left( 0 \right) =0,~~~F^{\prime \prime \prime} \left(\frac{\pi}{2} \right) = -54 \neq 0
\end{align*}	
we infer
\begin{align*}
	\mathbb{C}_{13ji}=\mathcal{O}\left( \frac{1}{(\omega_{i} - \omega_{j})^4} + \frac{1}{(\omega_{i}+ \omega_{j})^4} \right),
\end{align*}
as $i,j \longrightarrow \infty$. On the other hand, if $\omega_{i}-\omega_{j} \longarrownot\longrightarrow \infty$, then we see that
\begin{align*}
	\left | \mathbb{C}_{13ji} \right| &:=\left| \int_{0}^{\frac{\pi}{2}} \left( \Gamma_{1}(x) - \Gamma_{3}(x) \right) e_{i}(x)e_{j}(x)\tan^2(x) dx \right| \\
		& \leq 	
	\left\|
	\Gamma_{1} - \Gamma_{3} 
	\right \|_{L^{\infty}\left[0,\frac{\pi}{2}\right]}
	\left \| 
	e_{i}\tan
	\right \|_{L^{2}\left[0,\frac{\pi}{2}\right]}
	\left \| 
	e_{j}\tan
	\right \|_{L^{2}\left[0,\frac{\pi}{2}\right]} 	\lesssim 1.
\end{align*}
In conclusion,
\begin{align*}
	\omega_{i} \mathbb{C}_{13ji}  =
 \Bigg \{ 
    \begin{array}{lr}
        \sum_{\pm}\mathcal{O}\left( \frac{\omega_{i}}{(\omega_{i} \pm \omega_{j})^4} \right), & \text{if } \omega_{i}-\omega_{j} \longrightarrow \infty\\
        \mathcal{O}\left(\omega_{i}\right), & \text{if }  \omega_{i}-\omega_{j}\longarrownot\longrightarrow \infty,
    \end{array}
\end{align*}
as $i,j \longrightarrow \infty$.
\end{proof}
\noindent
Second, we focus on the elements of $\mathcal{A}_{2}$.
\begin{prop}[Fourier constants in $\mathcal{A}_{2}$]\label{A2}
	Let $N\in \mathbb{N}$. 	The following growth and decay estimates hold. 
\begin{center}
\begin{longtable}{ |l|l|l|l|l| }
\hline
\multicolumn{5}{ |c| }{Growth and decay estimates for the Fourier constants in $\mathcal{A}_{2}$  as $i,j,k \longrightarrow + \infty$} \\
\hline
\quad \quad  Constant & \quad  $F$ & 1st derivative $\neq 0$& $\exists~ \omega_{i} \pm \omega_{j} \pm \omega_{k} \longarrownot\longrightarrow \infty$ & \quad  $\forall ~\omega_{i} \pm \omega_{j} \pm \omega_{k} \longrightarrow \infty$ \\ \hline
\multirow{1}{*}{ \(\displaystyle \omega_{i}\mathbb{F}_{1jki}, \omega_{i}\mathbb{\overline{F}}_{1jki}  \)}
 &  \(\displaystyle  \frac{\Gamma_{1}}{\tan}  \)  &  \(\displaystyle  \quad F^{(7)} \left( \frac{\pi}{2} \right) \neq 0  \) & \quad \quad\quad $\mathcal{O} \left(\omega_{i} \omega_{k} \right)$
  & \(\displaystyle  
  \quad \sum_{\pm}
  \mathcal{O} \left(
  \frac{\omega_{i} }{ (\omega_{i} \pm \omega_{j}\pm \omega_{k})^{8}} \right)\)  \\ \hline
\multirow{1}{*}{ \(\displaystyle \omega_{i}\mathbb{F}_{2jki}, \omega_{i}\mathbb{\overline{F}}_{2jki}  \)}
 &  \(\displaystyle  \frac{\Gamma_{2}}{\tan}  \)  &  \(\displaystyle  \quad F^{(9)} \left( \frac{\pi}{2} \right) \neq 0  \) & \quad \quad\quad $\mathcal{O} \left(\omega_{i} \omega_{k} \right)$
  & \(\displaystyle \quad \sum_{\pm} \mathcal{O} \left( \frac{\omega_{i} }{ (\omega_{i} \pm \omega_{j}\pm \omega_{k})^{10}}\right) \)  \\ \hline
\multirow{1}{*}{ \(\displaystyle \omega_{i}\mathbb{F}_{3jki}, \omega_{i}\mathbb{\overline{F}}_{3jki}  \)}
 &  \(\displaystyle  \frac{\Gamma_{3}}{\tan}  \)  &  \(\displaystyle  \quad F^{\prime} \left( \frac{\pi}{2} \right) \neq 0  \) & \quad \quad\quad $\mathcal{O} \left(\omega_{i} \omega_{k} \right)$
  & \(\displaystyle \quad \sum_{\pm} \mathcal{O} \left( \frac{\omega_{i} }{ (\omega_{i} \pm \omega_{j}\pm \omega_{k})^{2}}\right) \)  \\ \hline
\multirow{1}{*}{ \(\displaystyle \omega_{i}\mathbb{F}_{4jki}, \omega_{i}\mathbb{\overline{F}}_{4jki}  \)}
 &  \(\displaystyle  \frac{\Gamma_{4}}{\tan}  \)  &  \(\displaystyle  \quad F^{\prime} \left( \frac{\pi}{2} \right) \neq 0  \) & \quad \quad\quad $\mathcal{O} \left(\omega_{i} \omega_{k} \right)$
  & \(\displaystyle \quad \sum_{\pm} \mathcal{O} \left( \frac{\omega_{i} }{ (\omega_{i} \pm \omega_{j}\pm \omega_{k})^{2}}\right) \)  \\ \hline
\multirow{1}{*}{ \(\displaystyle \omega_{0}\mathbb{\overline{R}}_{0kji},  \frac{\omega_{0}\mathbb{R}_{0kji}}{\omega_{i}}  \)}
 &  \(\displaystyle  e_{0}^{\prime} \cos^2  \)  &\quad \quad  \quad  \text{None} & \quad \quad\quad $\mathcal{O} \left( \omega_{j}\omega_{i}^{-1} \right)$
  & \(\displaystyle \frac{1}{\omega_{i}} \sum_{\pm} \mathcal{O} \left( \frac{1 }{ (\omega_{i} \pm \omega_{j}\pm \omega_{k})^{N}}\right) \)  \\ \hline
\end{longtable}
\end{center}	
\end{prop}
\begin{proof}
All these estimates follow directly from Lemma \ref{byparts2} and in particular from 
\begin{align*}
\int_{0}^{\frac{\pi}{2}} F(x) \sin( (2 b+1) x )  dx   = 
\sum_{k=0}^{N} \frac{(-1)^{k}}{a^{2k+1}} F^{(2k)} \left(  0 \right) 
+ (-1)^{b} \sum_{k=0}^{N} \frac{(-1)^{k}}{a^{2k+2}} F^{(2k+1)} \left( \frac{\pi}{2} \right) 
  + \mathcal{O} \left(  \frac{1}{a^{2N+2}} \right),
\end{align*}
as $b \longrightarrow \infty$. However, we illustrate the proof only for the first constant, namely $\omega_{i}\mathbb{F}_{1jki}$. For large values of $i,j,k$ and in the case where all $\omega_{i} \pm \omega_{j} \pm \omega_{k} \longrightarrow \infty$,
	\begin{align*}
		\mathbb{F}_{1jki}  &:= \int_{0}^{\frac{\pi}{2}}  \Gamma_{1}(x)  e_{i}(x)e_{j}(x)e_{k}(x) \tan^2(x) dx 
		 \simeq \int_{0}^{\frac{\pi}{2}}  \frac{\Gamma_{1}(x)}{\tan(x)}  \sin(\omega_{i}x) \sin(\omega_{j}x) \sin(\omega_{k}x) dx \\
		& = \frac{1}{4} \int_{0}^{\frac{\pi}{2}}  \frac{\Gamma_{1}(x)}{\tan(x)}  \sin \left( \left(\omega_{i}-\omega_{j}+\omega_{k} \right) x\right) dx -  \frac{1}{4} \int_{0}^{\frac{\pi}{2}}  \frac{\Gamma_{1}(x)}{\tan(x)}  \sin \left( \left(\omega_{i}-\omega_{j}-\omega_{k} \right) x\right) dx \\
		& -  \frac{1}{4} \int_{0}^{\frac{\pi}{2}}  \frac{\Gamma_{1}(x)}{\tan(x)}  \sin \left( \left(\omega_{i}+\omega_{j}+\omega_{k} \right) x\right) dx
		+  \frac{1}{4} \int_{0}^{\frac{\pi}{2}}  \frac{\Gamma_{1}(x)}{\tan(x)}  \sin \left( \left(\omega_{i}+\omega_{j}-\omega_{k} \right) x\right) dx.
	\end{align*}
Observe that all $\omega_{i} \pm \omega_{j} \pm \omega_{k}$ are odd,
\begin{align*}
	\omega_{i} - \omega_{j} + \omega_{k} & = 2(i-j+k+1)+1, \quad 
	\omega_{i} - \omega_{j} - \omega_{k}   = 2(i-j-k-2)+1, \\
	\omega_{i} + \omega_{j} + \omega_{k} & = 2(i+j+k+4)+1, \quad 
	\omega_{i} + \omega_{j} - \omega_{k}   = 2(i+j-k+1)+1.
\end{align*}
We define $F(x):=\frac{\Gamma_{1}(x) }{\tan(x)}$. Lemma \ref{byparts2} applies and since
\begin{align*}
	F \left(0 \right) = F^{\prime} \left(  \frac{\pi}{2} \right) =0, F^{\prime \prime} \left(0 \right) = F^{\prime \prime \prime} \left(  \frac{\pi}{2} \right) =0, F^{(4)} \left(0 \right) = F^{(5)} \left(  \frac{\pi}{2} \right) =0, ~~~F^{(7)} \left(  \frac{\pi}{2} \right) = \frac{483840}{\pi} \neq 0
\end{align*}	
we infer
\begin{align*}
	\mathbb{F}_{1jki} = \sum_{\pm} \mathcal{O}\left( \frac{1}{(\omega_{i} \pm \omega_{j} \pm \omega_{k})^8} \right),
\end{align*}
as $i,j,k \longrightarrow \infty$. Finally, for large values of $i,j,k$ such that some $\omega_{i} \pm \omega_{j} \pm \omega_{k} \longarrownot\longrightarrow \infty$,  Holder's inequality implies
\begin{align*}
	\left| \mathbb{F}_{1jki} \right|& :=\left| \int_{0}^{\frac{\pi}{2}}  \Gamma_{1}(x)  e_{i}(x)e_{j}(x)e_{k}(x) \tan^2(x) dx \right| \\ & \leq 	
	\left\|
	\Gamma_{1} 
	\right \|_{L^{\infty}\left[0,\frac{\pi}{2}\right]}
	\left \| 
	e_{k}
	\right \|_{L^{\infty}\left[0,\frac{\pi}{2}\right]}
	\left \| 
	e_{i}\tan
	\right \|_{L^{2}\left[0,\frac{\pi}{2}\right]}
	\left \|
	e_{j}\tan
	\right \|_{L^{2}\left[0,\frac{\pi}{2}\right]} 	\lesssim \omega_{k}.
\end{align*}
as $i,j,k \longrightarrow \infty$. In conclusion,
\begin{align*}
	\omega_{i}\mathbb{F}_{1jki} =
 \Bigg \{ 
    \begin{array}{lr}
        \sum_{\pm}\mathcal{O} \left( \frac{\omega_{i}}{(\omega_{i} \pm \omega_{j}\pm \omega_{k})^{8}} \right), & \text{if all }\omega_{i} \pm \omega_{j} \pm \omega_{k} \longrightarrow \infty \\
      \mathcal{O} \left(\omega_{i}\omega_{k} \right), & \text{if some }  \omega_{i} \pm \omega_{j} \pm \omega_{k} \longarrownot\longrightarrow \infty,
    \end{array}
\end{align*}
that completes the proof.
\end{proof}
\noindent
Now, we focus on the elements of $\mathcal{A}_{3}$.
\begin{prop}[Fourier constants in $\mathcal{A}_{3}$]\label{A3}
	Let $N\in \mathbb{N}$. The following growth and decay estimates hold. 
\begin{center}
\begin{longtable}{ |l|l|l|l|l| }
\hline
\multicolumn{5}{ |c| }{Growth and decay estimates for the Fourier constants in $\mathcal{A}_{3}$  as $i,j \longrightarrow + \infty$} \\
\hline
\quad  Constant & \quad \quad $F$ & 1st derivative $\neq 0$& $ \omega_{i} - \omega_{j}  \longarrownot\longrightarrow \infty $ & \quad\quad $ \omega_{i} - \omega_{j}  \longrightarrow \infty $  \\ \hline
\multirow{1}{*}{ \(\displaystyle \quad  \omega_{0}\omega_{i} \mathbb{ \overline{F}}_{10ji}  \)}
 &  \(\displaystyle \quad \Gamma_{1} e_{0}^{\prime}  \)  &  \(\displaystyle  \quad F^{(8)} \left( \frac{\pi}{2} \right) \neq 0  \) & \quad $\mathcal{O}\left(\omega_{i} \right)$
  & \(\displaystyle  \sum_{\pm} 
  \mathcal{O} \left(
  \frac{\omega_{i} }{ (\omega_{i} \pm \omega_{j})^{9}}\right) \)  \\ \hline
\multirow{1}{*}{ \(\displaystyle \quad   \omega_{0}\omega_{i} \mathbb{\overline{F}}_{20ji}  \)}
 &  \(\displaystyle \quad \Gamma_{2} e_{0}^{\prime}  \)  &  \(\displaystyle  \quad F^{(10)} \left( \frac{\pi}{2} \right) \neq 0  \) &\quad $\mathcal{O}\left(\omega_{i} \right)$
  & \(\displaystyle  \sum_{\pm} \mathcal{O} \left( \frac{\omega_{i} }{ (\omega_{i} \pm \omega_{j})^{11}} \right)\)  \\ \hline  
\multirow{1}{*}{ \(\displaystyle \quad  \omega_{0}\omega_{i} \mathbb{\overline{F}}_{30ji}  \)}
 &  \(\displaystyle \quad \Gamma_{3} e_{0}^{\prime}  \)  &  \(\displaystyle  \quad F^{(2)} \left( \frac{\pi}{2} \right) \neq 0  \) & \quad $\mathcal{O}\left(\omega_{i} \right)$
  & \(\displaystyle  \sum_{\pm} \mathcal{O} \left( \frac{\omega_{i} }{ (\omega_{i} \pm \omega_{j})^{3}}\right) \)  \\ \hline
\multirow{1}{*}{ \(\displaystyle \quad  \omega_{0}\omega_{i} \mathbb{\overline{F}}_{40ji}  \)}
 &  \(\displaystyle \quad \Gamma_{4} e_{0}^{\prime}  \)  &  \(\displaystyle  \quad F^{(2)} \left( \frac{\pi}{2} \right) \neq 0  \) & \quad $\mathcal{O}\left(\omega_{i} \right)$
  & \(\displaystyle  \sum_{\pm} \mathcal{O} \left( \frac{\omega_{i} }{ (\omega_{i} \pm \omega_{j})^{3}}\right) \)  \\ \hline
\multirow{1}{*}{ \(\displaystyle \quad   \omega_{0}\omega_{i} \mathbb{\overline{G}}_{0ji}  \)}
 &  \(\displaystyle \quad\quad  e_{0}^{\prime}  \)  &  \(\displaystyle  \quad F^{(2)} \left( \frac{\pi}{2} \right) \neq 0  \) & \quad $\mathcal{O}\left(\omega_{i}\right)$
  & \(\displaystyle  \sum_{\pm} \mathcal{O} \left( \frac{\omega_{i} }{ (\omega_{i} \pm \omega_{j})^{3}} \right)\)  \\ \hline
\multirow{1}{*}{ \(\displaystyle \quad  \mathbb{ \overline{Q}}_{10ji},\frac{\mathbb{Q}_{10ji}}{\omega_{i}}  \)}
 &  \(\displaystyle  \Gamma_{1} e_{0} \sin \cos \)  &  \(\displaystyle  \quad F^{(10)} \left( \frac{\pi}{2} \right) \neq 0  \) & \quad $\mathcal{O}\left( \omega_{i}^{-1} \right)$
  & \(\displaystyle \frac{1}{\omega_{i}}  \sum_{\pm} \mathcal{O} \left( \frac{1 }{ (\omega_{i} \pm \omega_{j})^{11}}\right) \)  \\ \hline
\multirow{1}{*}{ \(\displaystyle\quad   \mathbb{\overline{Q}}_{20ji},\frac{\mathbb{Q}_{20ji}}{\omega_{i}}  \)}
 &  \(\displaystyle  \Gamma_{2} e_{0} \sin \cos  \)  &  \(\displaystyle  \quad F^{(12)} \left( \frac{\pi}{2} \right) \neq 0  \) & \quad $\mathcal{O}\left(\omega_{i}^{-1}  \right)$
  & \(\displaystyle \frac{1}{\omega_{i}}  \sum_{\pm} \mathcal{O} \left( \frac{1}{ (\omega_{i} \pm \omega_{j})^{13}} \right)\)  \\ \hline
\multirow{1}{*}{ \(\displaystyle \quad  \mathbb{\overline{Q}}_{30ji},\frac{\mathbb{Q}_{30ji}}{\omega_{i}}  \)}
 &  \(\displaystyle  \Gamma_{3} e_{0} \sin \cos  \)  &  \(\displaystyle  \quad F^{(4)} \left( \frac{\pi}{2} \right) \neq 0  \) & \quad $\mathcal{O}\left(\omega_{i}^{-1}  \right)$
  & \(\displaystyle \frac{1}{\omega_{i}}  \sum_{\pm}  \mathcal{O} \left(\frac{1 }{ (\omega_{i} \pm \omega_{j})^{5}} \right)\)  \\ \hline
\multirow{1}{*}{ \(\displaystyle \quad  \mathbb{\overline{Q}}_{40ji},\frac{\mathbb{Q}_{40ji}}{\omega_{i}}  \)}
 &  \(\displaystyle  \Gamma_{4} e_{0} \sin \cos  \)  &  \(\displaystyle  \quad F^{(4)} \left( \frac{\pi}{2} \right) \neq 0  \) & \quad $\mathcal{O}\left(\omega_{i}^{-1} \right)$
  & \(\displaystyle \frac{1}{\omega_{i}}  \sum_{\pm}  \mathcal{O} \left(\frac{1 }{ (\omega_{i} \pm \omega_{j})^{5}}\right) \)  \\ \hline
\multirow{1}{*}{ \(\displaystyle\quad   \overline{\mathbb{I}}_{0ji},\frac{\mathbb{I}_{0ji}}{\omega_{i}}  \)}
 &  \(\displaystyle   e_{0} \sin \cos  \)  &  \(\displaystyle  \quad F^{(4)} \left( \frac{\pi}{2} \right) \neq 0  \) & \quad $\mathcal{O}\left(\omega_{i}^{-1}  \right)$
  & \(\displaystyle \frac{1}{\omega_{i}}  \sum_{\pm} \mathcal{O} \left( \frac{1 }{ (\omega_{i} \pm \omega_{j})^{5}}\right) \)  \\ \hline  
\multirow{1}{*}{ \(\displaystyle \omega_{0}^2\overline{\mathbb{R}}_{00ji}, \frac{\omega_{0}^2 \mathbb{R}_{00ji}}{\omega_{i}}  \)}
 &  \(\displaystyle   (e_{0}^{\prime})^2 \sin \cos  \)  &  \quad \quad\quad  \text{None} & \quad $\mathcal{O}\left(\omega_{i}^{-1} \right)$
  & \(\displaystyle \frac{1}{\omega_{i}}  \sum_{\pm}  \mathcal{O} \left(\frac{1 }{ (\omega_{i} \pm \omega_{j})^{N}}\right) \)  \\ \hline 
\multirow{1}{*}{ \(\displaystyle \quad \overline{\mathbb{S}}_{00ji},\frac{\mathbb{S}_{00ji}}{\omega_{i}}  \)}
 &  \(\displaystyle ~~  e_{0}^2 \sin \cos  \)  &  \quad \quad  \quad\text{None} &\quad $\mathcal{O}\left(\omega_{i}^{-1}\right)$
  & \(\displaystyle \frac{1}{\omega_{i}}  \sum_{\pm} \mathcal{O} \left( \frac{1 }{ (\omega_{i} \pm \omega_{j})^{N}}\right) \)  \\ \hline 
\end{longtable}
\end{center}	
\end{prop}
\begin{proof}
All these estimates follow directly from Lemma \ref{byparts2} and in particular from 
\begin{align*}
\int_{0}^{\frac{\pi}{2}} F(x) \sin( 2 b x )  dx   = 
 \sum_{k=0}^{N} \frac{(-1)^{k+1}}{a^{2k+1}} \left( (-1)^{b} F^{(2k)} \left(  \frac{\pi}{2} \right) -F^{(2k)} \left( 0 \right)  \right) + \mathcal{O} \left(  \frac{1}{a^{2N+2}} \right),
\end{align*}
as $b \longrightarrow \infty$. However, we illustrate the proof only for the first constant, namely $\omega_{0}\omega_{i}\mathbb{\overline{F}}_{10ji}$. For large values of $i,j$ and in the case where $\omega_{i} - \omega_{j} \longrightarrow \infty$,
	\begin{align*}
		\omega_{0}\mathbb{\overline{F}}_{10ji}  &:= \int_{0}^{\frac{\pi}{2}}  \Gamma_{1}(x)  e_{j} (x)e_{0}^{\prime}(x) \frac{e_{i}^{\prime}(x)}{\omega_{i}} \tan^2(x) dx 
		 \simeq \int_{0}^{\frac{\pi}{2}}  \Gamma_{1}(x) e_{0}^{\prime}(x)  \cos(\omega_{i}x) \sin(\omega_{j}x)  dx \\
		& = \frac{1}{2} \int_{0}^{\frac{\pi}{2}}  \Gamma_{1}(x) e_{0}^{\prime}(x)  \sin \left( \left(\omega_{i}+\omega_{j} \right) x\right) dx 
		- \frac{1}{2} \int_{0}^{\frac{\pi}{2}}  \Gamma_{1}(x) e_{0}^{\prime}(x)  \sin \left( \left(\omega_{i}-\omega_{j} \right) x\right) dx. 
		\end{align*}
Observe that both $\omega_{i} + \omega_{j}$ and $\omega_{i} - \omega_{j}$ are even,
\begin{align*}
	\omega_{i} + \omega_{j} = 2(3 +i+j), \quad \omega_{i} - \omega_{j} = 2(i-j).  
\end{align*}
We define $F(x):= \Gamma_{1}(x) e_{0}^{\prime}(x) $ and compute
\begin{align*}
	F \left( \frac{\pi}{2} \right) = F \left(  0 \right) =0, \dots, F^{(6)} \left( \frac{\pi}{2} \right) = F^{(6) } \left( 0 \right) =0, ~~~F^{(8)} \left(  \frac{\pi}{2} \right) = \sqrt{\frac{2}{\pi}} \frac{46448640}{\pi^3} \neq 0.
\end{align*}	
Now, Lemma \ref{byparts2} yields
\begin{align*}
	\omega_{0}\mathbb{\overline{F}}_{10ji} = \sum_{\pm}\mathcal{O}\left( \frac{1}{(\omega_{i} \pm \omega_{j})^9} \right),
\end{align*}
as $i,j \longrightarrow \infty$. On the other hand, for large values of $i,j$ such that $\omega_{i} - \omega_{j} \longarrownot\longrightarrow \infty$,  Holder's inequality implies 
\begin{align*}
	\left|\omega_{0}\mathbb{\overline{F}}_{10ji}  \right|& :=\left| \int_{0}^{\frac{\pi}{2}}  \Gamma_{1}(x)  e_{j} (x)e_{0}^{\prime}(x) \frac{e_{i}^{\prime}(x)}{\omega_{i}} \tan^2(x) dx \right| \\ & \leq 	
	\left\|
	\Gamma_{1} e^{\prime}_{0}
	\right \|_{L^{\infty}\left[0,\frac{\pi}{2}\right]}
	\left \| 
	e_{j}\tan
	\right \|_{L^{2}\left[0,\frac{\pi}{2}\right]}
	\left \|
	\frac{e^{\prime}_i}{\omega_{i}}\tan
	\right \|_{L^{2}\left[0,\frac{\pi}{2}\right]} 	\lesssim 1.
\end{align*}
as $i,j \longrightarrow \infty$. In conclusion,
\begin{align*}
	\omega_{0}\omega_{i} \mathbb{\overline{F}}_{10ji} =
 \Bigg \{ 
    \begin{array}{lr}
        \sum_{\pm}\mathcal{O} \left(  \omega_{i} (\omega_{i} \pm \omega_{j})^{-9} \right), & \text{if }\omega_{i} - \omega_{j}  \longrightarrow \infty \\
      \mathcal{O} \left(\omega_{i} \right), & \text{if }  \omega_{i} - \omega_{j} \longarrownot\longrightarrow \infty,
    \end{array}
\end{align*}
that completes the proof.
\end{proof}
\noindent
Finally, we focus on the elements of $\mathcal{A}_{4}$.
\begin{prop}[Fourier constants in $\mathcal{A}_{4}$]\label{A4}
	Let $N\in \mathbb{N}$. The following growth and decay estimates hold. 
\begin{center}
\begin{longtable}{ |l|l|l|l|l| }
\hline
\multicolumn{5}{ |c| }{Growth and decay estimates for the Fourier constants in $\mathcal{A}_{4}$  as $i,j,k \longrightarrow + \infty$} \\
\hline
\quad  Constant & \quad  $F$ & 1st derivative $\neq 0$& $ \exists~\omega_{i} \pm \omega_{j} \pm \omega_{k} \longarrownot\longrightarrow \infty $ & \quad \quad $ \forall~\omega_{i} \pm \omega_{j} \pm \omega_{k} \longrightarrow \infty $  \\ \hline
\multirow{1}{*}{ \(\displaystyle \mathbb{ \overline{I}}_{jki},\frac{\mathbb{I}_{jki}}{\omega_{i}}    \)}
 &  \(\displaystyle ~~ \cos^2  \)  &  \(\displaystyle  \quad F^{(2)} \left( \frac{\pi}{2} \right) \neq 0  \) & \quad \quad \quad $\mathcal{O}\left(\omega_{i}^{-1} \right)$
  & \(\displaystyle \frac{1}{\omega_{i}} \sum_{\pm} 
  \mathcal{O} \left( 
  \frac{1 }{ (\omega_{i} \pm \omega_{j} \pm \omega_{k})^{3}} \right) 
  \)  \\ \hline 
\multirow{1}{*}{ \(\displaystyle \overline{\mathbb{J}}_{jki}, \frac{\mathbb{J}_{jki}}{\omega_{i}}    \)}
 &  \(\displaystyle ~~ \cos^2  \)  &  \(\displaystyle  \quad F^{(2)} \left( \frac{\pi}{2} \right) \neq 0  \) & \quad \quad \quad $\mathcal{O}\left(\omega_{i}^{-1} \right)$
  & \(\displaystyle \frac{1}{\omega_{i}} \sum_{\pm}\mathcal{O} \left( \frac{1 }{ (\omega_{i} \pm \omega_{j} \pm \omega_{k})^{3}}\right) \)  \\ \hline   
\multirow{1}{*}{ \(\displaystyle \overline{\mathbb{P}}_{1jki},\frac{\mathbb{P}_{1jki}}{\omega_{i}} \)}
 &  \(\displaystyle  \Gamma_{1} \cos^2  \)  &  \(\displaystyle  \quad F^{(8)} \left( \frac{\pi}{2} \right) \neq 0  \) & \quad \quad \quad $\mathcal{O}\left(\omega_{i}^{-1} \right)$
  & \(\displaystyle \frac{1}{\omega_{i}} \sum_{\pm}\mathcal{O} \left( \frac{1 }{ (\omega_{i} \pm \omega_{j} \pm \omega_{k})^{9}}\right) \)  \\ \hline   
\multirow{1}{*}{ \(\displaystyle \mathbb{\overline{P}}_{2jki},\frac{\mathbb{P}_{2jki}}{\omega_{i}}  \)}
 &  \(\displaystyle  \Gamma_{2} \cos^2  \)  &  \(\displaystyle  \quad F^{(10)} \left( \frac{\pi}{2} \right) \neq 0  \) & \quad \quad \quad $\mathcal{O}\left(\omega_{i}^{-1}\right)$
  & \(\displaystyle \frac{1}{\omega_{i}} \sum_{\pm}\mathcal{O} \left( \frac{1 }{ (\omega_{i} \pm \omega_{j} \pm \omega_{k})^{11}} \right)\)  \\ \hline 
\multirow{1}{*}{ \(\displaystyle \mathbb{\overline{P}}_{3jki},\frac{\mathbb{P}_{3jki}}{\omega_{i}}\)}
 &  \(\displaystyle  \Gamma_{3} \cos^2  \)  &  \(\displaystyle  \quad F^{(2)} \left( \frac{\pi}{2} \right) \neq 0  \) & \quad \quad \quad $\mathcal{O}\left(\omega_{i}^{-1} \right)$
  & \(\displaystyle \frac{1}{\omega_{i}} \sum_{\pm}\mathcal{O} \left( \frac{1 }{ (\omega_{i} \pm \omega_{j} \pm \omega_{k})^{3}}\right) \)  \\ \hline 
\multirow{1}{*}{ \(\displaystyle \mathbb{\overline{P}}_{4jki},\frac{\mathbb{P}_{4jki}}{\omega_{i}}  \)}
 &  \(\displaystyle  \Gamma_{4} \cos^2  \)  &  \(\displaystyle  \quad F^{(2)} \left( \frac{\pi}{2} \right) \neq 0  \) & \quad \quad \quad $\mathcal{O}\left(\omega_{i}^{-1} \right)$
  & \(\displaystyle \frac{1}{\omega_{i}} \sum_{\pm} \mathcal{O} \left(\frac{1 }{ (\omega_{i} \pm \omega_{j} \pm \omega_{k})^{3}}\right) \)  \\ \hline 
\multirow{1}{*}{ \(\displaystyle \mathbb{\overline{Q}}_{1jki},\frac{\mathbb{Q}_{1jki}}{\omega_{i}}  \)}
 &  \(\displaystyle  \Gamma_{1} \cos^2  \)  &  \(\displaystyle  \quad F^{(8)} \left( \frac{\pi}{2} \right) \neq 0  \) & \quad \quad \quad $\mathcal{O}\left(\omega_{i}^{-1} \right)$
  & \(\displaystyle \frac{1}{\omega_{i}} \sum_{\pm}\mathcal{O} \left( \frac{1 }{ (\omega_{i} \pm \omega_{j} \pm \omega_{k})^{9}}\right) \)  \\ \hline 
\multirow{1}{*}{ \(\displaystyle \mathbb{\overline{Q}}_{2jki},\frac{\mathbb{Q}_{2jki}}{\omega_{i}}  \)}
 &  \(\displaystyle  \Gamma_{2} \cos^2  \)  &  \(\displaystyle  \quad F^{(10)} \left( \frac{\pi}{2} \right) \neq 0  \) & \quad \quad \quad $\mathcal{O}\left(\omega_{i}^{-1} \right)$
  & \(\displaystyle \frac{1}{\omega_{i}} \sum_{\pm}\mathcal{O} \left( \frac{1 }{ (\omega_{i} \pm \omega_{j} \pm \omega_{k})^{11}}\right) \)  \\ \hline     
\multirow{1}{*}{ \(\displaystyle \mathbb{\overline{Q}}_{3jki},\frac{\mathbb{Q}_{3jki}}{\omega_{i}}  \)}
 &  \(\displaystyle  \Gamma_{3} \cos^2  \)  &  \(\displaystyle  \quad F^{(2)} \left( \frac{\pi}{2} \right) \neq 0  \) & \quad \quad \quad $\mathcal{O}\left(\omega_{i}^{-1} \right)$
  & \(\displaystyle \frac{1}{\omega_{i}} \sum_{\pm} \mathcal{O} \left( \frac{1 }{ (\omega_{i} \pm \omega_{j} \pm \omega_{k})^{3}} \right) \)  \\ \hline 
\multirow{1}{*}{ \(\displaystyle \mathbb{\overline{Q}}_{4jki},\frac{\mathbb{Q}_{4jki}}{\omega_{i}}  \)}
 &  \(\displaystyle  \Gamma_{4} \cos^2  \)  &  \(\displaystyle  \quad F^{(2)} \left( \frac{\pi}{2} \right) \neq 0  \) & \quad \quad \quad $\mathcal{O}\left(\omega_{i}^{-1}\right)$
  & \(\displaystyle \frac{1}{\omega_{i}} \sum_{\pm}\mathcal{O} \left( \frac{1 }{ (\omega_{i} \pm \omega_{j} \pm \omega_{k})^{3}}\right) \)  \\ \hline 
\multirow{1}{*}{\(\displaystyle ~ \mathbb{\overline{S}}_{0jki},\frac{\mathbb{S}_{0jki}}{\omega_{i} }   \)}
 &  \(\displaystyle  e_{0} \cos^2  \)  & \quad \quad \quad \text{None} & \quad \quad \quad $\mathcal{O}\left(\omega_{i}^{-1} \right)$
  & \(\displaystyle \frac{1}{\omega_{i}} \sum_{\pm}\mathcal{O} \left( \frac{1 }{ (\omega_{i} \pm \omega_{j} \pm \omega_{k})^{N}} \right) \)  \\ \hline 
\multirow{1}{*}{\quad \(\displaystyle \omega_{0} \omega_{i}\overline{\mathbb{H}}_{0jki}    \)} 
 &  \(\displaystyle ~~ \frac{e_{0}^{\prime}}{\tan}  \)  
 & \quad \quad \quad \text{None} & \quad \quad \quad $\mathcal{O}\left(\omega_{i}\omega_{k} \right)$
  & \quad \(\displaystyle \sum_{\pm}\mathcal{O} \left( \frac{\omega_{i} }{ (\omega_{i} \pm \omega_{j} \pm \omega_{k})^{N}}\right) \)  \\ \hline   
\end{longtable}
\end{center}	
\end{prop}
\begin{proof}
All these estimates follow directly from Lemma \ref{byparts2} and in particular from 
\begin{align*}
\int_{0}^{\frac{\pi}{2}} F(x) \cos( (2 b+1) x )  dx   = 
  \sum_{k=0}^{N} \frac{(-1)^{k+b}}{a^{2k+1}} F^{(2k)} \left(  \frac{\pi}{2} \right) 
+ \sum_{k=0}^{N} \frac{(-1)^{k+1}}{a^{2k+2}} F^{(2k+1)} \left( 0 \right) 
  + \mathcal{O} \left(  \frac{1}{a^{2N+2}} \right), 
\end{align*}
as $b \longrightarrow \infty$. However, we illustrate the proof only for the first constant, namely $\overline{\mathbb{I}}_{jki}$. For large $i$, we have
\begin{align*}
	\int_{x}^{\frac{\pi}{2}} e_{i}(y) \sin(y) \cos(y) d y & \simeq 
	\int_{x}^{\frac{\pi}{2}} \sin(\omega_{i} y)  \cos^2(y) d y  = -\frac{2}{\omega_{i}(\omega_{i}^2-4)} \cos(\omega_{i}x)
	+\frac{\omega_{i}\cos^2(x)}{\omega_{i}^2-4} \cos(\omega_{i}x)\\
	&
	+ \frac{\sin(2x)}{\omega_{i}^2-4} \sin(\omega_{i}x) 
	\simeq \frac{1}{\omega_{i}} \cos^2(x) \cos(\omega_{i}x).
\end{align*}
Now, for large values of $i,j,k$ and in the case where all $\omega_{i} \pm \omega_{j} \pm \omega_{k} \longrightarrow \infty$,
	\begin{align*}
		\overline{\mathbb{I}}_{jki}  &:= \int_{0}^{\frac{\pi}{2}}  e_{j} (x) e_{k}(x) \left( \int_{x}^{\frac{\pi}{2}} e_{i}(y) \sin(y) \cos(y) d y \right) \tan^2(x) dx \\
		& \simeq \frac{1}{\omega_{i}} \int_{0}^{\frac{\pi}{2}}  \cos^2(x) \cos(\omega_{i}x)\sin(\omega_{j}x)\sin(\omega_{k}x) dx \\
		& =  \frac{1}{\omega_{i}} \Bigg( \int_{0}^{\frac{\pi}{2}}  \cos^2(x) \cos \left( \left( \omega_{i} +\omega_{j} - \omega_{k}   \right) x \right) dx -
		\int_{0}^{\frac{\pi}{2}}  \cos^2(x) \cos \left( \left( \omega_{i} + \omega_{j} + \omega_{k}   \right) x \right) dx \\
		& + \int_{0}^{\frac{\pi}{2}}  \cos^2(x) \cos \left( \left( \omega_{i} - \omega_{j} - \omega_{k}   \right) x \right) dx
		- \int_{0}^{\frac{\pi}{2}}  \cos^2(x) \cos \left( \left( \omega_{i} - \omega_{j} + \omega_{k}   \right) x \right) dx
		\Bigg).
		\end{align*}
Observe that all $\omega_{i} \pm \omega_{j} \pm \omega_{k}$,
\begin{align*}
	\omega_{i} - \omega_{j} + \omega_{k} & = 2(i-j+k+1)+1, \quad 
	\omega_{i} - \omega_{j} - \omega_{k}   = 2(i-j-k-2)+1, \\
	\omega_{i} + \omega_{j} + \omega_{k} & = 2(i+j+k+4)+1, \quad
	\omega_{i} + \omega_{j} - \omega_{k}   = 2(i+j-k+1)+1.
\end{align*}
We define $F(x):= \cos^2(x) $ and compute
\begin{align*}
	F \left( \frac{\pi}{2} \right) = F^{\prime} \left(  0 \right) =0, ~~~F^{\prime \prime} \left(  \frac{\pi}{2} \right) = 2 \neq 0.
\end{align*}	
Lemma \ref{byparts2} yields 
\begin{align*}
	\overline{\mathbb{I}}_{jki} = \frac{1}{\omega_{i}} \sum_{\pm} \mathcal{O}\left( \frac{1}{(\omega_{i} \pm \omega_{j} \pm \omega_{k})^3} \right),
\end{align*}
as $i,j,k \longrightarrow \infty$, whereas, for large values of $i,j,k$ such that some $\omega_{i} \pm \omega_{j} \pm \omega_{k} \longarrownot\longrightarrow \infty$,  Holder's inequality implies 
\begin{align*}
	\left| \overline{\mathbb{I}}_{jki} \right|& :=\left| \int_{0}^{\frac{\pi}{2}}  e_{j} (x) e_{k}(x) \left( \int_{x}^{\frac{\pi}{2}} e_{i}(y) \sin(y) \cos(y) d y \right) \tan^2(x) dx \right| \\ & \leq 	
	\left\|
	\int_{\cdot}^{\frac{\pi}{2}} e_{i}(y) \sin(y) \cos(y) d y
	\right \|_{L^{\infty}\left[0,\frac{\pi}{2}\right]}
	\left \| 
	e_{j}\tan
	\right \|_{L^{2}\left[0,\frac{\pi}{2}\right]}
	\left \| 
	e_{k}\tan
	\right \|_{L^{2}\left[0,\frac{\pi}{2}\right]} 	\lesssim \omega_{i}^{-1}.
\end{align*}
In conclusion,
\begin{align*}
	\overline{\mathbb{I}}_{jki} =
 \Bigg \{ 
    \begin{array}{lr}
       \frac{1}{\omega_{i}} \sum_{\pm} \mathcal{O}\left( (\omega_{i} \pm \omega_{j} \pm \omega_{k})^{-3} \right), & \text{if all } \omega_{i} \pm \omega_{j} \pm \omega_{m} \longrightarrow \infty \\
      \mathcal{O} \left(\omega_{i}^{-1}\right), & \text{if some }\omega_{i} \pm \omega_{j} \pm \omega_{k} \longarrownot\longrightarrow \infty,
    \end{array}
\end{align*}
that completes the proof.
\end{proof}
\subsubsection{Fourier constants in $\mathcal{B}$}
We write
\begin{align*}
	\mathcal{B} = \mathcal{B}_{1} \cup \mathcal{B}_{2}
\end{align*}
where
\begin{align*}
	  \mathcal{B}_{1}:=\left \{ \omega_{i}\mathbb{H}_{0jki},\omega_{i}\mathbb{G}_{jki}, \omega_{i}\overline{\mathbb{G}}_{jki}\right \}, \quad 
	 \mathcal{B}_{2}:=\left \{\frac{\mathbb{R}_{klji}}{\omega_{i}},\frac{\mathbb{S}_{jkli}}{\omega_{i}},\omega_{i}\mathbb{H}_{klji},\omega_{i}\overline{\mathbb{H}}_{ijkl}, \overline{\mathbb{R}}_{klji},\overline{\mathbb{S}}_{jkli} \right \}
\end{align*}
\begin{prop}[Fourier constants in $\mathcal{B}_{1}$]\label{B1}
	Let $N\in \mathbb{N}$. The following growth and decay estimates hold. 
\begin{center}
\begin{longtable}{ |l|l|l|}
\hline
\multicolumn{3}{ |c| }{Growth and decay estimates for the Fourier constants in $\mathcal{B}_{1}$  as $i,j,k \longrightarrow + \infty$} \\
\hline
\quad Constant  & $\exists~ \omega_{i} \pm \omega_{j} \pm \omega_{k} \longarrownot\longrightarrow \infty $ & \quad \quad \quad \quad$\forall~ \omega_{i} \pm \omega_{j} \pm \omega_{k} \longrightarrow \infty $  \\ \hline
\multirow{1}{*}{ \(\displaystyle~ \quad \omega_{i}\mathbb{H}_{0jki}  \)}
 & \quad \quad\quad $\mathcal{O}\left(\omega_{i}\omega_{k} \right) $
  &\quad \quad \quad\(\displaystyle  \sum_{\pm} \mathcal{O} \left( \frac{\omega_{i} }{ (\omega_{i} \pm \omega_{j} \pm \omega_{k})^{N}}\right) \)  \\ \hline 
\multirow{1}{*}{ \(\displaystyle \omega_{i}\mathbb{G}_{jki} ,\omega_{i}\overline{\mathbb{G}}_{jki} \)}
 & \quad \quad\quad $\mathcal{O}\left(\omega_{i}\omega_{k} \right) $
  &\quad \quad \quad \(\displaystyle \sum_{\pm}
  \mathcal{O} \left(
   \frac{\omega_{i}}{(\omega_{i} \pm \omega_{j} \pm  \omega_{k} )^2} \right) \)  \\ \hline 
\end{longtable}
\end{center}
\end{prop}
\begin{proof}
First, observe that 
\begin{align*}
	\omega_{i} - \omega_{j} + \omega_{k} & = 2(i-j+k+1)+1, \quad
	\omega_{i} - \omega_{j} - \omega_{k}   = 2(i-j-k-2)+1, \\
	\omega_{i} + \omega_{j} + \omega_{k} & = 2(i+j+k+4)+1,  \quad
	\omega_{i} + \omega_{j} - \omega_{k}   = 2(i+j-k+1)+1,
\end{align*}
are all odd. All results here follow from Lemma \ref{OscillatoryIntegrals} and Remark \ref{e0oscilating}. For large values of $i,j,k$ and in the case where all $\omega_{i} \pm \omega_{j} \pm \omega_{k} \longrightarrow \infty$,
\begin{align*}
	\mathbb{H}_{0jki} &:= \int_{0}^{\frac{\pi}{2}} e_{0}(x) e_{i}(x) e_{j}(x) e_{k}(x) \tan^2(x)  dx 
	\simeq \int_{0}^{\frac{\pi}{2}} e_{0}(x) \frac{\sin(\omega_{i}x) \sin(\omega_{j}x) \sin(\omega_{k}x) }{\tan(x)} dx \\
	& = \frac{1}{4} \int_{0}^{\frac{\pi}{2}} e_{0}(x)\frac{\sin\left( \left( \omega_{i} - \omega_{j} + \omega_{k} \right)x \right) }{\tan(x)} dx - \frac{1}{4} \int_{0}^{\frac{\pi}{2}}e_{0}(x)  \frac{\sin\left( \left( \omega_{i} - \omega_{j} - \omega_{k} \right)x \right) }{\tan(x)} dx \\
	& - \frac{1}{4} \int_{0}^{\frac{\pi}{2}} e_{0}(x)\frac{\sin\left( \left( \omega_{i} + \omega_{j} + \omega_{k} \right)x \right) }{\tan(x)} dx + \frac{1}{4} \int_{0}^{\frac{\pi}{2}} e_{0}(x) \frac{\sin\left( \left( \omega_{i} + \omega_{j} - \omega_{k} \right)x \right) }{\tan(x)} dx.
\end{align*}
By Remark \ref{e0oscilating}, we infer that in this case
\begin{align*}
	\mathbb{H}_{0jki} & = \frac{1}{4}
	\left(  2 \left(2\sqrt{2 \pi}-2\sqrt{2 \pi}\right) + \sum_{\pm } \mathcal{O} \left( \frac{1}{\left( \omega_{i} \pm \omega_{j} \pm \omega_{k} \right)^{N}} \right) \right) = 
	\sum_{\pm } \mathcal{O} \left( \frac{1}{\left( \omega_{i} \pm \omega_{j} \pm \omega_{k} \right)^{N}} \right),
\end{align*}
as $i,j,k \longrightarrow \infty$, whereas, for large values of $i,j,k$ such that some $\omega_{i} \pm \omega_{j} \pm \omega_{k} \longarrownot\longrightarrow \infty$,  Holder's inequality implies
\begin{align*}
	\left| \mathbb{H}_{0jki}  \right|& :=\left| \int_{0}^{\frac{\pi}{2}} e_{0}(x) e_{i}(x) e_{j}(x) e_{k}(x) \tan^2(x)  dx \right| \\ & \leq 	
	\left\|
	e_{0} 
	\right \|_{L^{\infty}\left[0,\frac{\pi}{2}\right]}
	\left \| 
	e_{k}
	\right \|_{L^{\infty}\left[0,\frac{\pi}{2}\right]}
	\left \| 
	e_{i}\tan
	\right \|_{L^{2}\left[0,\frac{\pi}{2}\right]}
	\left \|
	e_{j}\tan
	\right \|_{L^{2}\left[0,\frac{\pi}{2}\right]} 	\lesssim \omega_{k}.
\end{align*}
Second, for large values of $i,j,k$ and in the case where all $\omega_{i} \pm \omega_{j} \pm \omega_{k} \longrightarrow \infty$, we have
\begin{align*}
	\mathbb{G}_{jki} &:= \int_{0}^{\frac{\pi}{2}} e_{i}(x) e_{j}(x) e_{k}(x) \tan^2(x)  dx 
	 \simeq \int_{0}^{\frac{\pi}{2}} \frac{\sin(\omega_{i}x) \sin(\omega_{j}x) \sin(\omega_{k}x) }{\tan(x)} dx \\
	& = \frac{1}{4} \int_{0}^{\frac{\pi}{2}} \frac{\sin\left( \left( \omega_{i} - \omega_{j} + \omega_{k} \right)x \right) }{\tan(x)} dx- \frac{1}{4} \int_{0}^{\frac{\pi}{2}} \frac{\sin\left( \left( \omega_{i} - \omega_{j} - \omega_{k} \right)x \right) }{\tan(x)} dx \\
	& - \frac{1}{4} \int_{0}^{\frac{\pi}{2}} \frac{\sin\left( \left( \omega_{i} + \omega_{j} + \omega_{k} \right)x \right) }{\tan(x)} dx + \frac{1}{4} \int_{0}^{\frac{\pi}{2}} \frac{\sin\left( \left( \omega_{i} + \omega_{j} - \omega_{k} \right)x \right) }{\tan(x)} dx.
\end{align*} 
Hence, by Lemma \ref{OscillatoryIntegrals}, we get that in this case
\begin{align*}
	\mathbb{G}_{jki}  = \frac{1}{4} \left( 2 \left( \frac{\pi}{2} - \frac{\pi}{2} \right)+ \sum_{\pm} \mathcal{O} \left( \frac{1}{(\omega_{i} \pm \omega_{j} \pm  \omega_{k} )^2} \right) \right)
	= \sum_{\pm} \mathcal{O} \left( \frac{1}{(\omega_{i} \pm \omega_{j} \pm  \omega_{k} )^2} \right),
\end{align*}
as $i,j,k \longrightarrow \infty$. On the other hand, for large values of $i,j,k$ such that some $\omega_{i} \pm \omega_{j} \pm \omega_{k} \longarrownot\longrightarrow \infty$,  Holder's inequality implies
\begin{align*}
	\left| \mathbb{G}_{jki} \right|& :=\left| \int_{0}^{\frac{\pi}{2}} e_{i}(x) e_{j}(x) e_{k}(x) \tan^2(x)  dx \right|  \leq 	
	\left \| 
	e_{k}
	\right \|_{L^{\infty}\left[0,\frac{\pi}{2}\right]}
	\left \| 
	e_{i}\tan
	\right \|_{L^{2}\left[0,\frac{\pi}{2}\right]}
	\left \|
	e_{j}\tan
	\right \|_{L^{2}\left[0,\frac{\pi}{2}\right]} 	\lesssim \omega_{k}.
\end{align*}
Similarly, for large values of $i,j,k$ and in the case where all $\omega_{i} \pm \omega_{j} \pm \omega_{k} \longrightarrow \infty$, 
\begin{align*}
	\overline{\mathbb{G}}_{jki} &:= \int_{0}^{\frac{\pi}{2}} e_{k}(x) \frac{e_{j}^{\prime}(x)}{\omega_{j}} \frac{e_{i}^{\prime}(x)}{\omega_{i}} \tan^2(x)  dx 
	 \simeq \int_{0}^{\frac{\pi}{2}}  \frac{\sin(\omega_{k}x) \cos(\omega_{j}x) \cos(\omega_{i}x) }{\tan(x)} dx \\
	& = \frac{1}{4} \int_{0}^{\frac{\pi}{2}} \frac{\sin\left( \left( \omega_{i} - \omega_{j} + \omega_{k} \right)x \right) }{\tan(x)} dx - \frac{1}{4} \int_{0}^{\frac{\pi}{2}}  \frac{\sin\left( \left( \omega_{i} - \omega_{j} - \omega_{k} \right)x \right) }{\tan(x)} dx \\
	& + \frac{1}{4} \int_{0}^{\frac{\pi}{2}} \frac{\sin\left( \left( \omega_{i} + \omega_{j} + \omega_{k} \right)x \right) }{\tan(x)} dx - \frac{1}{4} \int_{0}^{\frac{\pi}{2}}  \frac{\sin\left( \left( \omega_{i} + \omega_{j} - \omega_{k} \right)x \right) }{\tan(x)} dx
\end{align*}
and finally
\begin{align*}
	\overline{\mathbb{G}}_{jki}  = \frac{1}{4} \left( 2 \left( \frac{\pi}{2} - \frac{\pi}{2} \right)+ \sum_{\pm} \mathcal{O} \left( \frac{1}{(\omega_{i} \pm \omega_{j} \pm  \omega_{k} )^2} \right) \right)
	= \sum_{\pm} \mathcal{O} \left( \frac{1}{(\omega_{i} \pm \omega_{j} \pm  \omega_{k} )^2} \right),
\end{align*}
as $i,j,k \longrightarrow \infty$. However, for large values of $i,j,k$ such that some $\omega_{i} \pm \omega_{j} \pm \omega_{k} \longarrownot\longrightarrow \infty$,  Holder's inequality implies
\begin{align*}
	\left| \overline{\mathbb{G}}_{jki}  \right|& :=\left| \int_{0}^{\frac{\pi}{2}} e_{k}(x) \frac{e_{j}^{\prime}(x)}{\omega_{j}} \frac{e_{i}^{\prime}(x)}{\omega_{i}} \tan^2(x)  dx \right|   \leq 	
	\left \| 
	e_{k}
	\right \|_{L^{\infty}\left[0,\frac{\pi}{2}\right]}
	\left \| 
	\frac{e^{\prime}_{i}}{\omega_{i}}\tan
	\right \|_{L^{2}\left[0,\frac{\pi}{2}\right]}
	\left \|
	\frac{e^{\prime}_{j}}{\omega_{j}}\tan
	\right \|_{L^{2}\left[0,\frac{\pi}{2}\right]} 	\lesssim \omega_{k},
\end{align*}
that completes the proof.
\end{proof}
\begin{prop}[Fourier constants in $\mathcal{B}_{2}$]\label{B2}
	Let $N\in \mathbb{N}$. The following growth and decay estimates hold. 
\begin{center}
\begin{longtable}{ |l|l|l|}
\hline
\multicolumn{3}{ |c| }{Growth and decay estimates for the Fourier constants in $\mathcal{B}_{2}$  as $i,j,k,l \longrightarrow + \infty$} \\
\hline
\quad  Constant  & $\exists~ \omega_{i} \pm \omega_{j} \pm \omega_{k} \pm \omega_{l} \longarrownot\longrightarrow \infty $ & \quad \quad $ \forall~\omega_{i} \pm \omega_{j} \pm \omega_{k} \pm \omega_{l} \longrightarrow \infty $  \\ \hline
\multirow{1}{*}{ \(\displaystyle \quad \frac{\mathbb{R}_{klji}}{\omega_{i}}, \frac{\mathbb{S}_{jkli}}{\omega_{i}}  \)} & \quad \quad\quad\quad \(\displaystyle 
\mathcal{O} \left(
\omega_{j}\omega_{i}^{-1}\right) \)  
  &\quad \(\displaystyle \frac{1}{\omega_{i}} \sum_{\pm} \mathcal{O} \left( \frac{1 }{ (\omega_{i} \pm \omega_{j} \pm \omega_{k} \pm \omega_{l})^{N}}\right) \)  \\ \hline 
\multirow{1}{*}{ \(\displaystyle \omega_{i}\mathbb{H}_{klji},\omega_{i}\overline{\mathbb{H}}_{ijkl}  \)}
  & \quad \quad\quad\quad \(\displaystyle \mathcal{O} \left(\omega_{i} \omega_{j} \omega_{k}\right) \) 
  &\quad \(\displaystyle \quad \sum_{\pm} \mathcal{O} \left( \frac{\omega_{i} }{ (\omega_{i} \pm \omega_{j} \pm \omega_{k} \pm \omega_{l})^{3}}\right) \)  \\ \hline 
\multirow{1}{*}{ \(\displaystyle \quad \overline{\mathbb{R}}_{klji},\overline{\mathbb{S}}_{jkli}   \)}
 & \quad \quad \quad\quad \(\displaystyle 
 \mathcal{O} \left(
\omega_{j}\omega_{i}^{-1} \right)  \)
  &\quad \(\displaystyle \frac{1}{\omega_{i}} \sum_{\pm}
   \mathcal{O} \left( \frac{1 }{ (\omega_{i} \pm \omega_{j} \pm \omega_{k} \pm \omega_{l})^{N}} \right)\)  \\ \hline 
\end{longtable}
\end{center}
\end{prop}
\begin{proof}
First, observe that 
\begin{align*}
	\omega_{i} - \omega_{j} + \omega_{k} - \omega_{l} & = 2 (i - j + k - l),\quad 
	\omega_{i} + \omega_{j} - \omega_{k} - \omega_{l}   = 2 (i + j - k - l), \\
	\omega_{i} - \omega_{j} - \omega_{k} - \omega_{l}& = 2 (-3 + i - j - k - l), \quad
	\omega_{i} + \omega_{j} + \omega_{k} - \omega_{l}   = 2 (3 + i + j + k - l),\\
	\omega_{i} - \omega_{j} - \omega_{k} + \omega_{l} & = 2 (i - j - k + l),\quad 
	\omega_{i} + \omega_{j} + \omega_{k} + \omega_{l}   = 2 (6 + i + j + k + l), \\
	\omega_{i} - \omega_{j} + \omega_{k} + \omega_{l} & = 2 (3 + i - j + k + l), \quad
	\omega_{i} + \omega_{j} - \omega_{k} + \omega_{l}   = 2 (3 + i + j - k + l), 
\end{align*}
are all even. All results here follow from Lemma \ref{OscillatoryIntegrals}. For large values of $i,j,k,l$ and in the case where all $\omega_{i} \pm \omega_{j} \pm \omega_{k}\pm \omega_{l} \longrightarrow \infty$, 
	\begin{align*}
	 \mathbb{R}_{klji}  &:=  \int_{0}^{\frac{\pi}{2}} e_{j}(x) \frac{e_{k}^{\prime}(x)}{\omega_{k}} \frac{e_{l}^{\prime}(x)}{\omega_{l}} \frac{e_{i}^{\prime}(x)}{\omega_{i}} \frac{\sin^3(x)}{\cos(x)}  dx 
	  \simeq  \int_{0}^{\frac{\pi}{2}}  \cos^2(x) \frac{ \cos(\omega_{i}x)  \sin(\omega_{j}x) \cos(\omega_{k}x) \cos(\omega_{l}x) }{\tan(x)} dx \\
	& =   
	\int_{0}^{\frac{\pi}{2}}  \cos^2(x) \frac{ \sin \left(  \left( \omega_{i} + \omega_{j} + \omega_{k} + \omega_{l}   \right) x \right) }{\tan(x)} dx 
	 -  \int_{0}^{\frac{\pi}{2}}  \cos^2(x) \frac{ \sin \left(  \left( \omega_{i} - \omega_{j} + \omega_{k} + \omega_{l}   \right) x \right) }{\tan(x)} dx \\
	& + 	\int_{0}^{\frac{\pi}{2}}  \cos^2(x) \frac{  \sin \left(  \left( \omega_{i} + \omega_{j} + \omega_{k} - \omega_{l}   \right) x \right) }{\tan(x)} dx 
	-  \int_{0}^{\frac{\pi}{2}}  \cos^2(x) \frac{  \sin \left(  \left( \omega_{i} - \omega_{j} + \omega_{k} - \omega_{l}   \right) x \right) }{\tan(x)} dx \\
	& + 	\int_{0}^{\frac{\pi}{2}}  \cos^2(x) \frac{  \sin \left(  \left( \omega_{i} + \omega_{j} - \omega_{k} + \omega_{l}   \right) x \right) }{\tan(x)} dx 
	-  \int_{0}^{\frac{\pi}{2}}  \cos^2(x) \frac{  \sin \left(  \left( \omega_{i} - \omega_{j} - \omega_{k} + \omega_{l}   \right) x \right )}{\tan(x)} dx \\
	& + 	\int_{0}^{\frac{\pi}{2}}  \cos^2(x) \frac{  \sin \left(  \left( \omega_{i} + \omega_{j} - \omega_{k} - \omega_{l}   \right) x \right)) }{\tan(x)} dx 
	-  \int_{0}^{\frac{\pi}{2}}  \cos^2(x) \frac{  \sin \left(  \left( \omega_{i} - \omega_{j} - \omega_{k} - \omega_{l}   \right) x \right) }{\tan(x)} dx \\
&=  \left( 4 \left(  \frac{\pi}{2} - \frac{\pi}{2} \right) + \sum_{\pm} \mathcal{O} \left( \frac{1}{\left( \omega_{i} - \omega_{j} - \omega_{k} + \omega_{l}   \right)^{N}}
	 \right)  \right) 
	  = \sum_{\pm} \mathcal{O} \left( \frac{1}{\left( \omega_{i} - \omega_{j} - \omega_{k} + \omega_{l}   \right)^{N}} \right).		  
\end{align*}
On the other hand, in the case where some $\omega_{i} \pm \omega_{j} \pm \omega_{m} \longarrownot\longrightarrow \infty$,  Holder's inequality implies
\begin{align*}
	\left| \frac{\mathbb{R}_{klji}}{\omega_{i}} \right| & =
	\frac{1}{\omega_{i}} \left| 
	 \int_{0}^{\frac{\pi}{2}} e_{j}(x) \cos(x) \sin(x)\frac{e_{k}^{\prime}(x)}{\omega_{k}} \frac{e_{l}^{\prime}(x)}{\omega_{l}} \frac{e_{i}^{\prime}(x)}{\omega_{i}} \tan^2(x)  dx
	\right| \\
	& \leq \frac{1}{\omega_{i}}
	\left \| 
	e_{j}
	\right\|_{L^{\infty}\left[0,\frac{\pi}{2}\right]}
	 \left \| 
	\cos \sin \frac{e_{k}^{\prime}}{\omega_{k}}
	\right \|_{L^{\infty}\left[0,\frac{\pi}{2}\right]}	
	\left \| 
	\frac{e_{l}^{\prime}}{\omega_{l}}\tan
	\right \|_{L^{2}\left[0,\frac{\pi}{2}\right]}
	\left \| 
	\frac{e_{i}^{\prime}}{\omega_{i}}\tan
	\right \|_{L^{2}\left[0,\frac{\pi}{2}\right]} 	\leq \frac{\omega_{j}}{\omega_{i}},
\end{align*}
where we used the $L^{\infty}$ bounds as well as the orthogonality from Lemma \ref{Linftyboundse}. Furthermore, for large values of $i,j,k,l$ and in the case where all $\omega_{i} \pm \omega_{j} \pm \omega_{k}\pm \omega_{l} \longrightarrow \infty$,   
	\begin{align*}
	\mathbb{S}_{jkli} &:=  \int_{0}^{\frac{\pi}{2}} e_{j}(x) e_{k}(x) e_{l} (x) \frac{e_{i}^{\prime}(x)}{\omega_{i}} \frac{\sin^3(x)}{\cos(x)}  dx 
	 \simeq  \int_{0}^{\frac{\pi}{2}}  \cos^2(x) \frac{ \cos(\omega_{i}x)  \sin(\omega_{j}x) \sin(\omega_{k}x) \sin(\omega_{l}x) }{\tan(x)} dx \\
	& =   -
	\int_{0}^{\frac{\pi}{2}}  \cos^2(x) \frac{ \sin \left(  \left( \omega_{i} + \omega_{j} + \omega_{k} + \omega_{l}   \right) x \right) }{\tan(x)} dx 
	 + \int_{0}^{\frac{\pi}{2}}  \cos^2(x) \frac{ \sin \left(  \left( \omega_{i} - \omega_{j} + \omega_{k} + \omega_{l}   \right) x \right) }{\tan(x)} dx \\
	& +	\int_{0}^{\frac{\pi}{2}}  \cos^2(x) \frac{  \sin \left(  \left( \omega_{i} + \omega_{j} + \omega_{k} - \omega_{l}   \right) x \right) }{\tan(x)} dx 
	-  \int_{0}^{\frac{\pi}{2}}  \cos^2(x) \frac{  \sin \left(  \left( \omega_{i} - \omega_{j} + \omega_{k} - \omega_{l}   \right) x \right) }{\tan(x)} dx \\
	& + 	\int_{0}^{\frac{\pi}{2}}  \cos^2(x) \frac{  \sin \left(  \left( \omega_{i} + \omega_{j} - \omega_{k} + \omega_{l}   \right) x \right) }{\tan(x)} dx 
	-  \int_{0}^{\frac{\pi}{2}}  \cos^2(x) \frac{  \sin \left(  \left( \omega_{i} - \omega_{j} - \omega_{k} + \omega_{l}   \right) x \right) }{\tan(x)} dx \\
	& -	\int_{0}^{\frac{\pi}{2}}  \cos^2(x) \frac{  \sin \left(  \left( \omega_{i} + \omega_{j} - \omega_{k} - \omega_{l}   \right) x \right) }{\tan(x)} dx 
	+ \int_{0}^{\frac{\pi}{2}}  \cos^2(x) \frac{  \sin \left(  \left( \omega_{i} - \omega_{j} - \omega_{k} - \omega_{l}   \right) x \right) }{\tan(x)} dx \\
 &=   4 \left(  \frac{\pi}{2} - \frac{\pi}{2} \right) + \sum_{\pm} \mathcal{O} \left( \frac{1}{\left( \omega_{i} - \omega_{j} - \omega_{k} + \omega_{l}   \right)^{N}}
	 \right) 
	 = \sum_{\pm} \mathcal{O} \left( \frac{1}{\left( \omega_{i} - \omega_{j} - \omega_{k} + \omega_{l}   \right)^{N}} \right).
\end{align*}
On the other hand, in the case where some $\omega_{i} \pm \omega_{j} \pm \omega_{k}\pm \omega_{l} \longarrownot\longrightarrow \infty$, Holder's inequality implies
\begin{align*}
	\left| \frac{\mathbb{S}_{jkli}}{\omega_{i}} \right| & =
	\frac{1}{\omega_{i}} \left| 
	 \int_{0}^{\frac{\pi}{2}} e_{j}(x) \cos(x) \sin(x)\frac{e_{i}^{\prime}(x)}{\omega_{i}} e_{k}(x) e_{l}(x) \tan^2(x)  dx
	\right| \\
	& \leq \frac{1}{\omega_{i}}
	\left \| 
	e_{j}
	\right\|_{L^{\infty}\left[0,\frac{\pi}{2}\right]}
	 \left \| 
	\cos \sin \frac{e_{i}^{\prime}}{\omega_{i}}
	\right \|_{L^{\infty}\left[0,\frac{\pi}{2}\right]}	
	\left \| 
	e_{k}\tan
	\right \|_{L^{2}\left[0,\frac{\pi}{2}\right]}
	\left \| 
	e_{l}\tan
	\right \|_{L^{2}\left[0,\frac{\pi}{2}\right]}
		\leq \frac{\omega_{j}}{\omega_{i}},
\end{align*}
where we used the $L^{\infty}$ bounds and the orthogonality from Lemma \ref{Linftyboundse}. Similarly, for large values of $i,j,k,l$ and in the case where all $\omega_{i} \pm \omega_{j} \pm \omega_{k}\pm \omega_{l} \longrightarrow \infty$,   
\begin{align*}
 \mathbb{H}_{klji} & := \int_{0}^{\frac{\pi}{2}} e_{i}(x) e_{j}(x) e_{k}(x) e_{l}(x) \tan^{2}(x) dx 
  \simeq   \int_{0}^{\frac{\pi}{2}} \frac{\sin(\omega_{i}x) \sin(\omega_{j}x) \sin(\omega_{k}x) \sin(\omega_{l}x) }{\tan^2(x)} dx \\
& \simeq \frac{1}{8} \int_{ \frac{2 \pi}{ \omega_{i} -\omega_{j} + \omega_{k} - \omega_{l} }  }^{\frac{\pi}{2}} \frac{\cos\left( 
\left(\omega_{i} -\omega_{j} + \omega_{k} - \omega_{l} \right)x \right)}{\tan^2(x)} dx 
+\frac{1}{8} \int_{ \frac{2 \pi}{ \omega_{i} + \omega_{j} - \omega_{k} - \omega_{l} }}^{\frac{\pi}{2} } \frac{\cos\left( 
\left(\omega_{i} + \omega_{j} - \omega_{k} - \omega_{l} \right)x \right)}{\tan^2(x)} dx \\
& - \frac{1}{8} \int_{ \frac{2 \pi}{ \omega_{i} -\omega_{j} - \omega_{k} - \omega_{l}}}^{\frac{\pi}{2}} \frac{\cos\left( 
\left(\omega_{i} -\omega_{j} - \omega_{k} - \omega_{l} \right)x \right)}{\tan^2(x)} dx 
-\frac{1}{8} \int_{\frac{2 \pi}{ \omega_{i} + \omega_{j} + \omega_{k} - \omega_{l}}}^{\frac{\pi}{2}} \frac{\cos\left( 
\left(\omega_{i} + \omega_{j} + \omega_{k} - \omega_{l} \right)x \right)}{\tan^2(x)} dx \\
& +\frac{1}{8} \int_{\frac{2 \pi}{ \omega_{i} -\omega_{j} - \omega_{k} + \omega_{l}}}^{\frac{\pi}{2}} \frac{\cos\left( 
\left(\omega_{i} -\omega_{j} - \omega_{k} + \omega_{l} \right)x \right)}{\tan^2(x)} dx 
+\frac{1}{8} \int_{\frac{2 \pi}{ \omega_{i} + \omega_{j} + \omega_{k} + \omega_{l}}}^{\frac{\pi}{2}} \frac{\cos\left( 
\left(\omega_{i} + \omega_{j} + \omega_{k}  + \omega_{l} \right)x \right)}{\tan^2(x)} dx \\
& -\frac{1}{8} \int_{\frac{2 \pi}{ \omega_{i} -\omega_{j} + \omega_{k} + \omega_{l}}}^{\frac{\pi}{2}} \frac{\cos\left( 
\left(\omega_{i} -\omega_{j} + \omega_{k} +  \omega_{l} \right)x \right)}{\tan^2(x)} dx 
-\frac{1}{8} \int_{\frac{2 \pi}{ \omega_{i} + \omega_{j} - \omega_{k} + \omega_{l}}}^{\frac{\pi}{2}} \frac{\cos\left( 
\left(\omega_{i} + \omega_{j} - \omega_{k} + \omega_{l} \right)x \right)}{\tan^2(x)} dx \\
& = \frac{c}{16 \pi} \Big(
+(\omega_{i} -\omega_{j} + \omega_{k} - \omega_{l})
+( \omega_{i} + \omega_{j} - \omega_{k} - \omega_{l} ) 
-( \omega_{i} -\omega_{j} - \omega_{k} - \omega_{l} ) \\
& -( \omega_{i} +\omega_{j} + \omega_{k} - \omega_{l} ) 
+( \omega_{i} -\omega_{j} - \omega_{k} + \omega_{l} )
+( \omega_{i} +\omega_{j} + \omega_{k} + \omega_{l} ) \\
& -( \omega_{i} -\omega_{j} + \omega_{k} + \omega_{l} )
-( \omega_{i} + \omega_{j} - \omega_{k} + \omega_{l} ) \Big)
+ \sum_{\pm} \mathcal{O} \left( \frac{1}{ ( \omega_{i} \pm \omega_{j} \pm  \omega_{k} \pm \omega_{l} )^{3}  } \right) \\
&= \sum_{\pm} \mathcal{O} \left( \frac{1}{ ( \omega_{i} \pm \omega_{j} \pm  \omega_{k} \pm \omega_{l} )^{3}  } \right).
\end{align*}
On the other hand, in the case where some $\omega_{i} \pm \omega_{j} \pm \omega_{k}\pm \omega_{l} \longarrownot\longrightarrow \infty$, Holder's inequality implies
\begin{align*}
	\left| \mathbb{H}_{klji} \right| & =
	\left| 
	 \int_{0}^{\frac{\pi}{2}} e_{i}(x) e_{j}(x) e_{k}(x) e_{l}(x) \tan^2(x)  dx
	\right| \\
	& \leq 
	\left \| 
	e_{j}
	\right\|_{L^{\infty}\left[0,\frac{\pi}{2}\right]}
	 \left \| 
	e_{k}
	\right \|_{L^{\infty}\left[0,\frac{\pi}{2}\right]}	
	\left \| 
	e_{i}\tan
	\right \|_{L^{2}\left[0,\frac{\pi}{2}\right]}
	\left \| 
	e_{l}\tan
	\right \|_{L^{2}\left[0,\frac{\pi}{2}\right]}
		\leq \omega_{j}\omega_{k}.
\end{align*}
Furthermore, for large values of $i,j,k,l$ and in the case where all $\omega_{i} \pm \omega_{j} \pm \omega_{k}\pm \omega_{l} \longrightarrow \infty$,    
\begin{align*}
	\overline{\mathbb{H}}_{ijkl}&:=\int_{0}^{\frac{\pi}{2}} e_{j}(x) e_{k}(x) \frac{e_{i}^{\prime}(x)}{\omega_{i}} \frac{e_{l}^{\prime}(x)}{\omega_{l}}\tan^{2}(x) dx 
	 \simeq   \int_{0}^{\frac{\pi}{2}} \frac{\sin(\omega_{j}x) \sin(\omega_{k}x) \cos(\omega_{i}x) \cos(\omega_{l}x) }{\tan^2(x)} dx \\
 & \simeq \frac{1}{8} \int_{ \frac{2 \pi}{ \omega_{i} +\omega_{j} - \omega_{k} - \omega_{l} }  }^{\frac{\pi}{2}} \frac{\cos\left( 
\left(\omega_{i} +\omega_{j} - \omega_{k} - \omega_{l} \right)x \right)}{\tan^2(x)} dx 
-\frac{1}{8} \int_{ \frac{2 \pi}{ \omega_{i} - \omega_{j} - \omega_{k} - \omega_{l} }}^{\frac{\pi}{2} } \frac{\cos\left( 
\left(\omega_{i} - \omega_{j} - \omega_{k} - \omega_{l} \right)x \right)}{\tan^2(x)} dx \\
& + \frac{1}{8} \int_{ \frac{2 \pi}{ \omega_{i} +\omega_{j} - \omega_{k} +\omega_{l}}}^{\frac{\pi}{2}} \frac{\cos\left( 
\left(\omega_{i} +\omega_{j} - \omega_{k} + \omega_{l} \right)x \right)}{\tan^2(x)} dx 
-\frac{1}{8} \int_{\frac{2 \pi}{ \omega_{i} - \omega_{j} - \omega_{k} + \omega_{l}}}^{\frac{\pi}{2}} \frac{\cos\left( 
\left(\omega_{i} - \omega_{j} - \omega_{k} + \omega_{l} \right)x \right)}{\tan^2(x)} dx \\
& -\frac{1}{8} \int_{\frac{2 \pi}{ \omega_{i} +\omega_{j} + \omega_{k} - \omega_{l}}}^{\frac{\pi}{2}} \frac{\cos\left( 
\left(\omega_{i} +\omega_{j} + \omega_{k} - \omega_{l} \right)x \right)}{\tan^2(x)} dx 
+\frac{1}{8} \int_{\frac{2 \pi}{ \omega_{i} - \omega_{j} + \omega_{k} - \omega_{l}}}^{\frac{\pi}{2}} \frac{\cos\left( 
\left(\omega_{i} - \omega_{j} + \omega_{k}  - \omega_{l} \right)x \right)}{\tan^2(x)} dx \\
& -\frac{1}{8} \int_{\frac{2 \pi}{ \omega_{i} +\omega_{j} + \omega_{k} + \omega_{l}}}^{\frac{\pi}{2}} \frac{\cos\left( 
\left(\omega_{i} +\omega_{j} + \omega_{k} + \omega_{l} \right)x \right)}{\tan^2(x)} dx 
+\frac{1}{8} \int_{\frac{2 \pi}{ \omega_{i} - \omega_{j} + \omega_{k} + \omega_{l}}}^{\frac{\pi}{2}} \frac{\cos\left( 
\left(\omega_{i} - \omega_{j} + \omega_{k}  + \omega_{l} \right)x \right)}{\tan^2(x)} dx
\end{align*}
and the rest of the proof coincides with the one above. On the other hand, if some $\omega_{i} \pm \omega_{j} \pm \omega_{k}\pm \omega_{l} \longarrownot\longrightarrow \infty$, Holder's inequality implies
\begin{align*}
	\left| \overline{\mathbb{H}}_{ijkl} \right| & =
	 \left| 
	 \int_{0}^{\frac{\pi}{2}} e_{j}(x) e_{k}(x) 
	 \frac{e_{i}^{\prime}(x)}{\omega_{i}}
	 \frac{e_{l}^{\prime}(x)}{\omega_{l}}
	  \tan^2(x)  dx
	\right| \\
	& \leq 
	\left \| 
	e_{j}
	\right\|_{L^{\infty}\left[0,\frac{\pi}{2}\right]}
	 \left \| 
	e_{k}
	\right \|_{L^{\infty}\left[0,\frac{\pi}{2}\right]}	
	\left \| 
	\frac{e_{i}^{\prime}}{\omega_{i}}\tan
	\right \|_{L^{2}\left[0,\frac{\pi}{2}\right]}
	\left \| 
	\frac{e_{l}^{\prime}}{\omega_{l}}\tan
	\right \|_{L^{2}\left[0,\frac{\pi}{2}\right]}
		\leq \omega_{j}\omega_{k}.
\end{align*} 
In addition, for large values of $i,j,k,l$ and in the case where all $\omega_{i} \pm \omega_{j} \pm \omega_{k}\pm \omega_{l} \longrightarrow \infty$, 
\begin{align*}
	\overline{\mathbb{R}}_{klji} & :=
\int_{0}^{\frac{\pi}{2}} e_{j}(x) \frac{e_{k}^{\prime}(x)}{\omega_{k}}  \frac{e_{l}^{\prime}(x)}{\omega_{l}} \left( \int_{x}^{\frac{\pi}{2}} e_{i}(y) \sin(y) \cos(y) d y \right) \tan^2(x) dx \\
& \simeq \frac{1}{\omega_{i}}  \int_{0}^{\frac{\pi}{2}} \cos^2(x) \frac{ \cos(\omega_{i}x)\sin(\omega_{j}x)\cos(\omega_{k}x)\cos(\omega_{l}x)  }{\tan(x)} dx \\
& \simeq \frac{1}{8 \omega_{i}} \Bigg( 
\int_{0}^{\frac{\pi}{2}} \cos^2(x) \frac{ \sin( (\omega_{i}+\omega_{j}+\omega_{k}+\omega_{l} )x)  }{\tan(x)} dx 
-\int_{0}^{\frac{\pi}{2}} \cos^2(x) \frac{ \sin( (\omega_{i}-\omega_{j}+\omega_{k}+\omega_{l} )x)  }{\tan(x)} dx \\
& \int_{0}^{\frac{\pi}{2}} \cos^2(x) \frac{ \sin( (\omega_{i}+\omega_{j}+\omega_{k}-\omega_{l} )x)  }{\tan(x)} dx
- \int_{0}^{\frac{\pi}{2}} \cos^2(x) \frac{ \sin( (\omega_{i}-\omega_{j}+\omega_{k}-\omega_{l} )x)  }{\tan(x)} dx \\
& \int_{0}^{\frac{\pi}{2}} \cos^2(x) \frac{ \sin( (\omega_{i}+\omega_{j}-\omega_{k}+\omega_{l} )x)  }{\tan(x)} dx
- \int_{0}^{\frac{\pi}{2}} \cos^2(x) \frac{ \sin( (\omega_{i}-\omega_{j}-\omega_{k}+\omega_{l} )x)  }{\tan(x)} dx \\
& \int_{0}^{\frac{\pi}{2}} \cos^2(x) \frac{ \sin( (\omega_{i}+\omega_{j}-\omega_{k}-\omega_{l} )x)  }{\tan(x)} dx
-\int_{0}^{\frac{\pi}{2}} \cos^2(x) \frac{ \sin( (\omega_{i}-\omega_{j}-\omega_{k}-\omega_{l} )x)  }{\tan(x)} dx
\Bigg) \\
& = \frac{1}{8 \omega_{i}} \left( 4 \left( \frac{\pi}{2} -\frac{\pi}{2}  \right) + \sum_{\pm} \mathcal{O}\left( \frac{1}{(\omega_{i} \pm \omega_{j} \pm \omega_{k} \pm \omega_{l})^{N}} \right) \right) \\
		 &
 = \frac{1}{ \omega_{i}}\sum_{\pm} \mathcal{O}\left( \frac{1}{(\omega_{i} \pm \omega_{j} \pm \omega_{k} \pm \omega_{l})^{N}} \right). 
\end{align*}
On the other hand, in the case where some $\omega_{i} \pm \omega_{j} \pm \omega_{k} \pm \omega_{l} \longarrownot\longrightarrow \infty$, Holder's inequality implies
\begin{align*}
	\left| \overline{\mathbb{R}}_{klji} \right| & =
	 \left| 
	 \int_{0}^{\frac{\pi}{2}} e_{j}(x) \frac{e_{k}^{\prime}(x)}{\omega_{k}}  \frac{e_{l}^{\prime}(x)}{\omega_{l}} \left( \int_{x}^{\frac{\pi}{2}} e_{i}(y) \sin(y) \cos(y) d y \right) \tan^2(x) dx 
	\right| \\
	& \leq 
	\left \| 
	e_{j}
	\right\|_{L^{\infty}\left[0,\frac{\pi}{2}\right]}
	 \left \| 
	\int_{\cdot}^{\frac{\pi}{2}} e_{i}(y) \sin(y) \cos(y) d y 
	\right \|_{L^{\infty}\left[0,\frac{\pi}{2}\right]}	
	\left \| 
	\frac{e_{k}^{\prime}}{\omega_{k}}\tan
	\right \|_{L^{2}\left[0,\frac{\pi}{2}\right]}
	\left \| 
	\frac{e_{l}^{\prime}}{\omega_{l}}\tan
	\right \|_{L^{2}\left[0,\frac{\pi}{2}\right]}
		\leq \frac{\omega_{j}}{\omega_{i}}.
\end{align*}
Similarly, for large values of $i,j,k,l$ and in the case where all $\omega_{i} \pm \omega_{j} \pm \omega_{k}\pm \omega_{l}  \longrightarrow \infty$,  
\begin{align*}
	\overline{\mathbb{S}}_{jkli}& :=
\int_{0}^{\frac{\pi}{2}} e_{j}(x) e_{k}(x) e_{l}(x) \left( \int_{x}^{\frac{\pi}{2}} e_{i}(y) \sin(y) \cos(y) d y \right) \tan^2(x) dx \\
& \simeq \frac{1}{\omega_{i}}  \int_{0}^{\frac{\pi}{2}} \cos^2(x) \frac{ \cos(\omega_{i}x)\sin(\omega_{j}x)\sin(\omega_{k}x)\sin(\omega_{l}x)  }{\tan(x)} dx \\
& \simeq \frac{1}{8 \omega_{i}} \Bigg( 
\int_{0}^{\frac{\pi}{2}} \cos^2(x) \frac{ \sin( (\omega_{i}+\omega_{j}-\omega_{k}+\omega_{l} )x)  }{\tan(x)} dx 
-\int_{0}^{\frac{\pi}{2}} \cos^2(x) \frac{ \sin( (\omega_{i}-\omega_{j}-\omega_{k}-\omega_{l} )x)  }{\tan(x)} dx \\
&- \int_{0}^{\frac{\pi}{2}} \cos^2(x) \frac{ \sin( (\omega_{i}+\omega_{j}+\omega_{k}+\omega_{l} )x)  }{\tan(x)} dx
+ \int_{0}^{\frac{\pi}{2}} \cos^2(x) \frac{ \sin( (\omega_{i}+\omega_{j}+\omega_{k}-\omega_{l} )x)  }{\tan(x)} dx \\
&- \int_{0}^{\frac{\pi}{2}} \cos^2(x) \frac{ \sin( (\omega_{i}-\omega_{j}-\omega_{k}+\omega_{l} )x)  }{\tan(x)} dx
+ \int_{0}^{\frac{\pi}{2}} \cos^2(x) \frac{ \sin( (\omega_{i}-\omega_{j}-\omega_{k}-\omega_{l} )x)  }{\tan(x)} dx \\
&+ \int_{0}^{\frac{\pi}{2}} \cos^2(x) \frac{ \sin( (\omega_{i}-\omega_{j}+\omega_{k}+\omega_{l} )x)  }{\tan(x)} dx
-\int_{0}^{\frac{\pi}{2}} \cos^2(x) \frac{ \sin( (\omega_{i}-\omega_{j}+\omega_{k}-\omega_{l} )x)  }{\tan(x)} dx
\Bigg)
\end{align*}
and the rest of the proof coincides with the one above. On the other hand, if some $\omega_{i} \pm \omega_{j} \pm \omega_{k}\pm \omega_{l} \longarrownot\longrightarrow \infty$, Holder's inequality implies
\begin{align*}
	\left|\overline{\mathbb{S}}_{jkli}  \right| & =
	 \left| 
	 \int_{0}^{\frac{\pi}{2}} e_{j}(x) e_{k}(x)  e_{l}(x) \left( \int_{x}^{\frac{\pi}{2}} e_{i}(y) \sin(y) \cos(y) d y \right) \tan^2(x) dx 
	\right| \\
	& \leq 
	\left \| 
	e_{j}
	\right\|_{L^{\infty}\left[0,\frac{\pi}{2}\right]}
	 \left \| 
	\int_{\cdot}^{\frac{\pi}{2}} e_{i}(y) \sin(y) \cos(y) d y 
	\right \|_{L^{\infty}\left[0,\frac{\pi}{2}\right]}	
	\left \| 
	e_{k}\tan
	\right \|_{L^{2} \left[0,\frac{\pi}{2}\right]}
	\left \| 
	e_{l}\tan
	\right \|_{L^{2}\left[0,\frac{\pi}{2}\right]}
		\leq \frac{\omega_{j}}{\omega_{i}},
\end{align*}
that completes the proof.
\end{proof}	
\appendix 
\section{Auxiliary integral estimates}
In this section, we prove the estimate \eqref{AuxiliaryIntegralsBounds}. To do so, we will use the notation
\begin{align*}
    \mathbbm{1}\left( \text{condition} \right)=
    \begin{cases}
    	1,\quad \text{if the condition if satisfied} \\
    	0, \quad \text{otherwise}
    \end{cases} 
\end{align*}
\begin{lemma}\label{lemmaAppendix}
For all $i=0,1,2,\dots$, we have
	\begin{align*}
      \int_{0}^{\frac{\pi}{2}} \cos^2 (\omega_{i}x) \tan^2(x)dx = \int_{0}^{\frac{\pi}{2}} \frac{\sin^2 (\omega_{i}x)}{\tan^2(x)} dx  = \frac{\pi}{2} \left(\omega_{i}-\frac{1}{2} \right).
\end{align*}
\end{lemma}
\begin{proof}
	Both results follow from Lemma \ref{Dirichlet}. For the first result, we get
	\begin{align*}
		& \int_{0}^{\frac{\pi}{2}} \cos^2 (\omega_{i}x) \tan^2(x)dx   = \int_{0}^{\frac{\pi}{2}}\left( \frac{\cos (\omega_{i}x)}{\cos(x)} \right)^2 \sin^2(x)dx  = \int_{0}^{\frac{\pi}{2}}\left(1+ 2 \sum_{\mu=1}^{i+1} (-1)^{\mu} \cos(2 \mu x) \right)^2 \sin^2(x)dx  \\
		& = \int_{0}^{\frac{\pi}{2}}\left(1+ 4 \sum_{\mu=1}^{i+1} (-1)^{\mu} \cos(2 \mu x) + 4 \sum_{\mu,\nu=1}^{i+1} (-1)^{\mu+\nu} \cos(2 \mu x)\cos(2 \nu x) \right) \sin^2(x)dx  \\
		& = \int_{0}^{\frac{\pi}{2}}\sin^2(x) dx +
		4 \sum_{\mu=1}^{i+1} (-1)^{\mu} \int_{0}^{\frac{\pi}{2}}\cos(2 \mu x) \sin^2(x) dx 
		+ 4 \sum_{\mu,\nu=1}^{i+1} (-1)^{\mu+\nu}\int_{0}^{\frac{\pi}{2}} \cos(2 \mu x)\cos(2 \nu x) \sin^2(x)dx.
	\end{align*}
Now, we use various trigonometric identities
to compute
\begin{align*}
	& \int_{0}^{\frac{\pi}{2}} \sin^2(x) dx = \frac{\pi}{4}, \\
	& \int_{0}^{\frac{\pi}{2}}\cos(2 \mu x) \sin^2(x) dx  = \frac{1}{2}  \int_{0}^{\frac{\pi}{2}}\cos(2 \mu x) dx -\frac{1}{4}  \int_{0}^{\frac{\pi}{2}}\cos(2 (\mu-1) x) dx - \frac{1}{4}  \int_{0}^{\frac{\pi}{2}}\cos(2 (\mu+1) x) dx  \\
	 &\quad \quad \quad \quad \quad \quad \quad \quad \quad \quad = -  \frac{\pi}{8} \mathbbm{1}\left( \mu=1 \right), \\
	& \int_{0}^{\frac{\pi}{2}}\cos(2 \mu x)\cos(2 \nu x) \sin^2(x) dx = -\frac{1}{8} \int_{0}^{\frac{\pi}{2}} \cos \left( 2(1-\mu -\nu)x\right)dx  + \frac{1}{4} \int_{0}^{\frac{\pi}{2}} \cos \left( 2(\mu -\nu)x\right)dx\\
	&\quad \quad \quad \quad \quad \quad \quad \quad \quad \quad \quad \quad \quad \quad  -\frac{1}{8} \int_{0}^{\frac{\pi}{2}} \cos \left( 2(1+\mu -\nu)x\right) dx   -\frac{1}{8} \int_{0}^{\frac{\pi}{2}} \cos \left( 2(1-\mu +\nu)x\right)dx \\
	&\quad \quad \quad \quad \quad \quad \quad \quad \quad \quad \quad \quad \quad \quad  + \frac{1}{4} \int_{0}^{\frac{\pi}{2}} \cos \left( 2(\mu +\nu)x\right) dx - \frac{1}{8} \int_{0}^{\frac{\pi}{2}} \cos \left( 2(1+\mu +\nu)x\right) dx  \\
	& = - \frac{\pi}{16} \mathbbm{1}\left( 1 -\mu-\nu =0 \right) 
	+ \frac{\pi}{8} \mathbbm{1}\left( \mu-\nu =0 \right)  -\frac{\pi}{16} \mathbbm{1}\left( 1+ \mu -\nu=0 \right)
	-\frac{\pi}{16} \mathbbm{1}\left( 1- \mu +\nu=0 \right). 
\end{align*}
Hence,
\begin{align*}
	\int_{0}^{\frac{\pi}{2}} \cos^2 (\omega_{i}x) \tan^2(x)dx  & = \frac{\pi}{4} + \frac{\pi}{2} 	
	 +\frac{\pi}{4} \sum_{\substack{\mu,\nu=1 \\ 1-\mu-\nu=0 }}^{i+1} 1  + \frac{\pi}{2} \sum_{\substack{\mu,\nu=1 \\ \mu-\nu=0 }}^{i+1} 1  + \frac{\pi}{4} \sum_{\substack{\mu,\nu=1 \\ 1+\mu-\nu=0 }}^{i+1} 1  + \frac{\pi}{4} \sum_{\substack{\mu,\nu=1 \\ 1-\mu+\nu=0 }}^{i+1} 1 \\
	&=  \frac{\pi}{4} + \frac{\pi}{2} 	
	 +\frac{\pi}{4} \cdot 0   + \frac{\pi}{2} \cdot (1+i)  + \frac{\pi}{4} \cdot i + \frac{\pi}{4} \cdot i 	= \frac{\pi}{2} \left(\omega_{i}-\frac{1}{2} \right).
	 \end{align*}
Similarly, for the second result, we get
	\begin{align*}
		& \int_{0}^{\frac{\pi}{2}} \frac{\sin^2 (\omega_{i}x)}{\tan^2(x)} dx =
	 \int_{0}^{\frac{\pi}{2}} \left( \frac{\sin (\omega_{i}x)}{\sin(x)}  \right)^2 \cos^2(x) dx = \int_{0}^{\frac{\pi}{2}}\left(1+ 2 \sum_{\mu=1}^{i+1} \cos(2 \mu x) \right)^2 \cos^2(x)dx  \\
		&\quad \quad\quad\quad\quad\quad\quad  = \int_{0}^{\frac{\pi}{2}}\left(1+ 4 \sum_{\mu=1}^{i+1} \cos(2 \mu x) + 4 \sum_{\mu,\nu=1}^{i+1}  \cos(2 \mu x)\cos(2 \nu x) \right) \cos^2(x)dx  \\
		& = \int_{0}^{\frac{\pi}{2}}\cos^2(x) dx +
		4 \sum_{\mu=1}^{i+1} \int_{0}^{\frac{\pi}{2}}\cos(2 \mu x) \cos^2(x) dx 
		+ 4 \sum_{\mu,\nu=1}^{i+1} \int_{0}^{\frac{\pi}{2}} \cos(2 \mu x)\cos(2 \nu x) \cos^2(x)dx.
	\end{align*}
Now, we use various trigonometric identities
to compute
\begin{align*}
    & \int_{0}^{\frac{\pi}{2}} \cos^2(x) dx = \frac{\pi}{4}, \\
	& \int_{0}^{\frac{\pi}{2}}\cos(2 \mu x) \cos^2(x) dx  = \frac{1}{2}  \int_{0}^{\frac{\pi}{2}}\cos(2 \mu x) dx +\frac{1}{4}  \int_{0}^{\frac{\pi}{2}}\cos(2 (\mu-1) x) dx + \frac{1}{4}  \int_{0}^{\frac{\pi}{2}}\cos(2 (\mu+1) x) dx  \\
	 &\quad \quad \quad \quad \quad \quad \quad \quad \quad \quad =   \frac{\pi}{8} \mathbbm{1}\left( \mu=1 \right), \\
	& \int_{0}^{\frac{\pi}{2}}\cos(2 \mu x)\cos(2 \nu x) \cos^2(x) dx = \frac{1}{8} \int_{0}^{\frac{\pi}{2}} \cos \left( 2(1-\mu -\nu)x\right)dx  + \frac{1}{4} \int_{0}^{\frac{\pi}{2}} \cos \left( 2(\mu -\nu)x\right)dx\\
	&\quad \quad \quad \quad \quad \quad \quad \quad \quad \quad \quad \quad \quad \quad  +\frac{1}{8} \int_{0}^{\frac{\pi}{2}} \cos \left( 2(1+\mu -\nu)x\right) dx   +\frac{1}{8} \int_{0}^{\frac{\pi}{2}} \cos \left( 2(1-\mu +\nu)x\right)dx \\
	&\quad \quad \quad \quad \quad \quad \quad \quad \quad \quad \quad \quad \quad \quad  + \frac{1}{4} \int_{0}^{\frac{\pi}{2}} \cos \left( 2(\mu +\nu)x\right) dx + \frac{1}{8} \int_{0}^{\frac{\pi}{2}} \cos \left( 2(1+\mu +\nu)x\right) dx  \\
	& =  \frac{\pi}{16} \mathbbm{1}\left( 1 -\mu-\nu =0 \right) 
	+ \frac{\pi}{8} \mathbbm{1}\left( \mu-\nu =0 \right)  +\frac{\pi}{16} \mathbbm{1}\left( 1+ \mu -\nu=0 \right)
	+\frac{\pi}{16} \mathbbm{1}\left( 1- \mu +\nu=0 \right). 
\end{align*}
Hence,
\begin{align*}
	\int_{0}^{\frac{\pi}{2}} \frac{\sin^2 (\omega_{i}x)}{\tan^2(x)} dx  & = \frac{\pi}{4} + \frac{\pi}{2} 	
	 +\frac{\pi}{4} \sum_{\substack{\mu,\nu=1 \\ 1-\mu-\nu=0 }}^{i+1} 1  + \frac{\pi}{2} \sum_{\substack{\mu,\nu=1 \\ \mu-\nu=0 }}^{i+1} 1  + \frac{\pi}{4} \sum_{\substack{\mu,\nu=1 \\ 1+\mu-\nu=0 }}^{i+1} 1  + \frac{\pi}{4} \sum_{\substack{\mu,\nu=1 \\ 1-\mu+\nu=0 }}^{i+1} 1 \\
	&=  \frac{\pi}{4} + \frac{\pi}{2} 	
	 +\frac{\pi}{4} \cdot 0   + \frac{\pi}{2} \cdot (1+i)  + \frac{\pi}{4} \cdot i + \frac{\pi}{4} \cdot i\\
	 &	= \frac{\pi}{2} \left(\omega_{i}-\frac{1}{2} \right),
	 \end{align*}	 
	 that completes the proof.
\end{proof} 

\section{Closed formulas for the Fourier constants with one index}
In this section, we list closed formulas for the Fourier constants with one index which are used in subsection \ref{giaAppendix}. 
	\begin{lemma}\label{LemmaAppendixB}	
	For all $j=0,1,2,\dots$, we have
	\begin{align*}
		\mathbb{C}_{130j} & = 
		\begin{cases}
			-\frac{3}{\pi},\quad j=0 \\[10pt]
			\frac{513}{320\pi}\sqrt{3},\quad j=1 \\[10pt]
			-\frac{39}{80\pi}\sqrt{\frac{3}{2}},\quad j=2 \\[10pt]
			-\frac{291}{560\pi}\frac{1}{\sqrt{10}},\quad j=3 \\[10pt]
			\frac{69}{1120\pi}\sqrt{\frac{3}{5}},\quad j=4 \\[10pt]
			\frac{162}{\pi \sqrt{2}} \frac{(-1)^{j}}{\sqrt{(j+1)(j+2)}} \frac{1}{3+2j} \frac{9+12j+4j^2}{(j-1)j(j+1)(j+2)(j+3)(j+4)},\quad j \geq 5, 
		\end{cases}\\[5pt]
		\mathbb{C}_{240j}  &= 
		\begin{cases}
			-\frac{279}{16\pi},\quad j=0 \\[10pt]
			\frac{1683}{320\pi}\sqrt{3},\quad j=1 \\[10pt]
			\frac{33}{40\pi}\sqrt{\frac{3}{2}},\quad j=2 \\[10pt]
			-\frac{99}{140\pi}\frac{1}{\sqrt{10}},\quad j=3 \\[10pt]
			-\frac{9}{280\pi}\sqrt{\frac{3}{5}},\quad j=4 \\[10pt]
			\frac{162}{\pi \sqrt{2}} \frac{(-1)^{j}}{\sqrt{(j+1)(j+2)}} \frac{1}{3+2j} \frac{9+12j+4j^2}{(j-1)j(j+1)(j+2)(j+3)(j+4)},\quad j \geq 5,	
		\end{cases}\\[5pt]
		\mathbb{D}_{130j} & = 
		\begin{cases}
			\frac{505521}{28672\pi^2},\quad j=0 \\[10pt]
			-\frac{981867}{143360\pi^2}\sqrt{3},\quad j=1 \\[10pt]
			-\frac{42183}{10240\pi^2}\sqrt{\frac{3}{2}},\quad j=2 \\[10pt]
			\frac{13783689}{788480\pi^2}\frac{1}{\sqrt{10}},\quad j=3 \\[10pt]
			\frac{1523079}{1576960\pi^2}\sqrt{\frac{3}{5}},\quad j=4 \\[10pt]
			-\frac{35626779}{20500480\pi^2}\sqrt{\frac{3}{7}},\quad j=5 \\[10pt]
			-\frac{5325129}{20500480\pi^2}\frac{1}{\sqrt{7}},\quad j=6 \\[10pt]
			\frac{46545}{585728\pi^2},\quad j=7 \\[10pt]
			\frac{201477}{3727360\pi^2}\frac{1}{\sqrt{5}},\quad j=8 \\[10pt]
			\frac{486\sqrt{2}}{\pi^2 } \frac{(-1)^{j}(3+2j)}{\sqrt{(j+1)(j+2)}} \frac{-17325-148200j+8272j^2 + 33264 j^3 +1386 j^4 -1512 j^5 -84 j^6 +24 j^7 +2 j^8}{(j-5)(j-4)(j-3)(j-2)(j-1)j(j+1)(j+2)(j+3)(j+4)(j+5)(j+6)(j+7)(j+8)},\quad j \geq 9,	
		\end{cases}\\[5pt]
		\mathbb{D}_{240j} & = 
		\begin{cases}
			\frac{1028457}{4096\pi^2},\quad j=0 \\[10pt]
			-\frac{236643}{20480\pi^2}\sqrt{3},\quad j=1 \\[10pt]
			-\frac{883833}{10240\pi^2}\sqrt{\frac{3}{2}},\quad j=2 \\[10pt]
			-\frac{53130249}{788480\pi^2}\frac{1}{\sqrt{10}},\quad j=3 \\[10pt]
			\frac{2723691}{143360\pi^2}\sqrt{\frac{3}{5}},\quad j=4 \\[10pt]
			\frac{266148291}{20500480\pi^2}\sqrt{\frac{3}{7}},\quad j=5 \\[10pt]
			\frac{11204163}{1863680\pi^2}\frac{1}{\sqrt{7}},\quad j=6 \\[10pt]
			\frac{5955813}{20500480\pi^2},\quad j=7 \\[10pt]
			\frac{2926809}{41000960\pi^2}\frac{1}{\sqrt{5}},\quad j=8 \\[10pt]
			\frac{1458\sqrt{2}}{\pi^2 } \frac{(-1)^{j}(3+2j)}{\sqrt{(j+1)(j+2)}} \frac{39375-261600j -29528 j^2 + 33264 j^3 +1386 j^4 -1512 j^5 -84 j^5 - 84j^6 +24 j^7 +2j^8}{(j-5)(j-4)(j-3)(j-2)(j-1)j(j+1)(j+2)(j+3)(j+4)(j+5)(j+6)(j+7)(j+8)},\quad j \geq 9,	
		\end{cases}\\[5pt]
		\mathbb{E}_{14230j} & = 
		\begin{cases}
			\frac{225639}{1792\pi^2},\quad j=0 \\[10pt]
			-\frac{137889}{4480\pi^2}\sqrt{3},\quad j=1 \\[10pt]
			-\frac{13233}{320\pi^2}\sqrt{\frac{3}{2}},\quad j=2 \\[10pt]
			\frac{25083}{448\pi^2}\frac{1}{\sqrt{10}},\quad j=3 \\[10pt]
			\frac{160737}{9856\pi^2}\sqrt{\frac{3}{5}},\quad j=4 \\[10pt]
			-\frac{285177}{197120\pi^2}\sqrt{\frac{3}{7}},\quad j=5 \\[10pt]
			-\frac{9313947}{2562560\pi^2\sqrt{7}},\quad j=6 \\[10pt]
			-\frac{1123959}{2562560\pi^2},\quad j=7 \\[10pt]
			-\frac{129009}{2562560\pi^2\sqrt{5}},\quad j=8 \\[10pt]
			\frac{3888\sqrt{2}}{\pi^2 } \frac{(-1)^{j}(3+2j)}{\sqrt{(j+1)(j+2)}} \frac{-315 + 2892 j + 460 j^2 - 309 j^3 - 29 j^4 + 9 j^5 + j^6}{(j-4)(j-3)(j-2)(j-1)j(j+1)(j+2)(j+3)(j+4)(j+5)(j+6)(j+7)},\quad j \geq 9,	
		\end{cases}\\[5pt]	
		\overline{\mathbb{C}}_{130j}  &= 
		\begin{cases}
			-\frac{105}{16\pi},\quad j=0 \\[10pt]
			\frac{627}{320\pi}\sqrt{3},\quad j=1 \\[10pt]
			-\frac{9}{40\pi}\sqrt{\frac{3}{2}},\quad j=2 \\[10pt]
			-\frac{111}{140\pi}\frac{1}{\sqrt{10}},\quad j=3 \\[10pt]
			\frac{3}{40\pi}\sqrt{\frac{3}{5}},\quad j=4 \\[10pt]
			\frac{3}{8\pi \sqrt{2}} \frac{(-1)^{j}}{\sqrt{(j+1)(j+2)}} \frac{18(56+48j+16j^2)}{(j-1)j(j+1)(j+2)(j+3)(j+4)},\quad j \geq 5,	
		\end{cases}\\[5pt]	
		\overline{\mathbb{C}}_{240j} & = 
		\begin{cases}
			-\frac{27}{\pi},\quad j=0 \\[10pt]
			\frac{1377}{320\pi}\sqrt{3},\quad j=1 \\[10pt]
			\frac{87}{80\pi}\sqrt{\frac{3}{2}},\quad j=2 \\[10pt]
			-\frac{339}{560\pi}\frac{1}{\sqrt{10}},\quad j=3 \\[10pt]
			-\frac{3}{160\pi}\sqrt{\frac{3}{5}},\quad j=4 \\[10pt]
			\frac{3}{8\pi \sqrt{2}} \frac{(-1)^{j}}{\sqrt{(j+1)(j+2)}} \frac{18(56+48j+16j^2)}{(j-1)j(j+1)(j+2)(j+3)(j+4)},\quad j \geq 5,	
		\end{cases}\\[5pt]
		\overline{\mathbb{D}}_{130j} & = 
		\begin{cases}
			\frac{100521}{4096\pi^2},\quad j=0 \\[10pt]
			-\frac{8763}{143360\pi^2}\sqrt{3},\quad j=1 \\[10pt]
			-\frac{677823}{71680\pi^2}\sqrt{\frac{3}{2}},\quad j=2 \\[10pt]
			\frac{7256121}{788480\pi^2}\frac{1}{\sqrt{10}},\quad j=3 \\[10pt]
			\frac{714927}{225280\pi^2}\sqrt{\frac{3}{5}},\quad j=4 \\[10pt]
			-\frac{2521353}{1863680\pi^2}\sqrt{\frac{3}{7}},\quad j=5 \\[10pt]
			-\frac{10458663}{20500480\pi^2}\frac{1}{\sqrt{7}},\quad j=6 \\[10pt]
			\frac{81573}{2928640\pi^2},\quad j=7 \\[10pt]
			\frac{3074073}{41000960\pi^2}\frac{1}{\sqrt{5}},\quad j=8 \\[10pt]
			\frac{162\sqrt{2}}{\pi^2 } \frac{(-1)^{j}}{\sqrt{(j+1)(j+2)}}\cdot \\  \quad\quad \cdot \frac{407925 - 574350 j - 1017146 j^2 - 240720 j^3 + 188030 j^4 + 
 64260 j^5 - 6888 j^6 - 3360 j^7 - 10 j^8 + 60 j^9 + 4 j^{10}}{(j-5)(j-4)(j-3)(j-2)(j-1)j(j+1)(j+2)(j+3)(j+4)(j+5)(j+6)(j+7)(j+8)},\quad j \geq 9,	
		\end{cases}\\[5pt]
		\overline{\mathbb{D}}_{240j} & = 
		\begin{cases}
			\frac{5907249}{28672\pi^2},\quad j=0 \\[10pt]
			\frac{8116251}{143360\pi^2}\sqrt{3},\quad j=1 \\[10pt]
			-\frac{3585921}{71680\pi^2}\sqrt{\frac{3}{2}},\quad j=2 \\[10pt]
			-\frac{89919801}{788480\pi^2}\frac{1}{\sqrt{10}},\quad j=3 \\[10pt]
			\frac{6828951}{1576960\pi^2}\sqrt{\frac{3}{5}},\quad j=4 \\[10pt]
			\frac{154847739}{20500480\pi^2}\sqrt{\frac{3}{7}},\quad j=5 \\[10pt]
			\frac{9032301}{1863680\pi^2}\frac{1}{\sqrt{7}},\quad j=6 \\[10pt]
			\frac{856593}{4100096\pi^2},\quad j=7 \\[10pt]
			\frac{5437431}{41000960\pi^2}\frac{1}{\sqrt{5}},\quad j=8 \\[10pt]
			\frac{486\sqrt{2}}{\pi^2 } \frac{(-1)^{j}}{\sqrt{(j+1)(j+2)}} \cdot \\  \quad\quad \cdot \frac{1107225 - 347550 j - 1281746 j^2 - 467520 j^3 + 150230 j^4 + 
 64260 j^5 - 6888 j^6 - 3360 j^7 - 10 j^8 + 60 j^9 + 4 j^{10}}{(j-5)(j-4)(j-3)(j-2)(j-1)j(j+1)(j+2)(j+3)(j+4)(j+5)(j+6)(j+7)(j+8)}, \quad j \geq 9,	
		\end{cases}\\[5pt]
		\overline{\mathbb{E}}_{14230j} & = 
		\begin{cases}
			\frac{253539}{1792\pi^2},\quad j=0 \\[10pt]
			\frac{68139}{4480\pi^2}\sqrt{3},\quad j=1 \\[10pt]
			-\frac{113397}{2240\pi^2}\sqrt{\frac{3}{2}},\quad j=2 \\[10pt]
			-\frac{4317}{448\pi^2}\frac{1}{\sqrt{10}},\quad j=3 \\[10pt]
			\frac{162663}{9856\pi^2}\sqrt{\frac{3}{5}},\quad j=4 \\[10pt]
			\frac{231249}{197120\pi^2}\sqrt{\frac{3}{7}},\quad j=5 \\[10pt]
			-\frac{4113129}{2562560\pi^2}\frac{1}{\sqrt{7}},\quad j=6 \\[10pt]
			-\frac{1104963}{2562560\pi^2},\quad j=7 \\[10pt]
			\frac{81519}{2562560\pi^2}\frac{1}{\sqrt{5}},\quad j=8 \\[10pt]
			\frac{1296\sqrt{2}}{\pi^2 } \frac{(-1)^{j}}{\sqrt{(j+1)(j+2)}} \frac{ -11655 + 4179 j + 15217 j^2 + 6381 j^3 - 1137 j^4 - 729 j^5 + 
 3 j^6 + 24 j^7 + 2 j^8}{(j-4)(j-3)(j-2)(j-1)j(j+1)(j+2)(j+3)(j+4)(j+5)(j+6)(j+7)},\quad j \geq 9,
		\end{cases}\\[10pt]		
\overline{\mathbb{Q}}_{100j} &:=
\frac{214035333120}{\pi^{5/2}} (-1)^j \sqrt{(1 + j) (2 + j)}
\frac{-3277699425  - 269297823960 j + 293711943816 j^2}{(1 - 2 j)^2 (3 - 2 j)^2 (5 - 2 j)^2 (7 - 2 j)^2  } \cdots\\
& \cdots\frac{ + 77509866720 j^3 - 
 99784020400 j^4 - 14916314880 j^5 + 12003789568 j^6 + 
 1948032000 j^7}{(-13 + 2 j) (-11 + 
   2 j) (-9 + 2 j) (1 + 2 j)^2 (3 + 2 j) (5 + 2 j)^2}\cdots \\
&\cdots\frac{ - 570197760 j^8 - 120207360 j^9 + 6178816 j^10 + 
 2580480 j^{11} + 143360 j^{12}}{(7 + 2 j)^2 (9 + 
   2 j)^2 (11 + 2 j)^2 (13 + 2 j)^2 (15 + 2 j) (17 + 2 j) (19 + 
   2 j) },\quad j \geq 0,	\\[10pt]
\overline{\mathbb{Q}}_{200j} &:=
\frac{-113010655887360}{\pi^{5/2}} (-1)^j \sqrt{(1 + j) (2 + j)} 
\frac{33108075 + 380845380 j - 393913652 j^2  }{(1 - 2 j)^2 (3 - 2 j)^2 (5 - 2 j)^2 (7 - 2 j)^2  }\cdots \\
&\cdots\frac{- 136760640 j^3 + 
 119740640 j^4 + 28080000 j^5 - 11212416 j^6 - 2933760 j^7}{(-13 + 2 j) (-11 + 
   2 j) (-9 + 2 j) (1 + 2 j)^2 (3 + 2 j) (5 + 2 j)^2}\cdots \\
&\cdots\frac{+ 
 239360 j^8 + 107520 j^{9} + 7168 j^{10}}{(7 + 2 j)^2 (9 + 
   2 j)^2 (11 + 2 j)^2 (13 + 2 j)^2 (15 + 2 j) (17 + 2 j) (19 + 2 j)},\quad j \geq 0, \\[10pt]
\overline{\mathbb{P}}_{100j} &:=
\frac{-11890851840}{\pi^{5/2}}(-1)^j \sqrt{(1 + j) (2 + j)}
\frac{35629018425 + 1577669058180 j - 2316779684100 j^2 }{(1 - 2 j)^2 (3 - 2 j)^2 (5 - 2 j)^2 (7 - 2 j)^2  }\cdots  \\
&\cdots \frac{- 6803783904 j^3 + 
 981811322512 j^4 - 66644159040 j^5 - 178082423360 j^6 + 
 948476928 j^7 }{(-13 + 2 j) (-11 + 
   2 j) (-9 + 2 j) (1 + 2 j)^2 (3 + 2 j) (5 + 2 j)^2}\cdots \\
&\cdots \frac{+ 15882539776 j^8 + 1123875840 j^9 - 626191360 j^{10} - 
 87220224 j^{11} + 6336512 j^{12} }{(7 + 2 j)^2 (9 + 
   2 j)^2 (11 + 2 j)^2 (13 + 2 j)^2 (15 + 2 j) }\cdots\\
  &\cdots\frac{+ 1720320 j^{13} + 81920 j^{14}}{(17 + 2 j) (19 + 2 j)},\quad j \geq 0, \\[10pt]
\overline{\mathbb{P}}_{200j} &:=
\frac{642105999360}{\pi^{5/2}}(-1)^j\sqrt{(1 + j) (2 + j)}
\frac{-2350673325 - 17348688120 j + 22705082600 j^2}{(1 - 
      2 j)^2 (3 - 2 j)^2 (5 - 2 j)^2 (7 - 2 j)^2 }\cdots  \\
&\cdots \frac{+ 2565707616 j^3 - 
 8972306928 j^4 - 257391360 j^5 + 1438027520 j^6 + 121433088 j^7 }{(-13 + 2 j) (-11 + 
      2 j) (-9 + 2 j) (1 + 2 j)^2 (3 + 2 j) (5 + 2 j)^2}\cdots \\
&\cdots \frac{97458944 j^8 - 15390720 j^9 + 1812480 j^{10} + 516096 j^{11} + 28672 j^{12}}{(7 + 
      2 j)^2 (9 + 2 j)^2 (11 + 2 j)^2 (13 + 2 j)^2 (15 + 2 j) (17 + 
      2 j) (19 + 2 j)},\quad j \geq 0, \\[10pt] 
\mathbb{Q}_{300j} &:=
\frac{1105920}{\pi^{5/2}}(-1)^j\sqrt{(1 + j) (2 + j)}
\frac{273378105 - 157311408 j  }{(-13 + 2 j) (-11 + 
    2 j) (-9 + 2 j) (-7 + 2 j)}\cdots  \\
&\cdots \frac{- 29892784 j^2 + 13889088 j^3 + 1385184 j^4}{(-5 + 2 j) (-3 + 2 j) (-1 + 2 j) (1 + 
    2 j) (5 + 2 j) (7 + 2 j)}\cdots \\
&\cdots \frac{- 
 352512 j^5 - 28416 j^6 + 3072 j^7 + 256 j^8}{(9 + 2 j) (11 + 2 j) (13 + 2 j) (15 + 
    2 j) (17 + 2 j) (19 + 2 j)},\quad j \geq 0, \\[10pt] 
\mathbb{Q}_{400j} &:=
\frac{3317760}{\pi^{5/2}}(-1)^j\sqrt{(1 + j) (2 + j)}
\frac{446350905 - 189244848 j  }{(-13 + 2 j) (-11 + 
    2 j) (-9 + 2 j) (-7 + 2 j)}\cdots  \\
&\cdots \frac{- 40537264 j^2 + 13889088 j^3 + 1385184 j^4 }{(-5 + 2 j) (-3 + 2 j) (-1 + 2 j) (1 + 
    2 j) (5 + 2 j) (7 + 2 j)}\cdots \\
&\cdots \frac{- 
 352512 j^5 - 28416 j^6 + 3072 j^7 + 256 j^8}{(9 + 2 j) (11 + 2 j) (13 + 2 j) (15 + 
    2 j) (17 + 2 j) (19 + 2 j)},\quad j \geq 0, \\[10pt]
\mathbb{P}_{300j} &:=
\frac{-36864}{\pi^{5/2}}(-1)^j \sqrt{(1 + j) (2 + j)}
\frac{-9861476625 + 10067010780 j + 608214484 j^2}{(-13 + 
     2 j) (-11 + 2 j) (-9 + 2 j) (-7 + 2 j) }\cdots  \\
&\cdots \frac{- 1532592960 j^3 - 
 20432800 j^4 + 80991360 j^5 + 2020992 j^6 - 1827840 j^7}{(-5 + 2 j) (-3 + 
     2 j) (-1 + 2 j) (1 + 2 j) (5 + 2 j) (7 + 2 j)}\cdots \\
&\cdots \frac{-  83200 j^8 + 15360 j^9 + 1024 j^{10}}{(9 + 2 j) (11 + 
     2 j) (13 + 2 j) (15 + 2 j) (17 + 2 j) (19 + 2 j) },\quad j \geq 0, \\[10pt] 
\mathbb{P}_{400j} &:=
\frac{-110592}{\pi^{5/2}}(-1)^j\sqrt{(1 + j) (2 + j)}
\frac{-12888500625 + 11300802780 j + 932387284 j^2}{(-13 + 
     2 j) (-11 + 2 j) (-9 + 2 j) (-7 + 2 j) }\cdots  \\
&\cdots \frac{- 1590653760 j^3 - 
 30109600 j^4 + 80991360 j^5 + 2020992 j^6 - 1827840 j^7}{(-5 + 2 j) (-3 + 
     2 j) (-1 + 2 j) (1 + 2 j) (5 + 2 j) (7 + 2 j) }\cdots \\
&\cdots \frac{ - 83200 j^8 + 15360 j^9 + 1024 j^{10}}{(9 + 2 j) (11 + 
     2 j) (13 + 2 j) (15 + 2 j) (17 + 2 j) (19 + 2 j) },\quad j \geq 0.
\end{align*}
\end{lemma}
\begin{proof}
	All these closed formulas are based on Lemma \ref{ClosedformulasFore} and follow similarly. Therefore, we illustrate the proof only for the first constant, namely for $\mathbb{C}_{130j}$. For all $j=0,1,2,\dots$, we have 
	\begin{align}
		\mathbb{C}_{130j} &:= \int_{0}^{\frac{\pi}{2}} 
		\left( \Gamma_{1}(x)-\Gamma_{3}(x) \right) e_{0}(x)e_{j}(x) \tan^{2}(x)dx\nonumber \\ 
		&=\frac{2}{\sqrt{\pi}} \frac{1}{\sqrt{\omega_{j}^2-1}} \int_{0}^{\frac{\pi}{2}} 
		\left( \Gamma_{1}(x)-\Gamma_{3}(x) \right) e_{0}(x) 		 
		\left(
		\omega_{j} \frac{\sin(\omega_{j}x)}{\tan(x)} -
		\cos(\omega_{j}x)
		\right)
		 \tan^{2}(x) dx\nonumber \\ 
		 &= \frac{\omega_{j}}{\sqrt{\omega_{j}^2-1}} \int_{0}^{\frac{\pi}{2}} w_{1}(x) \sin(\omega_{j}x) dx -
		\frac{1}{\sqrt{\omega_{j}^2-1}} \int_{0}^{\frac{\pi}{2}} w_{2}(x) \cos(\omega_{j}x) dx,\label{akuro}
	\end{align}
where
	\begin{align*}
		w_{1}(x)&:=\frac{2}{\sqrt{\pi}} \left( \Gamma_{1}(x)-\Gamma_{3}(x) \right) e_{0}(x) \tan(x) = \frac{3}{4\sqrt{2}\pi^2} \cos^2(x) q(x), \\	
		w_{2}(x) &:=\frac{2}{\sqrt{\pi}} \left( \Gamma_{1}(x)-\Gamma_{3}(x) \right) e_{0}(x) \tan^2(x)= \frac{3}{8\sqrt{2}\pi^2} \sin(2x) q(x)
	\end{align*}
and
\begin{align*}
	q(x):=
		384 ~x \cos^3(x) - 254 \sin(x)-8 \sin(3x)-16 \sin(5x) - 5 \sin(7x)+ \sin(9x).
\end{align*}
We use trigonometric identities to write
\begin{align*}
	w_{1}(x) &=x \left( \frac{90\sqrt{2}}{\pi^2} \cos(x)
	+\frac{45\sqrt{2}}{\pi^2} \cos(3x)
	+\frac{9\sqrt{2}}{\pi^2} \cos(5x) \right)\\
	&-\frac{393}{8\pi^2 \sqrt{2}} \sin(x)
	-\frac{429}{8\pi^2 \sqrt{2}} \sin(3x)
	-\frac{135}{16\pi^2 \sqrt{2}} \sin(5x)
	-\frac{75}{16\pi^2 \sqrt{2}} \sin(7x)\\
	&-\frac{9}{16\pi^2 \sqrt{2}} \sin(9x)
	+\frac{3}{16\pi^2 \sqrt{2}} \sin(11x).
\end{align*}
For $j\geq 5$, we compute each integral separately and find a closed formula for the first integral in \eqref{akuro}. Similarly, for the second. All the other values $\mathbb{C}_{130j}$, $j\in\{0,1,2,3,4\}$ are computed with Mathematica. 
\end{proof}

\bibliographystyle{plain}
\bibliography{FourierConstants_EKG}

\end{document}